\documentclass[11pt, a4paper,oneside, reqno]{amsart}
\usepackage[english]{babel}
\usepackage{amsmath, amsthm, amsfonts, mathrsfs, amssymb, amscd}
\usepackage{bm}
\usepackage{mathtools}
\mathtoolsset{centercolon}
\usepackage{mathabx}

\usepackage[toc]{appendix}

\usepackage[shortlabels]{enumitem}
\usepackage[colorlinks, citecolor = blue, urlcolor={red}]{hyperref}
\usepackage{geometry}
\allowdisplaybreaks

\newtheorem{theorem}{Theorem}
\newtheorem{corollary}[theorem]{Corollary}
\newtheorem{lemma}[theorem]{Lemma}

\newtheorem{proposition}[theorem]{Proposition}

\theoremstyle{remark} 
\newtheorem{remark}[theorem]{Remark}

\theoremstyle{definition} 
\newtheorem{definition}[theorem]{Definition}

\newtheorem{example}[theorem]{Example}

\numberwithin{theorem}{section}
\numberwithin{equation}{section}

\def\N{{\mathbb N}}

\def\R{{\mathbb R}}

\newcommand{\om}{\omega}

\renewcommand{\Re}{\operatorname{Re}}

\newcommand{\Dir}{{\rm Dir}}

\newcommand{\ii}{{\rm i}} 

\newcommand{\mbb}{\mathbb}

\newcommand{\mc}{\mathcal}

\newcommand{\mrm}{\mathrm}

\newcommand{\Hinf}{H^{\infty}}
\newcommand{\RR}{\mathbb{R}}

\newcommand{\CC}{\mathbb{C}}
\newcommand{\NN}{\mathbb{N}}

\newcommand{\OO}{\mathcal{O}}
\newcommand{\WW}{\mathbb{W}}
\renewcommand{\SS}{\mathcal{S}}

\newcommand{\half}{\frac{1}{2}}
\newcommand{\RRdh}{\RR^d_+}
\newcommand{\RRd}{\RR^d}
\newcommand{\Cc}{C_{\mathrm{c}}}
\renewcommand{\d}{\partial}
\newcommand{\del}{\Delta}

\newcommand{\delDir}{\del_{\operatorname{Dir}}}
\newcommand{\delNeu}{\del_{\operatorname{Neu}}}
\newcommand{\odd}{\operatorname{odd}}
\newcommand{\even}{\operatorname{even}}
\newcommand{\ddtt}{\:\frac{\mathrm{d}t}{t}}
\newcommand{\eps}{\varepsilon}
\newcommand{\ph}{\varphi}
\newcommand{\Gam}{\Gamma}
\newcommand{\gam}{\gamma}

\newcommand{\Neu}{{\rm Neu}}
\newcommand{\Tr}{\operatorname{Tr}}

\newcommand{\loc}{{\rm loc}}
\renewcommand{\tilde}[1]{\widetilde{#1}}

\DeclarePairedDelimiter\abs{\lvert}{\rvert}
\DeclarePairedDelimiter\brac[]

\DeclarePairedDelimiter\ha()

\DeclarePairedDelimiter{\nrm}\lVert\rVert

\newcommand{\nrmb}[1]{\bigl\|#1\bigr\|}

\newcommand{\bracb}[1]{\bigl[#1\bigr]}

\newcommand{\has}[1]{\Bigl(#1\Bigr)}

\newcommand{\dd}{\hspace{2pt}\mathrm{d}}

\DeclareMathOperator{\UMD}{UMD}

\begin{document}

\title[Functional calculus for the Laplacian]{Functional calculus on weighted Sobolev spaces for the Laplacian on the half-space}

\author[N. Lindemulder]{Nick Lindemulder}
\address[Nick Lindemulder]{}
\email{nick.lindemulder@gmail.com}

\author[E. Lorist]{Emiel Lorist}
\address[Emiel Lorist]{Delft Institute of Applied Mathematics\\
Delft University of Technology \\ P.O. Box 5031\\ 2600 GA Delft\\The
Netherlands} \email{e.lorist@tudelft.nl}

\author[F.B. Roodenburg]{Floris Roodenburg}
\address[Floris Roodenburg]{Delft Institute of Applied Mathematics\\
Delft University of Technology \\ P.O. Box 5031\\ 2600 GA Delft\\The
Netherlands} \email{f.b.roodenburg@tudelft.nl}

\author[M.C. Veraar]{Mark Veraar}
\address[Mark Veraar]{Delft Institute of Applied Mathematics\\
Delft University of Technology \\ P.O. Box 5031\\ 2600 GA Delft\\The
Netherlands} \email{m.c.veraar@tudelft.nl}

\makeatletter
\@namedef{subjclassname@2020}{
  \textup{2020} Mathematics Subject Classification}
\makeatother


\thanks{The first author is supported by the grant OCENW.KLEIN.358 of the Dutch Research Council (NWO). The third and fourth author are supported by the VICI grant VI.C.212.027 of the NWO}

\begin{abstract}
In this paper, we consider the Laplace operator on the half-space with Dirichlet and Neumann boundary conditions. We prove that this operator admits a bounded $\Hinf$-calculus on Sobolev spaces with power weights measuring the distance to the boundary. These weights do not necessarily belong to the class of Muckenhoupt $A_p$ weights.

We additionally study the corresponding Dirichlet and Neumann heat semigroup. It is shown that these semigroups, in contrast to the $L^p$-case, have polynomial growth. Moreover, maximal regularity results for the heat equation are derived on inhomogeneous and homogeneous weighted Sobolev spaces.
\end{abstract}

\maketitle

\setcounter{tocdepth}{1}
\tableofcontents


\section{Introduction and main results}\label{sec:intro}

Since the development of the $\Hinf$-calculus in \cite{CDMY96, Mc86}, motivated by the Kato square root problem (see \cite{AHLT02, MC90}), this holomorphic functional calculus turned out to provide an efficient way to tackle other problems in mathematics as well. For instance, in the theory of (stochastic) partial differential equations ((S)PDEs), the boundedness of the $\Hinf$-calculus plays an important role. It can be used to obtain well-posedness and regularity of the, possibly nonlinear, (S)PDE. There is a vast literature on the $\Hinf$-calculus for sectorial operators, see, e.g., the monographs and lecture notes \cite{Ar02, DHP03, Ha06, HNVW17, HNVW24, KLW23, KW04, N12, PS16}. For applications to PDEs, see, e.g., \cite{BEH20, DDHPV04, DK13, KKW06, KW17, MM18, To18, We06} and to SPDEs, see, e.g., \cite{AV22,DL11, Ho19, NVW12b, NVW12, NVW15b}).

A particular problem that can occur in the analysis of a (stochastic) PDE on a spatial domain $\OO\subseteq \RR^d$ is that the solution or its derivatives may exhibit blow-up behaviour near the boundary $\d\OO$. To solve the PDE on a Sobolev space $W^{k,p}(\OO)$, it is required to impose additional conditions such as smoothness of the domain and/or unnatural boundary conditions for the data.
\begin{enumerate}[(i)]
 \item \emph{Smoothness of the domain:\,} for regularity of second-order elliptic operators on $W^{k,p}(\OO)$ one needs that $\OO$ is a $C^{k+2}$-domain, see \cite[Chapter 9]{KrBook08}. This condition is also present in regularity theory for parabolic SPDEs, see \cite{Kr94b}.
  \item \emph{Unnatural boundary conditions for the data:} as illustrated in \cite[Chapter 9]{KrBook08}, one obtains regularity for second order elliptic operators on $W^{k,p}(\OO)$, however the operator will not be sectorial. To obtain a sectorial operator, additional boundary conditions on the data need to be specified, see \cite{DD11}.  Furthermore, for higher-order regularity for the heat equation $\d_t u-\del u =f$ on $\OO$ with homogeneous Dirichlet boundary conditions, additional conditions on $f$ need to be imposed as well, see, e.g., \cite[Section 7.1.3]{Ev10}.
\end{enumerate}
To circumvent these additional conditions, weighted spaces for the solution are used with a spatial weight of the form $w_{\gam}(x):=\mrm{dist}(x, \d\OO)^{\gam} $ for some suitable $\gam\in\RR$. In this way, the solution is allowed to have a certain blow-up near the boundary. Weights are also commonly used for SPDEs as is motivated in \cite[Examples 1.1 \& 1.2]{Kr94b}.\\

Motivated by the applications to (S)PDEs mentioned above, it is natural to study the $\Hinf$-calculus for differential operators on \emph{inhomogeneous weighted Sobolev spaces}. In this paper, we take up the study of two commonly used elliptic differential operators: the Laplace operator with Dirichlet and Neumann boundary conditions on the half-space $\RRdh$. We consider $\delDir$ on $W^{k,p}(\RRdh, w_{\gam+kp})$ and $\delNeu$ on $W^{k+1,p}(\RRdh, w_{\gam+kp})$ for $k\geq -1$ and $\gam\in (-1,2p-1)\setminus\{p-1\}$. Our main results include the bounded $\Hinf$-calculus for $-\delDir$ and $-\delNeu$, growth bounds on the corresponding semigroups and maximal regularity results on these weighted Sobolev spaces. The mathematical statements of the main theorems are presented in Section \ref{sec:results}.

The study of the $\Hinf$-calculus for differential operators with Dirichlet and Neumann boundary conditions on Sobolev spaces has been taken up before, see \cite{DDHPV04, KKW06}. The $\Hinf$-calculus on weighted Lebesgue spaces is considered for instance in \cite{AM06, AM07, LRW19, Ma04b}. In these papers, the weights belong to the class of Muckenhoupt $A_p$ weights. In \cite{LV18} the first and fourth authors studied the Dirichlet Laplacian on weighted Lebesgue spaces with weights outside the $A_p$ range. The current paper extends the $\Hinf$-calculus results of \cite{LV18} and we present the first theorems on inhomogeneous weighted Sobolev spaces with weights outside the $A_p$ class.\\

Parabolic and elliptic differential equations on weighted spaces have already been studied extensively in the literature, see, e.g., \cite{DK15, DK18, DKZ16, Ki08, KK04, KN14, KN16, Kr99b, Kr01} for the deterministic setting and \cite{KK18, Ki04_divform, Ki04_varcoef, KL99} for the stochastic setting. In these papers, weighted spaces are used to obtain (stochastic) maximal regularity for elliptic operators on \emph{homogeneous weighted Sobolev spaces}.
The new aspects of our approach, compared to the aforementioned works, are the following.

\begin{enumerate}[(i)]
\item We prove the boundedness of the $\Hinf$-calculus, which gives the boundedness of many singular integral operators. In particular, it yields (stochastic) maximal regularity and bounded imaginary powers.
 \item Using a scaling argument, we show that our results on inhomogeneous weighted Sobolev spaces recover maximal regularity for the heat equation on homogeneous spaces as well. In particular, we recover parts of the results for the heat equation in \cite{DK18, DKZ16, Kr99b, Kr01}.  Moreover, we can allow for weights in time.
  \item We can treat power weights with exponents $\gam\in (-1,2p-1)\setminus\{p-1\}$, extending the typical range $\gam\in(p-1,2p-1)$ considered for higher-order regularity in the homogeneous setting with Dirichlet boundary conditions.
  \item The main results are presented for both the Dirichlet and the Neumann Laplacian. The weighted higher-order regularity for the Neumann Laplacian is new. Weighted second order regularity for elliptic and parabolic equations with Neumann boundary conditions is considered in \cite{DK18, DKZ16, KN16}.
\end{enumerate}

In contrast to our study of the Laplacian on $\RRdh$ on inhomogeneous weighted Sobolev spaces, more general elliptic differential operators on domains are considered in the aforementioned works on homogeneous weighted Sobolev spaces. Moreover, negative and fractional smoothness parameters were studied.
In our case, we can also consider the following generalisations.
\begin{enumerate}[(i)]
  \item  The bounded $\Hinf$-calculus obtained for the Laplacian on $\RRdh$ can be transferred to bounded domains using
      perturbation theorems for the $\Hinf$-calculus. In particular, for the Dirichlet Laplacian, we can use the weights to weaken the required smoothness of the boundary from $C^{k+2}$ to $C^{1,\lambda}$ for $\lambda\in[0,1]$ depending on $\gam$. This will be done in a forthcoming paper.
  \item Our results for, e.g., the Dirichlet Laplacian remain true on the complex interpolation space
\begin{equation*}
  W^{k+\theta,p}(\RRdh, w_{\gam+(k+\theta)p}):=[W^{k,p}(\RRdh, w_{\gam+kp}), W^{k+1,p}(\RRdh, w_{\gam+(k+1)p})]_{\theta},
\end{equation*}
with $\theta\in (0,1)$. To characterise this complex interpolation space one can use the characterisations used in \cite{Kr99b, Kr01, Lo00} for homogeneous weighted Sobolev spaces. Similarly, it is expected that negative smoothness can be obtained via duality.

For negative smoothness we do consider $-\delDir$ on $W^{-1,p}(\RRdh, w_{\gam})$ with $\gam\in(-1,p-1)$. In this case, the $\Hinf$-calculus can be derived from the calculus for $-\delDir$ and $-\delNeu$ on $L^p(\RRdh,w_{\gam})$.
  \item For applications to stochastic PDEs it is well known that a bounded $\Hinf$-calculus yields stochastic maximal regularity in the setting without gradient noise, see \cite{NVW12}.  In an upcoming paper, we will study the setting with gradient noise in which the $\Hinf$-calculus plays a crucial role.
\end{enumerate}

\subsection{Main results}\label{sec:results}

We start with the definition of the inhomogeneous weighted Sobolev spaces on $\RRdh$. Throughout this section we assume that $p\in(1,\infty)$ and $\gam\in(-1,\infty)\setminus\{jp-1:j\in\NN_1\}$. Moreover, for $x\in \RRdh$,  we write $x=(x_1,\tilde{x})\in \RR_+\times \RR^{d-1}$. Let $w_{\gam}(x) := \mrm{dist}(x,\d\RRdh)^{\gam} = |x_1|^{\gam}$, then we define the weighted Lebesgue space $L^p(\RRdh,w_{\gam})$ as the space consisting of all strongly measurable $f\colon \RRdh\to \CC$ such that
\begin{equation*}
\|f\|_{L^p(\RRdh,w_{\gam})} := \Big( \int_{\RRdh}|f(x)|^p\:w_{\gam}(x)\dd x \big)^{1/p}<\infty.
\end{equation*}
The associated $k$-th order weighted Sobolev space for $k \in \N_0$ is defined as
\begin{equation*}
W^{k,p}(\RRdh,w_{\gam}) := \left\{ f \in \mc{D}'(\RRdh) : \forall |\alpha| \leq k, \partial^{\alpha}f \in L^p(\RRdh,w_{\gam}) \right\}.
\end{equation*}
For the definition of the weighted Sobolev space with $k=-1$, we refer to Section \ref{subsec:Dirvar}.

For $p\in(1,\infty)$, $k\in\NN_0$ and $\gam\in(-1,\infty)\setminus\{jp-1:j\in\NN_1\}$, we define the following spaces with Dirichlet and Neumann boundary conditions
  \begin{align*}
  W^{k,p}_{\Dir}(\RR_+^d, w_{\gam})&:=\left\{f\in W^{k,p}(\RR_+^d,w_{\gam}): \operatorname{Tr}(f)=0 \text{ if }k>\tfrac{\gam+1}{p}\right\},\\
  W^{k,p}_{\Neu}(\RR_+^d, w_{\gam})&:=\left\{f\in W^{k,p}(\RR_+^d,w_{\gam}): \operatorname{Tr}(\d_1 f)=0 \text{ if }k-1>\tfrac{\gam+1}{p}\right\},
\end{align*}
where $\Tr$ denotes the trace operator. We will elaborate in Section \ref{sec:weighted_Sob_spaces} on the existence of these traces. Equivalently, in many cases, these spaces with boundary conditions can be defined as the closure of test functions.\\

Let $p\in(1,\infty)$, $k\in\NN_0\cup\{-1\}$ and $\gam\in (-1,2p-1)\setminus\{p-1\}$ be such that $\gam+kp>-1$. The Dirichlet and Neumann Laplacian on $\RRdh$ as we will consider in this paper are defined as follows.
\begin{enumerate}[(i)]
    \item The \emph{Dirichlet Laplacian $\delDir$ on $W^{k,p}(\RRdh,w_{\gam+kp})$} is defined by
  \begin{equation*}
    \delDir u := \sum_{j=1}^d \d_j^2 u\quad \text{ with }\quad D(\delDir):=W^{k+2,p}_{\Dir}(\RRdh, w_{\gam+kp}).
  \end{equation*}
    \item  The \emph{Neumann Laplacian $\delNeu$ on $W^{k+1,p}(\RRdh,w_{\gam+kp})$} is defined by
  \begin{equation*}
    \delNeu u := \sum_{j=1}^d \d_j^2 u\quad \text{ with }\quad D(\delNeu):=W^{k+3,p}_{\Neu}(\RRdh, w_{\gam+kp}).
  \end{equation*}
  \end{enumerate}

In addition, we study the corresponding Dirichlet and Neumann heat semigroup. Let $G_z^d:\RRd\to \RR$ be the standard heat kernel on $\RRd$, defined by
\begin{equation}\label{eq:HeatKernelRd}
  G^d_z(x):=\frac{1}{(4\pi z)^{d/2}}e^{\frac{-|x|^2}{4z}},\qquad z\in\CC_+.
\end{equation}
For $z\in\CC_+$ and $x,y\in\RRdh$, we define the kernels
\begin{equation*}
  H^{d,\pm}_z(x,y):=G^d_z(x_1-y_1,\tilde{x}-\tilde{y})\pm G^d_z(x_1+y_1,\tilde{x}-\tilde{y}).
\end{equation*}
The Dirichlet and Neumann heat semigroups $T_{\Dir}$ and $T_\Neu$ are defined by
\begin{equation}\label{eq:Tt_intro}
\begin{aligned}
  T_{\Dir}(z)f(x)&\;:=H^{d,-}_{z} * f(x):=\int_{\RRdh}H^{d,-}_{z}(x,y)f(y)\dd y,\\
  T_{\Neu}(z)f(x)&\;:=H^{d,+}_z * f(x):=\int_{\RRdh}H^{d,+}_z(x,y)f(y)\dd y,
  \end{aligned}
\end{equation}
for any $f\in L^p(\RRdh, w_{\gam})$ with $\gam$ such that the formulas are well defined.\\

The main results of this paper are summarised in the following theorems. We start with the main result for the Dirichlet Laplacian (see Theorems \ref{thm:var_setting}, \ref{thm:sect:Dir}, \ref{thm:Dirsemi} and \ref{thm:calc:Dir}).
\begin{theorem}[Main results for the Dirichlet Laplacian]\label{thm:sect_calculus}
  Let $p\in(1,\infty)$, $k\in\NN_0\cup\{-1\}$ and $\gam\in (-1,2p-1)\setminus\{p-1\}$ be such that $\gam+kp>-1$. Then for $\delDir$ on $W^{k,p}(\RRdh,w_{\gam+kp})$ the following assertions hold for all $\lambda>0$:
  \begin{enumerate}[(i)]
     \item \label{it:sect_calculus1} $\lambda-\Delta_{\operatorname{Dir}}$ is sectorial of angle $\omega(\lambda-\Delta_{\operatorname{Dir}})=0$,
     \item \label{it:sect_calculus3} $(T_\Dir(z))_{z\in\Sigma_{\sigma}}$ with $\sigma\in (0,\frac{\pi}{2})$ is an analytic $C_0$-semigroup on $W^{k,p}(\RR^d_+,w_{\gam+kp})$ which is generated by $\delDir$,
     \item \label{it:sect_calculus2} $\lambda-\Delta_{\operatorname{Dir}}$ has a bounded $\Hinf$-calculus of angle $\omega_{H^{\infty}}(\lambda-\Delta_{\operatorname{Dir}})=0$.
   \end{enumerate}
   In addition, the semigroup $T_\Dir(t)$ on $W^{k,p}(\RRdh, w_{\gam+kp})$ satisfies the following growth properties:
   \begin{enumerate}[resume*]
   \item\label{it:sect_calculus5} if $\gam+kp\in(-1,2p-1)$, then $T_\Dir(t)$ is bounded and assertions \ref{it:sect_calculus1} and \ref{it:sect_calculus2} hold for $\lambda=0$ as well,
     \item\label{it:sect_calculus4} if $\gam+kp>2p-1$, then $T_\Dir(t)$ has polynomial growth and for any $\eps>0$ there are constants $c,C>0$ only depending on $p,k,\gam, \eps $ and $d$, such that
   \begin{equation*}
       c\big(1+t^{\frac{\gam+kp-2p+1}{2p}}\big)\leq\|T_\Dir(t)\|\leq C \big(1+t^{\frac{\gam+kp-2p+1+\eps}{2p}}\big), \qquad t\geq 0.
   \end{equation*}
   \end{enumerate}
\end{theorem}
Theorem \ref{thm:sect_calculus} in the special case $k=0$ has already been established by the first and fourth authors in \cite[Theorem 4.1 \& 5.7]{LV18}. The case $k=0$ is used as the basis for an induction argument to obtain Theorem \ref{thm:sect_calculus} for general $k\in\NN_0$. For this induction argument, we directly use the estimate from the definition of a bounded $\Hinf$-calculus together with perturbation arguments and commutator estimates. The fact that $\d_1$ and $\delDir$ do not commute complicates the analysis. \\

Concerning the Neumann Laplacian we obtain the following results (see Theorems \ref{thm:sect:Neu}, \ref{thm:Neusemi} and \ref{thm:calc:Neu}).
\begin{theorem}[Main results for the Neumann Laplacian]\label{thm:sect_calculus_neu}
  Let $p\in(1,\infty)$, $k\in\NN_0\cup\{-1\}$ and $\gam\in (-1,2p-1)\setminus\{p-1\}$ be such that $\gam+kp>-1$. Then for $\delNeu$ on $W^{k+1,p}(\RRdh,w_{\gam+kp})$ the following assertions hold for all $\lambda>0$:
  \begin{enumerate}[(i)]
     \item \label{it:sect_calculus1_neu} $\lambda-\Delta_{\operatorname{Neu}}$ is sectorial of angle $\omega(\lambda-\Delta_{\operatorname{Neu}})=0$,
     \item \label{it:sect_calculus3_neu} $(T_\Neu(z))_{z\in\Sigma_{\sigma}}$ with $\sigma\in (0,\frac{\pi}{2})$ is an analytic $C_0$-semigroup on $W^{k+1,p}(\RR^d_+,w_{\gam+kp})$ which is generated by $\delNeu$,
     \item \label{it:sect_calculus2_neu} $\lambda-\Delta_{\operatorname{Neu}}$ has a bounded $\Hinf$-calculus of angle $\omega_{H^{\infty}}(\lambda-\Delta_{\operatorname{Neu}})=0$.
   \end{enumerate}
   In addition, the semigroup $T_\Neu(t)$ on $W^{k+1,p}(\RRdh, w_{\gam+kp})$ satisfies the following growth properties:
   \begin{enumerate}[resume*]
   \item\label{it:sect_calculus5_neu} if $\gam+kp\in(-1,p-1)$, then $T_\Neu(t)$ is bounded and assertions \ref{it:sect_calculus1_neu} and \ref{it:sect_calculus2_neu} hold for $\lambda=0$ as well,
     \item\label{it:sect_calculus4_neu} if $\gam+kp>p-1$, then $T_\Neu(t)$ has polynomial growth and for any $\eps>0$ there are constants $c,C>0$ only depending on $p,k,\gam, \eps$ and $d$, such that
   \begin{equation*}
       c\big(1+t^{\frac{\gam+kp-p+1}{2p}}\big)\leq\|T_\Neu(t)\|\leq C \big(1+t^{\frac{\gam+kp-p+1+\eps}{2p}}\big), \qquad t\geq 0.
   \end{equation*}
   \end{enumerate}
\end{theorem}
Theorem \ref{thm:sect_calculus_neu} will be derived in a similar way as Theorem \ref{thm:sect_calculus}, but additionally, we can also use the results from Theorem \ref{thm:sect_calculus}. This is illustrated by the fact that $\d_1 R(\lambda,\delNeu)f=R(\lambda, \delDir)\d_1 f$ for appropriate $f$.

\begin{remark}\label{rem:Bc}  \hspace{2em}
  \begin{enumerate}[(i)]
  \item Theorems \ref{thm:sect_calculus} and \ref{thm:sect_calculus_neu} also hold for vector-valued weighted Sobolev spaces, which is the setting used throughout this paper.
    \item \label{it:rem:Bc} Theorems \ref{thm:sect_calculus} and \ref{thm:sect_calculus_neu} are false for $\gam\notin(-1,2p-1)$. If $\gam\geq 2p-1$, then the Laplacian is not sectorial (see Example \ref{ex:no_sect}). If $\gam\leq -1$, then the Laplacian will only be sectorial if additional boundary conditions are included in the domain of the operator, as is illustrated in \cite[Section 5.5]{LV18}.
\item By complex interpolation in the parameter $\gam$, we can obtain the statements in Theorem \ref{thm:sect_calculus} for $\gam=p-1$, although, the domain will need to be changed.
    The existing interpolation theory for weighted spaces with boundary conditions is inadequate here.
  \item  The growth of the Dirichlet and Neumann heat semigroup is a result of the behaviour of the weight away from the boundary $x_1=0$. Replacing the weight $w_{\gam+kp}$ by a weight which behaves like $w_{\gam+kp}$ near zero and becomes constant at infinity, or, considering a bounded domain, would result in a bounded semigroup.
      We refer to  \cite{Kr99c} and \cite{MT23} for related results on the Dirichlet heat semigroup on weighted spaces.
    \item It is an open question if the condition $\eps>0$ in the upper bound for the semigroups is optimal or that actually $\eps=0$ holds.
  \end{enumerate}
\end{remark}

As a consequence of Theorems \ref{thm:sect_calculus} and \ref{thm:sect_calculus_neu}, we obtain maximal regularity for the Laplace and heat equation. These results will be presented in Section \ref{sec:MR}.

\subsection{Outline} The outline of this paper is as follows. After introducing some preliminary results in Section \ref{sec:prelim}, we study Sobolev spaces with power weights and their properties in Section \ref{sec:weighted_Sob_spaces}. In Section \ref{sec:easy_cases} we collect some known and straightforward results concerning the $\Hinf$-calculus for the Laplacian on lower-order weighted Sobolev spaces. Then, in Section \ref{sec:sect_semigroup} we study the sectoriality of the Laplacian on higher-order weighted Sobolev spaces and in Section \ref{sec:semigroup} we prove growth estimates for the corresponding heat semigroups. Section \ref{sec:calculus} deals with the $\Hinf$-calculus of the Dirichlet and Neumann Laplacian. Finally, in Section \ref{sec:MR} we derive elliptic and parabolic maximal regularity results on inhomogeneous and homogeneous spaces.

\section{Preliminaries}\label{sec:prelim}

\subsection{Notation}\label{subsec:notation}
We denote by $\NN_0$ and $\NN_1$ the set of natural numbers starting at $0$ and $1$, respectively. For $a\in\RR$ we use the notation $(a)_+=a$ if $a\geq 0$ and $(a)_+=0$ otherwise.

For $d\in\NN_1$ the half-space is given by $\RRdh=\RR_+\times\RR^{d-1}$, where $\RR_+=(0,\infty)$ and for $x\in \RRdh$ we write $x=(x_1,\tilde{x})$ with $x_1\in \RR_+$ and $\tilde{x}\in \RR^{d-1}$. For $\gam\in\RR$ and $x\in\RRdh$ we define the power weight $w_{\gam}(x):=\mrm{dist}(x,\d\RRdh)^\gam=|x_1|^\gam$.

For two topological vector spaces $X$ and $Y$, the space of continuous linear operators is $\mc{L}(X,Y)$ and $\mc{L}(X):=\mc{L}(X,X)$. Unless specified otherwise, $X$ will always denote a Banach space with norm $\|\cdot\|_X$ and the dual space is $X':=\mc{L}(X,\CC)$.

For a linear operator $A:X\supseteq D(A)\to X$ on a Banach space $X$ we denote by $\sigma(A)$ and $\rho(A)$  the spectrum and resolvent set, respectively. For $\lambda\in\rho(A)$, the resolvent operator is given by $R(\lambda,A)=(\lambda-A)^{-1}\in \mc{L}(X)$.

We write $f\lesssim g$ (resp. $f\gtrsim g$) if there exists a constant $C>0$, possibly depending on parameters which will be clear from the context or will be specified in the text, such that $f\leq Cg$ (resp. $f\geq Cg$). Furthermore, $f\eqsim g$ means $f\lesssim g$ and $g\lesssim  f$.\\

For an open and non-empty $\OO\subseteq \RR^d$ and $k\in\NN_0\cup\{\infty\}$, the space $C^k(\OO;X)$ denotes the space of $k$-times continuously differentiable functions from $\OO$ to some Banach space $X$. In the case $k=0$ we write $C(\OO;X)$ for $C^0(\OO;X)$. 

Let $\Cc^{\infty}(\OO;X)$ be the space of compactly supported smooth functions on $\OO$ equipped with its usual inductive limit topology. The space of $X$-valued distributions is given by $\mc{D}'(\OO;X):=\mc{L}(\Cc^{\infty}(\OO);X)$. Moreover, $\Cc^{\infty}(\overline{\OO};X)$ is the space of smooth functions with its support in a compact set contained in $\overline{\OO}$.

 We denote the Schwartz space by $\SS(\RRd;X)$ and $\SS'(\RRd;X):=\mc{L}(\SS(\RRd);X)$ is the space of $X$-valued tempered distributions. For $f\in \SS'(\RRd;X)$ its Fourier transform is denoted by $\mc{F} f= \widehat{f}$ and its inverse as $\mc{F}^{-1}f$. For $\OO\subseteq \RR^d$ we define $\SS(\OO;X):=\{u|_{\OO}: u\in\SS(\RRd;X)\}$.

Finally, for $\theta\in(0,1)$ and a compatible couple $(X,Y)$ of Banach spaces, the complex interpolation space is denoted by $[X,Y]_{\theta}$.

\subsection{Sectorial operators}\label{subsec:sect}
For $\om\in(0,\pi)$, let $\Sigma_{\om}=\{z\in\CC\setminus\{0\}:|\operatorname{arg}(z)|<\om\}$ be the sector in the complex plane.
\begin{definition}\label{def:sect}
  An injective, closed linear operator $(A,D(A))$ with dense domain and dense range on a Banach space $X$ is called \emph{sectorial} if there exists a $\om\in (0,\pi)$ such that $\sigma(A)\subseteq\overline{\Sigma_{\om}}$ and
  \begin{equation}\label{eq:sect_est}
    \sup_{\lambda\in \CC\setminus \overline{\Sigma_{\om}}}\|\lambda R(\lambda,A)\|<\infty.
  \end{equation}
  Furthermore, the angle of sectoriality $\om(A)$ is defined as the infimum over all possible $\om$.
\end{definition}
\begin{remark}\label{rem:Inj_XUMD} Let $A$ be a linear operator on a Banach space $X$ satisfying \eqref{eq:sect_est}, then $A$ is closed. Moreover, if $X$ is reflexive, then $A$ has dense domain. Furthermore, if $A$ has dense range, then $A$ is injective. See \cite[Propositions 10.1.7(3) and 10.1.9]{HNVW17}.
\end{remark}
We state a lemma that relates the growth of the semigroup $e^{-zA}$ to the estimate on the resolvent of $A$. For related results of this type, we refer to \cite{Roz2017}.

\begin{lemma}\label{lem:growth_semigroup_abstract}
  Let $A$ be a linear operator on a Banach space $X$ and let $\alpha\geq 1$. The following are equivalent.
  \begin{enumerate}[(i)]
    \item \label{it:lem:growth1} There exist $\om\in(0,\frac{\pi}{2})$ and $C_1>0$ such that $\sigma(A)\subseteq \overline{\Sigma_{\om}}$ and
         \begin{equation*}
            \|(\lambda + A)^{-1}\|_{\mc{L}(X)}\leq C_1 \Big(\frac{1}{|\lambda|^{\alpha}}+ \frac{1}{|\lambda|}\Big),\qquad \lambda\in \Sigma_{\pi-\om}.
         \end{equation*}
    \item\label{it:lem:growth2} There exist $\eta\in(0,\frac{\pi}{2})$ and $C_2>0$ such that $-A$ generates an analytic $C_0$-semigroup on $\Sigma_{\eta}$ and
    \begin{equation*}
        \|e^{-zA}\|_{\mc{L}(X)}\leq C_2(|z|^{\alpha-1}+1),\qquad z\in \Sigma_{\eta}.
    \end{equation*}
  \end{enumerate}
\end{lemma}
\begin{proof}
  We first prove that \ref{it:lem:growth1} implies \ref{it:lem:growth2}. Let $\mu>0$, then there exists a constant $K$, independent of $\lambda$ and $\mu$, such that
  \begin{equation*}
    \|(\lambda +\mu + A)^{-1}\|_{\mc{L}(X)}\leq C_1\Big(\frac{1}{|\lambda+\mu|^{\alpha}}+ \frac{1}{|\lambda+\mu|}\Big)\leq \frac{K}{|\lambda|}\Big(\frac{1}{\mu^{\alpha-1}}+1\Big),
  \end{equation*}
  since $\mu\leq k_0|\lambda+\mu|$ and $|\lambda|\leq k_1|\lambda+\mu|$ with constants $k_0,k_1>0$ only depending on $\om$. By standard theory for analytic semigroups (see \cite[Theorem G.5.2]{HNVW17}) it follows
  \begin{equation*}
    \|e^{-z(\mu+A)}\|_{\mc{L}(X)}\leq K_1 \Big(\frac{1}{\mu^{\alpha-1}}+1\Big),\qquad  z\in \Sigma_{\frac{\pi}{2}-\om}.
  \end{equation*}
  The result follows upon taking $\mu=1/|z|$ because
  \begin{equation*}
    \|e^{-zA}\|_{\mc{L}(X)}=|e^{\frac{z}{|z|}}|\|e^{-z(\frac{1}{|z|}+A)}\|_{\mc{L}(X)}\leq C_2(|z|^{\alpha-1}+1).
  \end{equation*}
  We now prove that \ref{it:lem:growth2} implies \ref{it:lem:growth1}. First, let $\lambda\in \Sigma_{\om}$ for $\om\in(0,\frac{\pi}{2})$. By the Laplace transform we obtain
  \begin{align*}
   \|(\lambda+A)^{-1}\|_{\mc{L}(X)}& =\Big\| \int_0^\infty e^{-\lambda t}e^{-tA}\dd t\Big\|_{\mc{L}(X)}
   \leq C_2 \int_0^\infty e^{-\Re \lambda t}(t^{\alpha-1}+1)\dd t \\&= K \Big(\frac{1}{(\Re \lambda)^\alpha}+\frac{1}{\Re \lambda}\Big)
   \leq K_1 \Big(\frac{1}{|\lambda|^\alpha}+\frac{1}{|\lambda|}\Big),
  \end{align*}
  where the constants $K$ and $K_1$ only depend on $\om$ and $C_2$.
  By considering the rotated operator $e^{\ii\theta}A$ with $|\theta|<\eta$, we obtain \ref{it:lem:growth1}.
\end{proof}

\subsection{The holomorphic functional calculus}\label{subsec:func_cal}
We first introduce the following Hardy spaces. Let $\om\in(0,\pi)$, then $H^1(\Sigma_{\om})$ is the space of all holomorphic functions $f:\Sigma_{\om}\to \CC$ such that
\begin{equation*}
  \|f\|_{H^1(\Sigma_{\om})}:=\sup_{|\nu|<\om}\|t\mapsto f(e^{\ii\nu}t)\|_{L^1(\RR_+,\frac{\mathrm{d}t}{t})}<\infty.
\end{equation*}
Moreover, let $H^{\infty}(\Sigma_{\om})$ be the space of all bounded holomorphic functions on the sector with norm
\begin{equation*}
  \|f\|_{H^{\infty}(\Sigma_{\om})}:=\sup_{z\in\Sigma_{\om}}|f(z)|.
\end{equation*}

\begin{definition}
  Let $A$ be a sectorial operator on a Banach space $X$ and let $\om\in(\om(A),\pi)$, $\nu\in(\om(A), \om)$ and $f\in H^{1}(\Sigma_{\om})$. We define the operator
  \begin{equation*}
    f(A):=\frac{1}{2\pi \ii}\int_{\partial \Sigma_{\nu}}f(z)R(z,A)\dd z,
  \end{equation*}
  where $\d \Sigma_{\nu}$ is traversed downwards. The operator $A$ has a \emph{bounded $\Hinf(\Sigma_{\om})$-calculus} if there exists a $C>0$ such that
  \begin{equation*}
    \|f(A)\|\leq C\|f\|_{\Hinf(\Sigma_{\om})}\quad \text{ for all }f\in H^1(\Sigma_{\om})\cap\Hinf(\Sigma_{\om}).
  \end{equation*}
  Furthermore, the angle of the $\Hinf$-calculus $\om_{\Hinf}(A)$ is defined as the infimum over all possible $\om>\om(A)$.
\end{definition}
For details on the $\Hinf$-calculus we refer to \cite{Ha06} and \cite[Chapter 10]{HNVW17}.

\begin{remark}\label{rem:inv_calc}
  If the sectorial operator $A$ is in addition invertible, then the behaviour of the function $f\in H^{1}(\Sigma_{\om})$ in the neighbourhood of $0$ is immaterial. By Cauchy's theorem, we can equivalently define
  \begin{equation*}
  f(A)=\frac{1}{2\pi \ii}\int_{\Gam_{\nu}}f(z)R(z,A)\dd z,
\end{equation*}
where $\Gam_{\nu}$ is the boundary of $\Sigma_{\nu}\setminus B(0,\delta)$ with $\delta>0$ small enough such that $B(0,\delta)\subseteq \rho(A)$. See \cite[Section 2.5.1]{Ha06} for more details.
\end{remark}

For the main results of this paper as presented in Section \ref{sec:results}, we can also use vector-valued spaces and to this end, we need the $\UMD$ condition.
Recall that a Banach space $X$ satisfies the geometric condition $\UMD$ (unconditional martingale differences) if and only if the Hilbert transform extends to a bounded operator on $L^p(\RR;X)$. We list the following relevant properties of $\UMD$ spaces.
\begin{enumerate}[(i)]
\item Hilbert spaces are $\UMD$ spaces.
  \item If $p\in(1,\infty)$, $(S,\Sigma,\mu)$ is a $\sigma$-finite measure space and $X$ is a $\UMD$ Banach space, then $L^p(S;X)$ is a $\UMD$ Banach space.
  \item $\UMD$ Banach spaces are reflexive.
\end{enumerate}
In particular, $X=\CC$ is a $\UMD$ Banach space. For more details on the $\UMD$ property we refer to \cite[Chapter 4 \& 5]{HNVW16}.\\

We state the following lemma on the $\Hinf$-calculus for the Laplacian on Sobolev and Bessel potential spaces with weights $w\in A_p(\RR^d)$. This result is folklore and follows easily from the $\Hinf$-calculus on $L^p(\RRd,w;X)$ and lifting.

Let $s\in\RR$, then the Bessel potential operator is
$$J_sf=(1-\del)^{\frac{s}{2}}f:=\mc{F}^{-1}\big((1+|\cdot|^2)^{\frac{s}{2}}\mc{F} f\big),\quad f\in \SS'(\RRd;X).$$
For $p\in(1,\infty)$, $s\in\RR$ and $w\in A_p(\RR^d)$ (see Section \ref{sec:weighted_Sob_spaces}), the weighted Bessel potential space $H^{s,p}(\RRd, w;X)\subseteq \SS'(\RRd;X)$ is defined as the space consisting of all $f\in \SS'(\RRd;X)$ such that $J_sf\in L^p(\RRd, w;X)$ and
\begin{equation*}
  \|f\|_{H^{s,p}(\RRd, w;X)}:=\|J_sf\|_{L^p(\RRd, w;X)}<\infty.
\end{equation*}
\begin{lemma}\label{lem:calc_RRd_Wkp}
  Let $p\in(1,\infty)$, $s\in\RR$, $w\in A_p(\RR^d)$ and let $X$ be a $\UMD$ Banach space. Then $-\del$ on $H^{s,p}(\RRd,w;X)$ with domain $H^{s+2,p}(\RRd,w;X)$ has a bounded $\Hinf$-calculus of angle $\om_{\Hinf}(-\del)=0$.
  In particular, if $k\in\NN_0$, then the same statement holds for $-\del$ on $W^{k,p}(\RRd, w;X)$ with domain $W^{k+2,p}(\RRd,w;X)$.
\end{lemma}
\begin{proof}
  Let $\om\in(0,\pi)$. The case $s=0$ follows from \cite[Proposition 3.6(b)]{MV15}, i.e.,  there is a constant $C>0$ such that for $f\in H^1(\Sigma_{\om})\cap \Hinf(\Sigma_{\om})$
  \begin{equation}\label{eq:calcLp}
  \|f(-\del)u\|_{L^p(\RRd,w;X)}\leq C \|f\|_{\Hinf(\Sigma_{\om})}\|u\|_{L^p(\RRd,w;X)},\qquad u\in L^p(\RRd,w;X).
  \end{equation}
  Moreover, the $\Hinf$-calculus of the Laplacian is given by (see \cite[Theorem 10.2.25]{HNVW17})
  $$f(-\del)u=\mc{F}^{-1}\big(f(|\cdot|^2)\mc{F}u\big),\qquad u\in L^p(\RRd,w;X).$$
  Therefore, by \eqref{eq:calcLp} and lifting we obtain for $u\in H^{s,p}(\RRd, w;X)$
  \begin{align*}
    \|f(-\del)u\|_{H^{s,p}(\RRd,w;X)}& =\|J_sf(-\del)u\|_{L^p(\RRd, w;X)}=\|f(-\del)J_su\|_{L^p(\RRd, w;X)} \\
     & \leq C \|f\|_{\Hinf(\Sigma_{\om})}\|J_su\|_{L^p(\RRd, w;X)}= C \|f\|_{\Hinf(\Sigma_{\om})}\|u\|_{H^{s,p}(\RRd, w;X)},
     \end{align*}
     which proves boundedness of the $\Hinf$-calculus on $H^{s,p}(\RRd, w;X)$.
  The last statement about the $\Hinf$-calculus on $W^{k,p}(\RRd,w;X)$ follows from \cite[Proposition 3.2]{MV15}.
\end{proof}
We note that the $\UMD$ condition on $X$ is necessary for the $\Hinf$-calculus on $L^p(\RRd;X)$, see \cite[Section 10.5]{HNVW17}.

\section{Weighted Sobolev spaces}\label{sec:weighted_Sob_spaces}
In this section, we introduce the inhomogeneous weighted Sobolev spaces and derive certain properties that will be used throughout this paper.

\subsection{Inhomogeneous Sobolev spaces with weights}
Let $\OO\in \{\RRd,\RRdh\}$. We call a locally integrable function $w:\OO\to \RR_+$ a \emph{weight}.
For $p \in [1,\infty)$, $w$ a weight and $X$ a Banach space, we define the weighted Lebesgue space $L^p(\OO,w;X)$ as the Bochner space consisting of all strongly measurable $f\colon \OO\to X$ such that
\begin{equation*}
\|f\|_{L^p(\OO,w;X)} := \Big(\int_{\OO}\|f(x)\|^p_X\:w(x)\dd x \Big)^{1/p}<\infty.
\end{equation*}
An important class of weights is the class of \emph{Muckenhoupt $A_p$ weights}. For $p\in(1,\infty)$ and a weight $w:\OO\to \RR_+$ we have $w\in A_p(\OO)$, if
\begin{equation*}
    [w]_{A_p(\OO)}:=\sup_{B}\Big(\frac{1}{|B|}\int_B w(x)\dd x \Big) \Big(\frac{1}{|B|}\int_{B}w(x)^{-\frac{1}{p-1}}\dd x\Big)^{p-1}<\infty,
\end{equation*}
where the supremum is taken over all balls $B\subseteq \OO$. A weight is called even if $w(-x_1,\tilde{x})=w(x_1,\tilde{x})$ for $(x_1,\tilde{x})\in\RRdh$.
We have $w\in A_p(\RRd)$ and $w$ is even if and only if $w\in A_p(\RRdh)$. We refer to \cite[Chapter 7]{Gr14_classical_3rd} for more details on Muckenhoupt weights and their properties.\\

For $\gam \in\RR$ we define the spatial power weight $w_{\gam}$ on $\RRdh$ by
\begin{equation*}
w_{\gam}(x) := \mrm{dist}(x,\d\RRdh)^{\gam} = |x_1|^{\gam}, \qquad x \in \RRdh.
\end{equation*}
For $\OO\in \{\RRd,\RRdh\}$ it holds that $|x_1|^{\gam}$ is in $A_p(\OO)$ if and only if $\gam\in(-1,p-1)$, see \cite[Example 7.1.7]{Gr14_classical_3rd} or \cite{DILTV19}.\\

We now turn to the definition of inhomogeneous weighted Sobolev spaces. Let $p\in(1,\infty)$ and $\OO\in\{\RRd,\RRdh\}$. Let $w$ be a weight such that $w^{-\frac{1}{p-1}}\in L^1_{\loc}(\OO)$. For $k\in\NN_0$ and $X$ a Banach space, we define the
$k$-th order weighted Sobolev space as
\begin{equation*}
W^{k,p}(\OO,w;X) := \left\{ f \in \mc{D}'(\OO;X) : \forall |\alpha| \leq k, \partial^{\alpha}f \in L^p(\OO,w;X) \right\}
\end{equation*}
equipped with the canonical norm.

The local $L^1$ condition for $w^{-\frac{1}{p-1}}$ ensures that all the derivatives $\d^{\alpha}f$ are locally integrable (by H\"older's inequality), so that the weighted Sobolev space is well defined.
The condition $w^{-\frac{1}{p-1}}\in L^1_{\loc}(\RRd)$ holds in particular if $w\in A_p(\RRd)$ or $w=w_{\gam}$ with $\gam\in(-\infty,p-1)$. The condition $w^{-\frac{1}{p-1}}\in L^1_{\loc}(\RRdh)$ holds if $w\in A_p(\RRdh)$ or $w=w_{\gam}$ with $\gam\in\RR$. For $\gam\geq p-1$ one has to be careful with defining the weighted Sobolev spaces on the full space because functions might not be locally integrable near $x_1=0$, but on $\RRdh$ we can allow for any $\gam\in\RR$, see \cite{KO1984}.

Moreover, we recall from \cite[Lemma 3.1]{LV18} that for $p\in(1,\infty)$ and $w$ such that $w^{-\frac{1}{p-1}}\in L^1_{\loc}(\RR_+)$, we have the Sobolev embedding
\begin{equation}\label{eq:embedding_Ap_C}
  W^{1,p}(\RR_+, w; X)\hookrightarrow C([0,\infty);X).
\end{equation}
We will frequently make use of Hardy's inequality, see for instance \cite[Lemma 3.2]{LV18}.
\begin{lemma}[Hardy's inequality on $\RR_+$]\label{lem:Hardy}
  Let $p\in(1,\infty)$ and let $X$ be a Banach space. Let $u\in W^{1,p}(\RR_+,w_{\gam};X)$ and assume either
\begin{enumerate}[(i)]
\item $\gam<p-1$ and $u(0)=0$, or,
\item $\gam>p-1$.
\end{enumerate}Then
  \begin{equation*}
    \|u\|_{L^p(\RR_+,w_{\gam-p};X)}\leq C_{p,\gam}\|u'\|_{L^p(\RR_+,w_{\gam};X)}.
  \end{equation*}
\end{lemma}

Using Hardy's inequality and \eqref{eq:embedding_Ap_C}, we can define weighted Sobolev spaces with zero boundary conditions.
For $p\in(1,\infty)$, $k\in\NN_0$, $\gam\in (-1,\infty)\setminus\{jp-1:j\in\NN_1\}$ and $X$ a Banach space, we define the following spaces with vanishing traces
\begin{equation}\label{eq:SobBC}
    \begin{aligned}
W^{k,p}_{0}(\RR_+^d, w_{\gam}; X)&:=\left\{f\in W^{k,p}(\RR_+^d,w_{\gam};X): \operatorname{Tr}( \d^{\alpha}f)=0 \text{ if }k-|\alpha|>\tfrac{\gam+1}{p}\right\}, \\
  W^{k,p}_{\Dir}(\RR_+^d, w_{\gam}; X)&:=\left\{f\in W^{k,p}(\RR_+^d,w_{\gam};X): \operatorname{Tr}(f)=0 \text{ if }k>\tfrac{\gam+1}{p}\right\},\\
  W^{k,p}_{\Neu}(\RR_+^d, w_{\gam}; X)&:=\left\{f\in W^{k,p}(\RR_+^d,w_{\gam};X): \operatorname{Tr}(\d_1 f)=0 \text{ if }k-1>\tfrac{\gam+1}{p}\right\}.
\end{aligned}
\end{equation}
For $w\in A_p(\RRdh)$ the Sobolev spaces with boundary conditions are
\begin{equation}\label{eq:SobBCAp}
    \begin{aligned}
W^{k,p}_{0}(\RR_+^d, w; X)&:=\left\{f\in W^{k,p}(\RR_+^d,w;X): \operatorname{Tr}( \d^{\alpha}f)=0 \text{ if }|\alpha|\leq k-1\right\}, \\
  W^{k,p}_{\Dir}(\RR_+^d, w; X)&:=\left\{f\in W^{k,p}(\RR_+^d,w;X): \operatorname{Tr}(f)=0 \text{ if }k\geq 1\right\},\\
  W^{k,p}_{\Neu}(\RR_+^d, w; X)&:=\left\{f\in W^{k,p}(\RR_+^d,w;X): \operatorname{Tr}(\d_1 f)=0 \text{ if }k\geq 2\right\}.
\end{aligned}
\end{equation}
We check that all the traces in the above definitions exist. For the traces in \eqref{eq:SobBCAp} this follows from \eqref{eq:embedding_Ap_C}. For \eqref{eq:SobBC}, let $f\in W^{k,p}(\RRdh, w_{\gam};X)$ and $\alpha\in\NN_0^d$ be such that $k-|\alpha|>\tfrac{\gam+1}{p}$. Then for $\gam\in (jp-1,(j+1)p-1)$ with $j\in \NN_0$, note that $k-|\alpha|\geq j+1$ and thus by Hardy's inequality (Lemma \ref{lem:Hardy}) and \eqref{eq:embedding_Ap_C}
\begin{equation*}
\begin{aligned}
   W^{k-|\alpha|,p}(\RRdh,w_{\gam};X)&\hookrightarrow  W^{k-|\alpha|,p}(\RR_+,w_{\gam};L^{p}(\RR^{d-1};X))\\
  &\hookrightarrow W^{j+1,p}(\RR_+, w_{\gam};L^p(\RR^{d-1};X))\\
  &\hookrightarrow W^{1,p}(\RR_+, w_{\gam-jp};L^p(\RR^{d-1};X))
  \hookrightarrow C([0,\infty);L^p(\RR^{d-1};X)).
\end{aligned}
\end{equation*}
We conclude that $\d^{\alpha} f \in C([0,\infty); L^p(\RR^{d-1};X))$ and thus the spaces in \eqref{eq:SobBC} are well defined.

\begin{remark}\label{rem:W=W_0}
An important observation is that for $\gam\in (-1,\infty)\setminus\{jp-1:j\in\NN_1\}$ we have by definition
    \begin{align*}
W^{k,p}_\Dir(\RRdh, w_{\gam};X)&=W^{k,p}_0(\RRdh, w_{\gam};X)=W^{k,p}(\RRdh,w_{\gam};X) && \text{ if } \gam> kp-1,\\
W^{k,p}_\Dir(\RRdh, w_{\gam};X)&=W^{k,p}_0(\RRdh, w_{\gam};X)&&\text{ if } \gam>(k-1)p-1.
    \end{align*}
Although we will not consider weights $w_{\gam}$ with $\gam\leq -1$ we can nonetheless define
\begin{equation*}
  W^{k,p}_\Dir(\RRdh, w_{\gam};X)=W^{k,p}_0(\RRdh, w_{\gam};X)=W^{k,p}(\RRdh,w_{\gam};X),
\end{equation*}
see \cite[Lemma 3.1(2)]{LV18}.
\end{remark}

\subsection{Properties of weighted Sobolev spaces}
We collect certain properties of weighted Sobolev spaces for later reference.
We start with the following weighted Sobolev embeddings, which are a direct consequence of Lemma \ref{lem:Hardy}, see also \cite[Section 8.8]{Ku85}.
\begin{corollary}[Hardy's inequality on $\RRdh$]\label{cor:Sob_embRRdh}
  Let $p\in(1,\infty)$, $k\in\NN_1$, $\gam\in \RR$ and let $X$ be a Banach space. Then
        \begin{align*}
     W_0^{k,p}(\RRdh,w_{\gam};X)&\hookrightarrow W^{k-1,p} (\RRdh,w_{\gam-p};X) &&\text{ if }\gam<p-1,\\
     W^{k,p}(\RRdh,w_{\gam};X)&\hookrightarrow W^{k-1,p} (\RRdh,w_{\gam-p};X) &&\text{ if }\gam>p-1,\\
      W_0^{k,p}(\RRdh,w_{\gam};X)&\hookrightarrow W_0^{k-1,p} (\RRdh,w_{\gam-p};X)&&\text{ if }\gam\notin\{jp-1:j\in\NN_1\}.
  \end{align*}
\end{corollary}
We continue with a density result. To deal with the Neumann boundary condition we do not only need approximation by compactly supported functions, but density of functions of which only certain derivatives have compact support, is required as well. For $j\in\NN_0$ and $X$ a Banach space, define
\begin{equation}\label{eq:setDense}
  C^{\infty}_{{\rm c},j}(\overline{\RRdh};X):=\{f\in \Cc^{\infty}(\overline{\RRdh};X): \d_1^j f\in \Cc^{\infty}(\RRdh;X)\}.
\end{equation}
Note that the condition $\d_1^j f\in \Cc^{\infty}(\RRdh;X)$ implies that $\d^{\alpha} f\in \Cc^{\infty}(\RRdh;X)$ for all $\alpha=(\alpha_1,\tilde{\alpha})\in\NN_0\times \NN_0^{d-1}$ with $\alpha_1\geq j$.

\begin{lemma}\label{lem:density}
  Let $p\in (1,\infty)$, $j,k\in \NN_0$ such that $k\geq j$ and  $\gam>(k-j)p-1$, and let $X$ be a Banach space. Then $C^{\infty}_{{\rm c},j}(\overline{\RRdh};X)$ is dense in $W^{k,p}(\RRdh,w_{\gam};X)$.
  If, in addition, $X$ is reflexive and $k\geq j+1$, then the statement holds for $\gam=(k-j)p-1$ as well.
\end{lemma}
\begin{remark}
In particular, $\Cc^{\infty}(\RRdh;X)=C^{\infty}_{{\rm c},0}(\overline{\RRdh};X)$ is dense in $W^{k,p}(\RRdh,w_{\gam};X)$ if $\gam> kp-1$, or $\gam\geq kp-1$ when $X$ is reflexive.
  The density of $\Cc^{\infty}(\RRdh;X)$ in $W^{k,p}(\RRdh,w_{\gam};X)$ for $ \gam\in \RR\setminus\{jp-1:j\in \NN_1\}$ with $\gam>kp-1$ also follows from \cite[Proposition 3.8]{LV18} and Remark \ref{rem:W=W_0}.

  In general, it holds that $\Cc^{\infty}(\RRdh;X)$ is dense in $W_0^{k,p}(\RRdh,w_{\gam};X)$ for $\gam\in \RR\setminus\{jp-1:j\in\NN_1\}$, see \cite[Proposition 3.8]{LV18}. However, density of $\Cc^{\infty}(\RRdh;X)$ in $W^{k,p}(\RRdh,w_{\gam};X)$ is not true for all $\gam\in \RR$.
\end{remark}
\begin{proof}
Take $\eps>0$ and fix $f \in W^{k,p}(\RRdh,w_{\gamma};X)$. By  \cite[Theorem 7.2 \& Remark 11.12(iii)]{Ku85}, which also holds in the vector-valued case, and a standard cut-off argument, we find a $g\in \Cc^{\infty}(\overline{\RRdh};X)$ with its support in $[0,R]\times [-R,R]^{d-1}$ for some $R>0$ such that
\begin{equation}\label{eq:density1}
  \|f-g\|_{W^{k,p}(\R^d_+,w_{\gamma};X)}<\eps.
\end{equation}
Let $\phi \in C^\infty(\R_+)$ be such that $\phi =0$ on $[0,\tfrac12]$ and $\phi=1$ on $[1,\infty)$ and set $\phi_n(x_1) := \phi(nx_1)$. We construct a sequence $(g_n)_{n\geq1}$ as follows:
\begin{itemize}
  \item If $j=0$, define $g_n(x):=\phi_n(x_1)g(x)$.
  \item If $j\geq 1$, define
  $$
  g_n(x):= \sum_{m=0}^{j-1} g(1,\tilde{x})\frac{(x_1-1)^m}{m!} + \frac{1}{(j-1)!}\int_1^{x_1} (x_1-t)^{j-1}\phi_n(t)\partial^j_1g(t,\tilde{x})\dd t.
  $$
\end{itemize}
Note that by integration by parts $g_n(x) = g(x)$ for all $x \in \R^d_+$ with $x_1\geq\frac{1}{n}$. For $\abs{\alpha}\leq k$ with $\alpha_1\leq j$ and $x \in \R^d_+$ with $x_1<1$ we have
$$
\abs{\partial^\alpha g_n(x)}\leq C \|g\|_{C^{k+j}(\R^d_+;X)}\|\phi\|_{L^\infty(\R_+)}.
$$ Moreover, we have
\begin{equation}\label{eq:prodrule}
  \partial_1^jg_n(x) = \phi_n(x_1)\partial^j_1g(x),\qquad x \in \R^d_+,
\end{equation}
so that in particular $\partial_1^jg_n \in \Cc^\infty({\R^d_+};X)$. Let $K_R:=[-R,R]^{d-1}$. Using the properties of $g_n$ we obtain
\begin{align*}
  \|g_n&-g\|_{W^{k,p}(\R^d_+,w_{\gamma};X)} \\&\leq \sum_{\abs{\alpha}\leq k} \has{\int_{K_R}\int_0^{\frac1n}\|\partial^\alpha g(x)\|_X^p x_1^{\gamma} \dd x_1\dd \tilde{x}}^{\frac1p}  \\&\hspace{0.5cm}+\sum_{\substack{\abs{\alpha}\leq k\\ \alpha_1\leq j}} \has{\int_{K_R}\int_0^{\frac1n}\|\partial^\alpha g_n(x)\|_X^p x_1^{\gamma} \dd x_1\dd \tilde{x}}^{\frac1p}  \\&\hspace{0.5cm}+ C \sum_{\substack{\abs{\alpha}\leq k\\ \alpha_1> j}} \sum_{m=0}^{\alpha_1-j}  \has{\int_{K_R}\int_0^{\frac1n} n^{mp}|\phi^{(m)}(nx_1)|^p\|\partial^{(\alpha_1-m,\tilde{\alpha})}g(x)\|_X^p x_1^{\gamma} \dd x_1\dd \tilde{x}}^{\frac1p}\\
  & \leq C\,n^{\frac{1}{p}((k-j)p-1-\gamma)} \cdot \tfrac{(2R)^{\frac{d-1}{p}}}{\gamma+1} \|g\|_{C^{k+j}(\R_+^d;X)}\|\phi\|_{C^{k-j}(\R_+)},
\end{align*}
 where the sum for $\alpha_1>j$ comes from taking $\alpha_1-j$ derivatives of \eqref{eq:prodrule} and the product rule.
Hence, as $\gamma>(k-j)p-1$, taking $n$ large enough we obtain using \eqref{eq:density1} that
$$
\|f-g_n\|_{W^{k,p}(\R^d_+,w_{\gamma};X)} \leq \|f-g\|_{W^{k,p}(\R^d_+,w_{\gamma};X)}+ \|g-g_n\|_{W^{k,p}(\R^d_+,w_{\gamma};X)} <2\varepsilon.
$$

Now let $\gam=(k-j)p-1$ and $X$ be a reflexive Banach space. Note that the sequence $(g_n)_{n\geq 1}$ is bounded in the reflexive space $W^{k,p}(\RRdh, w_{\gam};X)$. As a corollary of the Banach-Alaoglu theorem, $(g_n)_{n\geq 1}$ has a weakly convergent subsequence, say $g_{n_\ell}\to \tilde{g}$ weakly in $W^{k,p}(\RRdh, w_{\gam};X)$ as $\ell\to\infty$. Since $g_n\to f$ in $\mc{D}'(\RRdh;X)$ as $n\to\infty$ as well, we find $\tilde{g}=f$ by uniqueness of the limit. With the Hahn-Banach separation theorem, this implies that $$f\in \overline{C^{\infty}_{{\rm c},j}(\overline{\RRdh};X)}^{{\rm weak}}=\overline{C^{\infty}_{{\rm c},j}(\overline{\RRdh};X)}^{\|\cdot\|},$$ where the closures are taken in the weak and norm topology of $W^{k,p}(\RRdh, w_{\gam};X)$, respectively.
\end{proof}

Let $\theta\in \RR$ and define the pointwise multiplication operator
\begin{equation}\label{eq:M}
  M^\theta:\Cc^{\infty}(\RR^d_+;X)\to \Cc^{\infty}(\RR^d_+;X) \quad \text{ by }\quad M^\theta u(x)=x_1^\theta \cdot u(x),\qquad x\in\RRdh.
\end{equation}
By duality the operator extends to $M^\theta:\mc{D}'(\RRdh;X)\to \mc{D}'(\RRdh;X)$ and $M^{-\theta}$ acts as inverse for $M^{\theta}$ on $\mc{D}'(\RRdh;X)$. Moreover, we write $M:=M^1$.
We first study the boundedness properties of the multiplication operator $M^j$ with $j\in \NN_1$.

\begin{lemma}\label{lem:LV3.13_ext}
  Let $p\in (1,\infty)$, $j\in\NN_1$, $k\in\NN_0$ and let $X$ be a Banach space.
  \begin{enumerate}[(i)]
    \item \label{it:lem:LV3.13_ext1} If $\gam\in (jp-1,\infty)$, then
    \begin{equation*}
      M^j:W^{k,p}(\RR^d_+,w_{\gam};X)\to W^{k,p}(\RR^d_+,w_{\gam-jp};X)\quad \text{ is bounded.}
    \end{equation*}
    \item \label{it:lem:LV3.13_ext2} If $\gam\in (-1,\infty)\setminus \{\ell p -1:\ell\in\NN_1\}$, then
    \begin{equation*}
      M^{j}:W^{k,p}_0(\RR^d_+,w_{\gam};X)\to W^{k,p}_0(\RR^d_+,w_{\gam-jp};X)\quad \text{ is an isomorphism.}
    \end{equation*}
  \end{enumerate}
\end{lemma}
\begin{proof}
It suffices to consider $\RR_+$ instead of $\RR^d_+$, since derivatives with respect to $x_i$ with $i\in \{2,\dots,d\}$ commute with $M^j$. For \ref{it:lem:LV3.13_ext1} note that by the product rule
  \begin{equation*}
    (M^j u)^{(i)}=\sum_{n=0}^{\min\{i,j\}} c_{n,i,j} M^{j-n}u^{(i-n)},\qquad i\in\{0,\dots k\}.
  \end{equation*}
  Hence, by Hardy's inequality (Lemma \ref{lem:Hardy} using $\gam>jp-1$)
  \begin{align*}
    \|M^j u\|_{W^{k,p}(\RR_+,w_{\gam-jp};X)} &
      \lesssim  \sum_{i=0}^k\sum_{n=0}^{\min\{i,j\}} \|M^{j-n}u^{(i-n)}\|_{L^p(\RR_+,w_{\gam-jp};X)}  \\
   &=  \sum_{i=0}^k\sum_{n=0}^{\min\{i,j\}} \|u^{(i-n)}\|_{L^p(\RR_+,w_{\gam-np};X)}
      \lesssim \|u\|_{W^{k,p}(\RR_+,w_{\gam};X)}.
  \end{align*}
For \ref{it:lem:LV3.13_ext2}, by density of $\Cc^{\infty}(\RR_+;X)$ in $W_0^{k,p}(\RR_+, w_{\gam};X)$ (see \cite[Proposition 3.8]{LV18}), we have that $$M^j:W^{k,p}_0(\RR_+,w_{\gam};X)\to W^{k,p}_0(\RR_+,w_{\gam-jp};X)$$ is bounded. It remains to show that the inverse is also bounded, whereby density it suffices to consider $u\in \Cc^{\infty}(\RR_+;X)$. Again by the product rule
  \begin{equation*}
    (M^{-j}u)^{(i)}=\sum_{n=0}^ic_{n,i,j}M^{-j-n}u^{(i-n)},
  \end{equation*}
   so that by Hardy's inequality
   \begin{equation*}
\|M^{-j} u\|_{W^{k,p}(\RR_+,w_{\gam};X)} \lesssim \sum_{i=0}^k\sum_{n=0}^i \|M^{-j-n}u^{(i-n)}\|_{L^p(\RR_+,w_{\gam};X)}\lesssim \|u\|_{W^{k,p}(\RR_+,w_{\gam-jp};X)}.\qedhere
   \end{equation*}
\end{proof}

Analogous to Lemma \ref{lem:LV3.13_ext} we obtain the following result for $M^\theta$ with $\theta>0$.
\begin{lemma}\label{lem:M_bound_frac}
  Let $p\in (1,\infty)$, $\theta >0$, $k\in\NN_0$ and let $X$ be a Banach space. Moreover, let $\gam,\gam-\theta p\in (-1,\infty)\setminus \{j p -1:j\in\NN_1\}$. Then
    \begin{equation*}
      M^{\theta}:W^{k,p}_0(\RR^d_+,w_{\gam};X)\to W^{k,p}_0(\RR^d_+,w_{\gam-\theta p};X)\quad \text{ is an isomorphism.}
    \end{equation*}
\end{lemma}
\begin{proof}
  As in the proof of Lemma \ref{lem:LV3.13_ext} it suffices to consider $\RR_+$ and $u\in \Cc^{\infty}(\RR_+;X)$. The required estimates follow from the formula
  \begin{equation*}
    (M^{\pm\theta} u)^{(i)} =\sum_{n=0}^ic_{n,i,\theta}M^{\pm\theta-n} u^{(i-n)}, \qquad i\in\{0,\dots, k\},
  \end{equation*}
  and Hardy's inequality (Lemma \ref{lem:Hardy}).
\end{proof}

Using Lemma \ref{lem:LV3.13_ext} we obtain a useful characterisation for induction arguments.
\begin{lemma}\label{lem:norms_induction}
  Let $p\in (1,\infty)$, $k\in\NN_1$, $\gam\in ((k-1)p-1,\infty)\setminus\{jp-1: j\geq k\}$ and let $X$ be a Banach space. Then
    \begin{equation*}
    \|f\|_{W^{k,p}(\RRdh, w_{\gam};X)}\eqsim \sum_{|\alpha|\leq 1}\|M \d^{\alpha}f\|_{W^{k-1,p}(\RRdh,w_{\gam-p};X)},\qquad f\in W^{k,p}(\RRdh,w_{\gam};X),
  \end{equation*}
  where the constant only depends on $p,k,\gam, d$ and $X$.
\end{lemma}
\begin{proof}
From Lemma \ref{lem:LV3.13_ext} we have that
  \begin{equation*}
    M:W_0^{k-1 ,p}(\RRdh,w_{\gam};X)\to W_0^{k-1,p}(\RRdh,w_{\gam-p};X)
  \end{equation*}
  is an isomorphism and hence for $g\in W_0^{k-1,p}(\RRdh,w_{\gam};X)$ it holds
  \begin{equation*}
  \begin{aligned}
    \|Mg\|_{W^{k-1,p}(\RRdh,w_{\gam-p};X)}&\lesssim \|g\|_{W^{k-1,p}(\RRdh,w_{\gam};X)}\quad \text{ and } \\
    \|g\|_{W^{k-1,p}(\RRdh,w_{\gam};X)}& = \|M^{-1}M g\|_{W^{k-1,p}(\RRdh,w_{\gam};X)}\lesssim \|Mg\|_{W^{k-1,p}(\RRdh,w_{\gam-p};X)}.
  \end{aligned}
  \end{equation*}
  Therefore, Remark \ref{rem:W=W_0} gives that for $f\in W^{k,p}(\RRdh, w_{\gam};X)$, we have
  \begin{align*}
    \|f\|_{W^{k,p}(\RRdh, w_{\gam};X)} &= \sum_{|\alpha|\leq 1}\|\d^{\alpha}f\|_{W^{k-1,p}(\RRdh, w_{\gam};X)}
      \eqsim \sum_{|\alpha|\leq 1} \|M \d^{\alpha}f\|_{W^{k-1,p}(\RRdh, w_{\gam-p};X)}.\qedhere
  \end{align*}
\end{proof}

We close this section with a complex interpolation result in the parameter $\gam$, which will only play a role in Section \ref{subsec:growth} to improve the estimate on the growth of the heat semigroup.
Complex interpolation of $L^p$-spaces with a change of measure dates back to Stein and Weiss \cite{St56, SW58}.
Related results on weighted Sobolev spaces can be found in \cite{CE19,KK23, Lo82}.
\begin{proposition}\label{prop:compl_int_gam}
  Let $p \in (1,\infty)$, $k\in\NN_0$ and let $X$ be a Banach space. Let $-1< \gam_0<\gam<\gam_1<\infty$ be such that $\gamma =(1-\theta) \gamma_0+\theta \gamma_1>kp-1$ for some $\theta \in (0,1)$. Then
  \begin{align*}
    \bracb{W^{k,p}(\R^d_+,w_{\gamma_0};X),W^{k,p}(\R^d_+,w_{\gamma_1};X)}_{\theta} = W^{k,p}(\R^d_+,w_\gamma;X).
  \end{align*}
\end{proposition}
\begin{proof}
For $k=0$ the result is a special case of the Stein-Weiss theorem, see, e.g., \cite[Theorem 14.3.1]{HNVW24}. From now on assume that $k\in\NN_1$ and we first prove the inclusion ``$\hookrightarrow$". Let $N:=\sum_{|\alpha|\leq k} 1$ and define the operator
\begin{equation*}
  T:W^{k,p}(\RRdh, w_{\gam_0};X)+W^{k,p}(\RRdh, w_{\gam_1};X)\to L^p(\RRdh,w_{\gam_0};X^N)+L^p(\RRdh,w_{\gam_1};X^N),
\end{equation*}
by $Tf:=(\d^{\alpha}f)_{|\alpha|\leq k}$. It is straightforward to verify that for $j\in\{0,1\}$ the mapping
\begin{equation*}
  T:W^{k,p}(\RRdh, w_{\gam_j};X)\to L^p(\RRdh,w_{\gam_j};X^N)
\end{equation*}
is bounded. Therefore, by properties of the complex interpolation method and the Stein-Weiss theorem, we have that
\begin{equation*}
  T:[W^{k,p}(\RRdh, w_{\gam_0};X),W^{k,p}(\RRdh, w_{\gam_1};X)]_\theta\to L^p(\RRdh,w_{\gam};X^N)
\end{equation*}
is bounded and
\begin{equation*}
  \|f\|_{W^{k,p}(\RRdh,w_{\gam};X)}\lesssim\|Tf\|_{L^p(\RRdh,w_{\gam};X^N)}\lesssim \|f\|_{[W^{k,p}(\RRdh, w_{\gam_0};X),W^{k,p}(\RRdh, w_{\gam_1};X)]_\theta},
\end{equation*}
where the constant only depends on $p,k,\gam,\gam_0,\gam_1$ and $d$.

To prove the other inclusion ``$\hookleftarrow$", note that $ \Cc^\infty(\RRdh;X)$ is dense in $W^{k,p}(\R^d_+,w_{\gamma};X)$ by Lemma \ref{lem:density}. Thus it suffices to show that for all $f \in \Cc^\infty(\RRdh;X)$
\begin{align*}
\|f\|_{\brac{W^{k,p}(\R^d_+,w_{\gamma_0};X),W^{k,p}(\R^d_+,w_{\gamma_1};X)}_{\theta}}\lesssim\|f\|_{ W^{k,p}(\R^d_+,w_\gamma;X)} ,
  \end{align*}
  where the constant may depend on $p,k,\gam, \gam_0,\gam_1$ and $d$. For $j\in\{0,1\}$ we define $\beta_j:=\gam+j(\gam_1-\gam_0)$. Let $ \mathbb{S}:=\{s\in \CC: 0<\Re s< 1\}$ and for $z\in \overline{\mathbb{S}}$ define  $T(z) \colon \Cc^\infty(\R^d_+;X) \to W^{k,p}(\RRdh, w_{\beta_0};X)+W^{k,p}(\RRdh, w_{\beta_1};X)$ by
  \begin{equation*}
    T(z) f(x) := e^{z^2-\theta^2} \cdot x_1^{- \frac{z}{p}\ha{{\gamma_1}-{\gamma_0}} } \cdot f(x),\qquad x \in \R^d_+.
  \end{equation*}
  Then $T(\cdot) f$ is bounded and continuous on $\overline{\mbb{S}}$ and analytic on $\mbb{S}$ for all $f \in \Cc^\infty(\R^d_+;X)$. Let $t \in \R$ and $\alpha=(\alpha_1,\tilde{\alpha})\in \NN_0\times\NN_0^{d-1} $ with $\abs{\alpha} \leq k$. By the product rule and Hardy's inequality (Corollary \ref{cor:Sob_embRRdh} using that $\gam>kp-1$), we have
    \begin{align*}
    \nrmb{\partial^\alpha T(j+\ii t)f}&_{ L^{p}(\R^d_+,w_{\beta_j};X)}
    \leq e \nrm{\partial^{\alpha}f}_{ L^{p}(\R^d_+,w_{\gamma};X)} \\ &+\sum_{n=0}^{\alpha_1-1}c_{n,\alpha_1}e^{1-t^2}\prod_{m=1}^{\alpha_1-n}\Big|\tfrac{1+|t|}{p}|\gam_1-\gam_0|-m+1\Big| \|M^{n-\alpha_1}\d_1^n \d^{\tilde{\alpha}}f\|_{ L^{p}(\R^d_+,w_{\gamma};X)} \\
    \lesssim&\;  \|f\|_{ W^{k,p}(\R^d_+,w_{\gamma};X)}
  \end{align*}
 and therefore for $j\in\{0,1\}$
  \begin{align*}
   \sup_{t \in \R} \, \nrmb{ T(j+\ii t)f}_{ W^{k,p}(\R^d_+,w_{\beta_j};X)} \lesssim  \nrm{ f}_{ W^{k,p}(\R^d_+,w_{\gamma};X)}.
  \end{align*}
   Using that $\Cc^{\infty}(\RRdh;X)$ is dense in $W^{k,p}(\R^d_+,w_{\gamma};X)$, it follows by Stein interpolation \cite[Theorem 2.1]{Vo92} that
  \begin{align*}
   \nrmb{ T(\theta)f}_{\brac{W^{k,p}(\R^d_+,w_{\beta_0};X),W^{k,p}(\R^d_+,w_{\beta_1};X)}_{\theta}} \lesssim \nrm{f}_{ W^{k,p}(\R^d_+,w_{\gamma};X) }.
  \end{align*}
    It remains to show that for $f\in \Cc^{\infty}(\RRdh;X)$ we have
  \begin{equation}\label{eq:compl_int_converse}
    \|f\|_{[W^{k,p}(\R^d_+,w_{\gam_0};X), W^{k,p}(\R^d_+,w_{\gam_1};X)]_\theta}\lesssim \|T(\theta)f\|_{[W^{k,p}(\R^d_+,w_{\beta_0};X), W^{k,p}(\R^d_+,w_{\beta_1};X)]_\theta}.
  \end{equation}
 To this end, let $T^{-1}(\theta):\Cc^{\infty}(\RRdh;X)\to \Cc^{\infty}(\RRdh;X)$ be defined by
 \begin{equation*}
    T^{-1}(\theta) g(x) :=   x_1^{ \frac{\theta}{p}\ha{{\gamma_1}-{\gamma_0}} } \cdot g(x),\qquad x \in \R^d_+.
  \end{equation*}
  Then we claim that for $g\in \Cc^{\infty}(\RRdh;X)$ we have
  \begin{equation*}
    \|T^{-1}(\theta)g\|_{[W^{k,p}(\R^d_+,w_{\gam_0};X), W^{k,p}(\R^d_+,w_{\gam_1};X)]_\theta}\lesssim \|g\|_{[W^{k,p}(\R^d_+,w_{\beta_0};X), W^{k,p}(\R^d_+,w_{\beta_1};X)]_\theta}.
  \end{equation*}
  Applying this to $g:=T(\theta)f\in \Cc^{\infty}(\RRdh;X)$ proves \eqref{eq:compl_int_converse}, which in turn proves the proposition.

  To prove the claim, note that by Lemma \ref{lem:density} (using that $\beta_j>\gam>kp-1$) and properties of the complex interpolation method, it suffices to prove for $g\in \Cc^{\infty}(\RRdh;X)$ and $j\in \{0,1\}$
  \begin{equation*}
    \|T^{-1}(\theta)g\|_{W^{k,p}(\RRdh, w_{\gam_j};X)}\lesssim \|g\|_{W^{k,p}(\RRdh, w_{\beta_j};X)}.
  \end{equation*}
  Note that $\gam_j+\theta(\gam_1-\gam_0)=\beta_j$, so for any $|\alpha|\leq k$ we obtain by Hardy's inequality
  \begin{align*}
    \|\d^{\alpha} T^{-1}(\theta) g \|_{L^p(\RRdh,w_{\gam_j},X)}&\lesssim \Big(\|\d^{\alpha}g\|_{L^p(\RRdh,w_{\beta_j},X)}+\sum_{n=0}^{\alpha_1-1}\|M^{n-\alpha_1}\d_1^n \d^{\tilde{\alpha}}g\|_{L^p(\RRdh,w_{\beta_j},X)}\Big)\\
    &\lesssim \|g\|_{W^{k,p}(\RRdh, w_{\beta_j};X)}.
  \end{align*}
  This proves the claim and finishes the proof.
\end{proof}

\section{The Laplacian on lower-order weighted Sobolev spaces}\label{sec:easy_cases}
We first note some properties of  the Dirichlet and Neumann Laplacian on the half-space in some known and easy cases. That is, we collect the results for the Dirichlet Laplacian on $L^p$ from \cite{LV18} in Section \ref{subsec:strongDir}. Moreover, with a similar reflection technique as in \cite{LV18}, we derive boundedness of the $\Hinf$-calculus for the Neumann Laplacian in the case of $A_p$ weights in Section \ref{subsec:Neu_Ap}. Finally, in Section \ref{subsec:Dirvar} we consider the weak setting for the Dirichlet Laplacian corresponding to $k=-1$ in Theorem \ref{thm:sect_calculus}.\\

Throughout this paper the Dirichlet and Neumann Laplacian will be defined as follows.
\begin{definition}\label{def:delRRdh}
Let $p\in(1,\infty)$, $\gam\in (-1,2p-1)\setminus\{p-1\}$, $k\in\NN_0\cup\{-1\}$ be such that $\gam+kp>-1$ and let $X$ be a Banach space.
The Dirichlet and Neumann Laplacian on $\RRdh$ are defined as follows.
\begin{enumerate}[(i)]
    \item The \emph{Dirichlet Laplacian $\delDir$ on $W^{k,p}(\RRdh,w_{\gam+kp};X)$} is defined by
  \begin{equation*}
    \delDir u := \sum_{j=1}^d \d_j^2 u\quad \text{ with }\quad D(\delDir):=W^{k+2,p}_{\Dir}(\RRdh, w_{\gam+kp};X).
  \end{equation*}
   For $k=-1$ we elaborate on the definition of $W^{-1,p}(\RRdh,w_{\gam};X)$ with $\gam\in(-1,p-1)$ in Section \ref{subsec:Dirvar}.
    \item  The \emph{Neumann Laplacian $\delNeu$ on $W^{k+1,p}(\RRdh,w_{\gam+kp};X)$} is defined by
  \begin{equation*}
    \delNeu u := \sum_{j=1}^d \d_j^2 u\quad \text{ with }\quad D(\delNeu):=W^{k+3,p}_{\Neu}(\RRdh, w_{\gam+kp};X).
  \end{equation*}
  \end{enumerate}
  Moreover, recall that the semigroups $T_{\Dir}$ and $T_\Neu$ are as defined in \eqref{eq:Tt_intro}.
\end{definition}

\subsection{The strong setting for the Dirichlet Laplacian}\label{subsec:strongDir}
The $\Hinf$-calculus for the Dirichlet Laplacian on $L^p(\RRdh, w_{\gam};X)$ with $\gam\in (-1,2p-1)\setminus\{p-1\}$ is already obtained in \cite{LV18}. We will use this theorem (see Theorem \ref{thm:LVresult} below) in Sections \ref{sec:sect_semigroup} and \ref{sec:calculus} as the basis for an induction argument to show that the Dirichlet Laplacian is sectorial and has a bounded $\Hinf$-calculus on $W^{k,p}(\RRdh, w_{\gam+kp};X)$ with $k\geq 1$ as well.
\begin{theorem}[{\cite[Theorems 4.1 \& 5.7]{LV18}}]\label{thm:LVresult}
  Let $p\in (1,\infty)$, $w\in A_p(\RRdh)$ or $w=w_{\gam}$ with $\gamma\in (-1,2p-1)\setminus\{p-1\}$ and let $X$ be a $\UMD$ Banach space. Let $\delDir$ on $L^p(\RRdh, w;X)$ be as in Definition \ref{def:delRRdh} with $k=0$ and let $T_{\Dir}$ on $L^p(\RRdh, w;X)$ be as in \eqref{eq:Tt_intro}.
  Then, the following assertions hold for all $\lambda\geq 0$:
  \begin{enumerate}[(i)]
    \item $\lambda-\delDir$ is a sectorial operator of angle $\om(\lambda-\delDir)=0$,
    \item\label{it:LVresult2} $(T_{\Dir}(z))_{z\in\Sigma_{\sigma}}$ with $\sigma\in(0,\frac{\pi}{2})$  is a bounded analytic $C_0$-semigroup on $L^p(\RR^d_+,w;X)$ which is generated by $\delDir$,
    \item $\lambda-\delDir$ has a bounded $\Hinf$-calculus of angle $\om_{H^{\infty}}(\lambda-\delDir)=0$.
  \end{enumerate}
\end{theorem}

\begin{remark}[The case $\gam=p-1$]\label{rem:p-1}
  We will not consider the special case $\gam=p-1$ since Hardy's inequality fails in this case. Nonetheless, Theorem \ref{thm:LVresult} remains true if one considers $\delDir$ on $L^p(\RRdh,w_{p-1};X)$ with the different domain
  \begin{equation*}
    D(\delDir):=\overline{\big\{f\in \Cc^{\infty}(\overline{\RRdh};X): f|_{\d\RRdh}=0\big\}}^{W^{2,p}(\RRdh,w_{p-1};X)}.
  \end{equation*}
  This can be proved using \cite[Propositions 5.1 \& 5.3]{LV18} which do hold for $\gam=p-1$, and interpolation (Proposition \ref{prop:compl_int_gam}).
\end{remark}

The proof of Theorem \ref{thm:LVresult} for $w\in A_p(\RRdh)$, which includes $w_{\gam}$ with $\gam\in(-1,p-1)$, goes via an odd reflection to $\RRd$ and using the $\Hinf$-calculus of $\del$ on $L^p(\RRd, w;X)$, see \cite[Theorem 4.1]{LV18}. Since the kernel of the Dirichlet heat semigroup has a zero of order one at the boundary, the range for $\gam$ can be extended.

\subsection{The Neumann Laplacian in the \texorpdfstring{$A_p$}{Ap} setting}\label{subsec:Neu_Ap}
Let $p\in(1,\infty)$ and $w\in A_p(\RRdh)$ or, equivalently, $w\in A_p(\RRd)$ and $w$ is even. Using a similar reflection technique as in the proof of Theorem \ref{thm:LVresult}, we prove that the Neumann Laplacian has a bounded $\Hinf$-calculus on $L^p(\RRdh,w;X)$ and $W^{1,p}(\RRdh, w;X)$. To reflect the Neumann boundary condition we need an even extension in the first variable.

For $f\in L^p(\RRdh,w;X)$ and $y=(y_1,\tilde{y})\in \RR\times \RR^{d-1}$ we define the odd and even extensions by
\begin{equation}\label{eq:odd_ext}
  \begin{aligned}
  E_{\odd}f(y)&:=f_{\odd}(y):=\operatorname{sign}(y_1)f(|y_1|,\tilde{y})\\
  E_{\even}f(y)&:=f_{\even}(y):=f(|y_1|,\tilde{y}).
\end{aligned}
\end{equation}
Moreover, for $k\in\NN_0$ let $W^{k,p}_{\odd}(\RRd,w;X)$ and $W^{k,p}_{\even}(\RRd,w;X)$ be the closed subspaces of all odd and even functions in $W^{k,p}(\RRd,w;X)$, respectively.

We have the following lemma on odd and even extensions.
\begin{lemma}\label{lem:even_ext}
  Let $p\in(1,\infty)$, $w\in A_p(\RRd)$ be even and let $X$ be a Banach space. Then
  \begin{align*}
  E_{\odd}&: W^{k,p}_{\Dir}(\RRdh, w;X)\to W^{k,p}_{\odd}(\RRd,w;X), &&k\in\{0,1,2\},\\
    E_{\even}&: W^{k,p}_{\Neu}(\RRdh, w;X)\to W^{k,p}_{\even}(\RRd,w;X), &&k\in\{0,1,2,3\},
  \end{align*}
  are isomorphisms and
  \begin{align*}
   \|u\|_{W^{k,p}(\RRdh,w;X)}&\leq \|E_{\odd}u \|_{W^{k,p}(\RRd,w;X)}\leq 2^{\frac{1}{p}}\|u\|_{W^{k,p}(\RRdh,w;X)},&&k\in\{0,1,2\},\\
    \|u\|_{W^{k,p}(\RRdh,w;X)}&\leq \|E_{\even} u \|_{W^{k,p}(\RRd,w;X)}\leq 2^{\frac{1}{p}}\|u\|_{W^{k,p}(\RRdh,w;X)}&&k\in\{0,1,2,3\}.
  \end{align*}
\end{lemma}
\begin{remark}
  It should be noted that Lemma \ref{lem:even_ext} does not hold for odd extensions and $k\geq 3$, or even extensions and $k\geq 4$ unless additional boundary conditions are specified (cf. Remark \ref{rem:Bc}\ref{it:rem:Bc}).
\end{remark}
\begin{proof} The statement for the odd extensions is proved in \cite[Lemma 4.2]{LV18}, thus it remains to prove the statement for the even extensions.
  It is straightforward to verify the case $k=0$. For $k=1$ the Neumann trace does not exist and therefore, arguing similar as in \cite[Lemma 4.2]{LV18}, we obtain for $u\in W^{1,p}(\RRdh,w;X)$ the formula
  \begin{equation}\label{eq:dalpha}
    \d^{\alpha}u_{\even}(x)=\operatorname{sign}(x_1)^{\alpha_1}(\d^{\alpha}u)(|x_1|,x),
  \end{equation}
  where $\alpha=(\alpha_1,\tilde{\alpha})$ and $|\alpha|\leq 1$. This proves that $u_{\even}\in W^{1,p}(\RRd,w;X)$ and that the required estimates hold. This finishes the case $k=1$.

  Note that $\operatorname{sign}(x_1)$ in \eqref{eq:dalpha} is not differentiable so we cannot differentiate once more. Instead, if $u\in W^{2,p}_{\Neu}(\RRdh,w;X)$, then $\d_1 u\in W^{1,p}_{\Dir}(\RRdh,w;X)$ and we can use that $\d_1u_{\even}=(\d_1 u)_{\odd} $ to obtain
  \begin{align*}
    \|u_{\even}&\|_{W^{2,p}(\RRd,w;X)} \\ &=\|u_{\even}\|_{W^{1,p}(\RRd,w;X)}+\|(\d_1u)_{\odd}\|_{W^{1,p}(\RRd,w;X)} +\sum_{j=2}^d \|\d_j u_{\even}\|_{W^{1,p}(\RRd,w;X)} \\
    & \leq 2^{\frac{1}{p}}\big(\|u\|_{W^{1,p}(\RRdh,w;X)}+\|\d_1u\|_{W^{1,p}(\RRdh,w;X)}+\sum_{j=2}^d \|\d_j u\|_{W^{1,p}(\RRdh,w;X)}\big)\\
    &=  2^{\frac{1}{p}} \|u\|_{W^{2,p}(\RRdh,w;X)},
  \end{align*}
  using the estimates for the odd and even extension with $k=1$. Similarly, for $k=3$, the estimate follows using the estimates for the odd and even extensions with $k=2$.

  Conversely, let $k\in\{2,3\}$ be fixed and let $u_{\even}\in W^{k,p}(\RRd,w;X)$. Then by \cite[Lemma 3.6]{LV18} there exists a sequence $(u_n)_{n\geq 1}\subseteq \Cc^{\infty}(\RRd;X)$ such that $u_n\to u_{\even}$ in $W^{k,p}(\RRd,w;X)$ as $n\to\infty$. In addition, $u_n(-\cdot,\cdot)\to u_{\even}(-\cdot,\cdot)=u_{\even}(\cdot,\cdot)$ in $W^{k,p}(\RRd,w;X)$. This implies that the sequence
  $v_n:=\half(u_n+u_n(-\cdot,\cdot))$ satisfies
  \begin{equation*}
    v_n\in \Cc^{\infty}(\RRd;X),\quad (\d_1v_n)(0,\cdot)=0\quad\text{ and }\quad v_n\to u_{\even}\text{ in }W^{k,p}(\RRd,w;X).
  \end{equation*}
  Continuity of the trace implies that $u=u_{\even}|_{\RRdh}$ satisfies $\Tr (\d_1u)=0 $. This proves that $u\in W^{k,p}_{\Neu}(\RRdh,w;X)$ since the norm estimates are again clear.
  \end{proof}

Using the even extensions we can obtain the $\Hinf$-calculus for the Laplacian on $\RRdh$ with Neumann boundary conditions.
\begin{theorem}\label{thm:LVresult_Neumann}
  Let $p\in (1,\infty)$, $w\in A_p(\RRdh)$ and let $X$ be a $\UMD$ Banach space.
  For $\ell\in\{0,1\}$, define the Neumann Laplacian $\delNeu$ on $W^{\ell,p}(\RRdh, w;X)$ with domain $D(\delNeu)=W^{\ell+2,p}_{\Neu}(\RRdh,w;X)$ and let $T_{\Neu}$ on $W^{\ell,p}(\RRdh, w;X)$ be as in \eqref{eq:Tt_intro}.
   The following assertions hold for all $\lambda\geq 0$:
  \begin{enumerate}[(i)]
    \item $\lambda-\delNeu$ is a sectorial operator of angle $\om(\lambda-\delNeu)=0$,
        \item\label{it:LVresult_Neumann2}   $(T_{\Neu}(z))_{z\in\Sigma_{\sigma}}$ with $\sigma\in(0,\frac{\pi}{2})$  is a bounded analytic $C_0$-semigroup on $W^{\ell,p}(\RR^d_+,w;X)$ which is generated by $\delNeu$,
    \item $\lambda-\delNeu$ has a bounded $\Hinf$-calculus of angle $\om_{H^{\infty}}(\lambda-\delNeu)=0$.
  \end{enumerate}
\end{theorem}
\begin{proof}
  This proof goes via a reflection argument from $\RRdh$ to $\RRd$ and using the $\Hinf$-calculus for $\del$ on $L^p(\RR^d,w;X)$ and $W^{1,p}(\RR^d,w;X)$, see Lemma \ref{lem:calc_RRd_Wkp}. The argument is completely similar as in \cite[Section 4]{LV18} for the Dirichlet case if one considers the even extension from Lemma \ref{lem:even_ext} instead of the odd extension.
\end{proof}

We state two corollaries of Theorems \ref{thm:LVresult} and \ref{thm:LVresult_Neumann} concerning the Dirichlet and Neumann resolvent equation. Similar results on $L^p(\RRdh,w;\CC)$ with $w\in A_p(\RRdh)$ are already contained in \cite{DK18, DKZ16}.
\begin{corollary}[\cite{LV18}, Corollary 4.3 \& 5.8]\label{cor:LVresult_MR}
  Let $p\in(1,\infty)$, $w\in A_p(\RRdh)$ or $w=w_{\gam}$ with $\gam\in (-1,2p-1)\setminus\{p-1\}$ and let $X$ be a $\UMD$ Banach space.
  Then for all $f\in L^p(\RR^d_+,w;X)$ and $\lambda\in\Sigma_{\pi-\om}$ with $\om\in(0,\pi)$ there exists a unique $u\in W^{2,p}_{\Dir}(\RR^d_+,w;X)$ such that $\lambda u-\delDir u=f$. Moreover, this solution satisfies
  \begin{equation*}
    \sum_{|\beta|\leq 2}|\lambda|^{1-\frac{|\beta|}{2}}\|\d^{\beta}u\|_{L^p(\RR^d_+,w;X)}\leq C \|f\|_{L^p(\RR^d_+,w;X)},
  \end{equation*}
  where the constant $C$ only depends on $p, w, \om, d$ and $X$.
\end{corollary}
For the Neumann Laplacian we have the following result.
\begin{corollary}\label{cor:LVresult_MR_neumann}
  Let $p\in(1,\infty)$, $\ell\in\{0,1\}$, $w\in A_p(\RRdh)$ and let $X$ be a $\UMD$ Banach space.
  Then for all $f\in W^{\ell,p}(\RR^d_+,w;X)$ and $\lambda\in\Sigma_{\pi-\om}$ with $\om\in(0,\pi)$ there exists a unique $u\in W^{\ell+2,p}_{\Neu}(\RR^d_+,w;X)$ such that $\lambda u-\delNeu u=f$. Moreover, this solution satisfies
  \begin{equation*}
    \sum_{|\beta|\leq 2}|\lambda|^{1-\frac{|\beta|}{2}}\|\d^{\beta}u\|_{W^{\ell,p}(\RR^d_+,w;X)}\leq C \|f\|_{W^{\ell,p}(\RR^d_+,w;X)},
  \end{equation*}
  where the constant $C$ only depends on $p, \ell, w, \om, d$ and $X$.
\end{corollary}
\begin{proof}
This follows from Theorem \ref{thm:LVresult_Neumann} using the same argument as in \cite[Corollary 4.3]{LV18}.
\end{proof}

\subsection{The weak setting for the Dirichlet Laplacian}\label{subsec:Dirvar}
In this section, we prove Theorem \ref{thm:sect_calculus} with $k=-1$, which corresponds to the weak setting for the Dirichlet Laplacian. This follows from combining the results for the Dirichlet and Neumann Laplacian in the strong setting. \\

The space $W^{-1,p}(\RRdh,w;X)$ with $w\in A_p(\RRdh)$ consists of all $f\in \mc{S}'(\RRdh;X)$ such that
\begin{equation}\label{eq:W-1_2}
  f=f_0+\sum_{j=1}^d\d_jf_j\quad\text{ with }\quad f_0,f_j\in L^p(\RRdh,w;X)\text{ for all }j\in\{1,\dots,d\}
\end{equation}
equipped with the norm
\begin{equation*}
  \|f\|_{W^{-1,p}(\RRdh,w;X)}:=\inf\Big\{\sum_{j=0}^d \|f_j\|_{L^p(\RRdh, w;X)}: \eqref{eq:W-1_2} \text{ holds}\Big\}.
\end{equation*}
As usual, the derivatives are understood in the sense of distributions, i.e., \eqref{eq:W-1_2} reads
\begin{equation*}
  f(\ph) = f_0(\ph) - \sum_{j=1}^d f_j(\d_j\ph),\qquad \ph\in \mc{S}(\RRdh;X).
\end{equation*}
 For $f\in W^{-1,p}(\RRdh,w;X)$ with $w\in A_p(\RRdh)$ we define the Dirichlet heat semigroup as
\begin{equation}\label{eq:defTk=-1}
  (T_\Dir(z)f)(\ph):= f(T_\Dir(z)\ph),\qquad \ph\in \mc{S}(\RRdh;X).
\end{equation}
This definition coincides with the definition for functions $f\in L^p(\RRdh,w;X)$. Moreover, for $f \in W^{-1,p} (\RRdh,w;X)$ and $\ph\in \mc{S}(\RRdh;X)$ it holds that
\begin{equation}\label{eq:Tt_k=-1}
  \begin{aligned}
  (T_{\Dir}(z)f)(\ph) &= (T_{\Dir}(z)f_0)(\ph)+(\d_1 T_{\Neu}(z)f_1)(\ph)+\sum_{j=2}^d (\d_j T_{\Dir}(z)f_j)(\ph)\\
  &=(T_{\Dir}(z)f_0)(\ph)-( T_{\Neu}(z)f_1)(\d_1\ph)-\sum_{j=2}^d ( T_{\Dir}(z)f_j)(\d_j\ph).
\end{aligned}
\end{equation}

Using Corollaries \ref{cor:LVresult_MR} and \ref{cor:LVresult_MR_neumann} we can derive elliptic regularity for the Dirichlet Laplacian on $W^{-1,p}(\RRdh,w;X)$. For a similar result on $\RRd$ without weights, we refer to \cite[Section 4.4]{KrBook08}.
\begin{proposition}
  Let $p\in(1,\infty)$, $w\in A_p(\RRdh)$ and let $X$ be a $\UMD$ Banach space.
  Then for all $f_j\in L^p(\RRdh,w;X)$ with $j\in\{0,\dots, d\}$ and $\lambda\in\Sigma_{\pi-\om}$ with $\om\in(0,\pi)$, there exists a unique $u\in W^{1,p}_{\Dir}(\RR^d_+,w;X)$ such that $\lambda u-\delDir u=f_0+\sum_{j=1}^d\d_jf_j$. Moreover, this solution satisfies
  \begin{equation*}
    \sum_{|\beta|\leq 1}|\lambda|^{\half-\frac{|\beta|}{2}}\|\d^{\beta}u\|_{L^p(\RR^d_+,w;X)}\leq C |\lambda|^{-\half}\|f_0\|_{L^{p}(\RR^d_+,w;X)}+C\sum_{j=1}^d\|f_j\|_{L^{p}(\RR^d_+,w;X)},
  \end{equation*}
  where the constant $C$ only depends on $p, w, \om, d$ and $X$.
\end{proposition}
\begin{proof}
  For the existence and the estimate it suffices to consider the equations
  \begin{equation}\label{eq:v_j_-1}
      \begin{aligned}
    \lambda v_0 -\delDir v_0&= f_0,& \quad  v_0(0,\cdot)&=0, \\
    \lambda v_1-\delNeu v_1 &=f_1,&\quad  (\d_1v_1)(0,\cdot)&=0, \\
    \lambda v_j-\delDir  v_j&=f_j, &\quad  v_j(0,\cdot)&=0\;\text{ for }\;j\in\{2,\dots,d\}.
  \end{aligned}
  \end{equation}
  Corollaries \ref{cor:LVresult_MR} and \ref{cor:LVresult_MR_neumann} yield solvability of the above equations and the estimates
  \begin{equation}\label{eq:W-1_sect1}
    \sum_{|\beta|\leq 1}|\lambda|^{\half-\frac{|\beta|}{2}}\|\d^{\beta}v_0\|_{L^p(\RR^d_+,w;X)}\leq C |\lambda|^{-\half}\|f_0\|_{L^p(\RR^d_+,w;X)}
  \end{equation}
  and
   \begin{equation}\label{eq:W-1_sect2}
    \sum_{1\leq |\beta|\leq 2}|\lambda|^{1-\frac{|\beta|}{2}}\|\d^{\beta}v_j\|_{L^p(\RR^d_+,w;X)}\leq C \|f_j\|_{L^p(\RR^d_+,w;X)},\quad j\in\{1,\dots,d\}.
  \end{equation}
  Then $u:=v_0+\sum_{j=1}^d\d_j v_j \in W^{1,p}_\Dir(\RRdh,w;X)$ satisfies $\lambda u -\delDir u =f_0+\sum_{j=1}^d\d_jf_j$ and
  \begin{align*}
 \sum_{|\beta|\leq 1}|\lambda|^{\half-\frac{|\beta|}{2}}\|\d^{\beta}u\|_{L^p(\RR^d_+,w;X)}&\leq \sum_{|\beta|\leq 1}|\lambda|^{\half-\frac{|\beta|}{2}}\Big(\|\d^{\beta}v_0\|_{L^p(\RR^d_+,w;X)}+\sum_{j=1}^d\|\d^{\beta}\d_jv_j\| _{L^p(\RR^d_+,w;X)}\Big)\\
 &\leq C |\lambda|^{-\half}\|f_0\|_{L^{p}(\RR^d_+,w;X)}+C\sum_{j=1}^d\|f_j\|_{L^{p}(\RR^d_+,w;X)},
  \end{align*}
  using \eqref{eq:W-1_sect1} and \eqref{eq:W-1_sect2}. Finally, since $\lambda u -\delDir u =0$ has a solution $u\in W^{1,p}_\Dir(\RRdh,w;X)$, we obtain by Lemma \ref{lem:even_ext} that $u_{\odd}\in W^{1,p}(\RR^d,w;X)\hookrightarrow \mc{S}'(\RR^d;X)$ and satisfies $\lambda u_{\odd}-\del u_{\odd}=0$ on $\RR^d$. By employing the Fourier transform it follows that $u_{\odd}=0$ and thus $u=0$ as well. This proves the uniqueness and finishes the proof.
\end{proof}

Sectoriality and boundedness of the $\Hinf$-calculus for $-\delDir$ on $W^{-1,p}(\RRdh,w;X)$ can be derived from the same properties of $-\delDir$ and $-\delNeu$ on $L^p(\RRdh,w;X)$.
We close this section with the main result for the Dirichlet Laplacian in the weak setting. 
\begin{theorem}\label{thm:var_setting}
  Let $p\in (1,\infty)$, $w\in A_p(\RRdh)$ and let $X$ be a $\UMD$ Banach space.
  Define the Dirichlet Laplacian $\delDir$ on $W^{-1,p}(\RRdh, w;X)$ with domain $D(\delDir)=W^{1,p}_{\Dir}(\RRdh,w;X)$ and let $T_{\Dir}$ on $W^{-1,p}(\RRdh, w;X)$ be as in \eqref{eq:defTk=-1}.
   The following assertions hold for all $\lambda\geq 0$:
  \begin{enumerate}[(i)]
    \item\label{it:W-1_sect} $\lambda-\delDir$ is a sectorial operator of angle $\om(\lambda-\delDir)=0$,
        \item \label{it:W-1_semi} $(T_{\Dir}(z))_{z\in\Sigma_{\sigma}}$ with $\sigma\in(0,\frac{\pi}{2})$ is a bounded analytic $C_0$-semigroup on $W^{-1,p}(\RR^d_+,w;X)$ which is generated by $\delDir$,
    \item \label{it:W-1_calc}$\lambda-\delDir$ has a bounded $\Hinf$-calculus of angle $\om_{H^{\infty}}(\lambda-\delDir)=0$.
  \end{enumerate}
\end{theorem}
 \begin{proof}
   Let $f\in W^{-1,p}(\RRdh,w;X)$ such that $f=f_0+\sum_{j=1}^d\d_jf_j$ with $f_j\in L^p(\RRdh,w;X)$ for $j\in\{0,\dots,d\}$. Moreover, let $v_j$ be the solutions to the equations in \eqref{eq:v_j_-1}.

   To prove \ref{it:W-1_sect}, note that by Corollaries \ref{cor:LVresult_MR} and \ref{cor:LVresult_MR_neumann}, we obtain for $z\in \Sigma_{\pi-\om}$ with $\om\in(0,\pi)$
   \begin{equation*}
     |z|\|R(z,\delDir)f\|_{W^{-1,p}(\RRdh,w;X)}\leq |z|\sum_{j=0}^d\|v_j\|_{L^p(\RRdh,w;X)}\leq C\sum_{j=0}^d \|f_j\|_{L^p(\RRdh,w;X)}.
   \end{equation*}
   Taking the infimum over all possible $f_j$ yields that $-\delDir$ is sectorial of angle 0. Together with \cite[Proposition 16.2.1]{HNVW24} this proves \ref{it:W-1_sect}.

We continue with the proof of \ref{it:W-1_semi}. By \ref{it:W-1_sect} and \cite[Theorem G.5.2]{HNVW17} we have that $\delDir$ generates the bounded analytic $C_0$-semigroup $(S(z))_{z\in \Sigma_{\sigma}}$ and
\begin{equation}\label{eq:inv_laplace_trnsfrm_W-1}
  S(z) f =\frac{1}{2\pi \ii}\int_{\Gam}e^{z s}R(s, \delDir)f\dd s,\qquad z\in \Sigma_\sigma, \;\; f\in W^{-1,p}(\RRdh, w;X),
\end{equation}
where $\Gam$ is the upwards orientated boundary of $\Sigma_{\sigma'}\setminus \overline{B(0,r)}$ for some $r>0$ and $\sigma' \in (\frac{\pi}{2}+|\arg z|,\frac{\pi}{2}+\sigma)$.

We show that $S(z)=T_\Dir(z)$ on $W^{-1,p}(\RRdh,w;X)$ for $z\in \Sigma_{\sigma}$.  It follows from \eqref{eq:inv_laplace_trnsfrm_W-1}, Theorems \ref{thm:LVresult}\ref{it:LVresult2} and \ref{thm:LVresult_Neumann}\ref{it:LVresult_Neumann2}, commuting $\d_1$ and the resolvent, and \eqref{eq:Tt_k=-1} that
\begin{align*}
  S(z) f &=\frac{1}{2\pi \ii}\int_{\Gam}e^{z s}\Big(R(s,\delDir)f_0+\d_1R(s,\delNeu)f_1+\sum_{j=2}^d\d_jR(s,\delDir)f_j\Big)\dd s\\
  &= T_\Dir (z)f_0+ \d_1 T_\Neu(z) f_1 + \sum_{j=2}^d \d_j T_\Dir(z)f_j= T_\Dir(z)f,
\end{align*}
which completes the proof of \ref{it:W-1_semi}.

Finally, to prove \ref{it:W-1_calc} let $\om\in(0,\pi)$ and $\phi\in H^1(\Sigma_{\om})\cap \Hinf(\Sigma_{\om})$. By Theorems \ref{thm:LVresult} and \ref{thm:LVresult_Neumann}, we obtain
\begin{align*}
\|\phi(-\delDir) f \|_{W^{-1,p}(\RRdh,w;X)}\leq&\; \|\phi(-\delDir)f_0\|_{L^p(\RRdh,w;X)}+ \|\phi(-\delNeu)f_1\|_{L^p(\RRdh,w;X)}\\
&\;+\sum_{j=2}^d \|\phi(-\delDir)f_j\|_{L^p(\RRdh,w;X)}\\
\leq&\; C \|\phi\|_{\Hinf(\Sigma_\om)}\sum_{j=0}^d \|f_j\|_{L^p(\RRdh,w;X)}.
\end{align*}
Taking the infimum over all possible $f_j$ yields that $-\delDir$ has a bounded $\Hinf$-calculus of angle 0. Together with \cite[Proposition 16.2.6]{HNVW24} this proves \ref{it:W-1_calc}.
\end{proof}

\section{Sectoriality of the Dirichlet and Neumann Laplacian}\label{sec:sect_semigroup}
In this section we study the sectoriality of $-\delDir$ and $-\delNeu$ on weighted Sobolev spaces of arbitrary order $k\in\NN_0$. The results of this section are summarised in the following theorems.

\begin{theorem}[Sectoriality of $\lambda-\delDir$]\label{thm:sect:Dir}
  Let $p\in(1,\infty)$, $k\in\NN_0$, $\gam\in (-1,2p-1)\setminus\{p-1\}$ and let $X$ be a $\UMD$ Banach space. Let $\delDir$ on $W^{k,p}(\RRdh,w_{\gam+kp};X)$ be as in Definition \ref{def:delRRdh}.
  Assume that either
  \begin{enumerate}[(i)]
    \item  $\gam+kp\in (-1,2p-1)$ and $\lambda\geq0$, or,
    \item  $\gam+kp>2p-1$ and $\lambda>0$.
  \end{enumerate}
  Then $\lambda-\delDir$ is sectorial of angle $\om(\lambda-\delDir)=0$.
\end{theorem}

\begin{theorem}[Sectoriality of $\lambda-\delNeu$]\label{thm:sect:Neu}
  Let $p\in(1,\infty)$, $k\in\NN_0\cup\{-1\}$, $\gam\in (-1,2p-1)\setminus\{p-1\}$ and let $X$ be a $\UMD$ Banach space. Let $\delNeu$ on $W^{k+1,p}(\RRdh,w_{\gam+kp};X)$ be as in Definition \ref{def:delRRdh}.
  Assume that either
  \begin{enumerate}[(i)]
    \item  $\gam+kp\in (-1,p-1)$ and $\lambda\geq0$, or,
    \item $\gam+kp>p-1$ and $\lambda>0$.
  \end{enumerate}
   Then $\lambda-\delNeu$ is sectorial of angle $\om(\lambda-\delNeu)=0$.
\end{theorem}

Theorems \ref{thm:sect:Dir} and \ref{thm:sect:Neu} will be proved using elliptic regularity for the corresponding resolvent equations. The elliptic regularity for the Dirichlet and Neumann resolvent equation is shown using an induction argument on $k$ in Sections \ref{subsec:sectDir} and \ref{subsec:sectNeu}, respectively. Finally, in Section \ref{subsec:sect:proof} the elliptic regularity is used to finish the proofs of Theorems \ref{thm:sect:Dir} and \ref{thm:sect:Neu}.

\subsection{Elliptic regularity for the Dirichlet Laplacian}\label{subsec:sectDir}
We start with a preliminary result on the solvability of the resolvent equation if the right hand side is compactly supported.
Recall that we defined the Schwartz space on $\RRdh$ as $\SS(\RRdh;X):=\{u|_{\RRdh}:u\in \SS(\RRd;X)\}$.
\begin{lemma}\label{lem:Schwartz}
   Let $X$ be a Banach space. Then for all $f\in \Cc^{\infty}(\RRdh;X)$ and $\lambda\in \Sigma_{\pi-\om}$ with $\om\in(0,\pi)$ there exists a unique $u\in \SS(\RRdh;X)$ such that
  \begin{equation*}
    \lambda u - \del u=f, \quad u(0, \cdot)=0.
  \end{equation*}
  \end{lemma}
\begin{proof}
  Let $f_{\odd}$ be the odd extension of $f$ as in \eqref{eq:odd_ext}.
Then $f_{\odd}\in \Cc^{\infty}(\RR^d;X)\subseteq \SS(\RR^d;X)$ and taking the Fourier transform of the resolvent equation $\lambda u-\del u=f_{\odd}$ on $\RRd$ yields a solution $u_{\odd}\in \SS(\RR^{d};X)$ given by
\begin{equation*}
  u_{\odd}(x_1,\tilde{x}) = \Big(\mc{F}^{-1}\big[\xi\mapsto \tfrac{(\mc{F} f_{\odd})(\xi)}{\lambda+|\xi|^2}\big]\Big)(x_1,\tilde{x}),\quad \xi\in\RRd,\; \lambda\in \Sigma_{\pi-\om}.
\end{equation*}
Note that $-u_{\odd}(-x_1,\tilde{x})$ also solves the resolvent equation on $\RR^d$. By uniqueness of the solution we obtain $u_{\odd}(x_1,\tilde{x})=-u_{\odd}(-x_1,\tilde{x})$ and therefore $u_{\odd}(0,\tilde{x})=0$.
Thus the restriction $u:=u_{\odd}|_{\RRdh}$ solves $\lambda u-\del u=f$ on $\RRdh$ with $u(0,\cdot)=0$. The uniqueness follows from Corollary \ref{cor:LVresult_MR}.
\end{proof}

Recall that the pointwise multiplication operator $M$ is defined in \eqref{eq:M}. By $[A_1,A_2]=A_1A_2-A_2A_1$ we denote the commutator of two operators $A_1$ and $A_2$. For $a\in\RR$ we use the notation $(a)_+=a$ if $a\geq 0$ and $(a)_+=0$ otherwise. Moreover, note that for all $u\in\SS(\RRdh;X)$ we have the commutation relation
 \begin{equation}\label{eq:comm_Md_del}
    [M^j\d^{\alpha},\del]u=-j(j-1)M^{j-2}\d^{\alpha}u-2jM^{j-1}\d_1\d^{\alpha}u,\qquad j\in\NN_0,\;\;\alpha\in\NN_0^d.
  \end{equation}

The next proposition provides the key argument for proving the sectoriality of the Dirichlet Laplacian on $W^{k,p}(\RRdh, w_{\gam+kp};X)$ for $k\geq 1$.
A version of the proposition below with improved growth $g_{k,\gam}$ on the right hand side of the estimate will be obtained in Section \ref{subsec:growth}.

\begin{proposition}\label{prop:sect_est}
 Let $p\in(1,\infty)$, $k\in \NN_0$, $\gam\in (-1,2p-1)\setminus\{p-1\}$, $\om\in(0,\pi)$ and let $X$ be a $\UMD$ Banach space. Let $\delDir$ on $W^{k,p}(\RRdh, w_{\gam+kp};X)$ be as in Definition \ref{def:delRRdh}. Then for all $f\in W^{k,p}(\RRdh,w_{\gam+kp};X)$ and $\lambda\in \Sigma_{\pi-\om}$, there exists a unique $u\in W^{k+2,p}_{\Dir}(\RRdh,w_{\gam+kp};X)$ such that $\lambda u-\delDir u =f$. Moreover, this solution satisfies
 \begin{equation}\label{eq:est_sect_prop}
  \sum_{|\beta|\leq 2} |\lambda|^{1-\frac{|\beta|}{2}}\|\d^{\beta}u\|_{W^{k,p}(\RRdh,w_{\gam+kp};X)} \leq C g_{k,\gam}(\lambda)  \|f\|_{W^{k,p}(\RRdh,w_{\gam+kp};X)},
\end{equation}
where
\begin{equation}\label{eq:K_IH}
  g_{k,\gam}(\lambda) := \begin{cases}
       1+ |\lambda|^{-\frac{(k-1)_+}{2}} & \mbox{if }\gam\in (-1,p-1) \\
        1+ |\lambda|^{-\frac{k}{2}} & \mbox{if }\gam\in (p-1,2p-1)
      \end{cases},
\end{equation}
and the constant $C$ only depends on $p, k,\gam, \om,d$ and $X$.
\end{proposition}
\begin{proof} For any $k\in\NN_1$ the uniqueness of $u\in W^{k+2,p}_{\Dir}(\RRdh,w_{\gam+kp};X)$ is clear from Corollary \ref{cor:LVresult_MR} and Hardy's inequality (Corollary \ref{cor:Sob_embRRdh}).
The proof for the existence and the estimate goes by induction on $k\geq 0$. The case $k=0$ is stated in Corollary \ref{cor:LVresult_MR}. Assume that the statement of the proposition holds for a fixed $k\in\NN_0$ and for all $\gam\in (-1,2p-1)\setminus\{p-1\}$.
It remains to prove the statement for $k+1$. Throughout the proof, $C$ denotes a constant only depending on $p,k,\gam,\om, d$ and $X$, which may change from line to line.

First, take $f\in \Cc^{\infty}(\RRdh;X)$. Lemma \ref{lem:Schwartz} implies that
\begin{equation}\label{eq:sect_resol_eq_short}
  \lambda u -\delDir u=f, \qquad u(0,\cdot)=0,
\end{equation}
has a unique solution $u\in\SS(\RRdh;X)$.
For $|\alpha|\leq k+1$, let $v_{\alpha}=M^{k+1}\d^{\alpha}u$. Then we have $v_{\alpha}(0,\cdot)=0$ and by \eqref{eq:comm_Md_del} and \eqref{eq:sect_resol_eq_short}
\begin{equation}\label{eq:sect_resol_eq2_short}
  \lambda v_{\alpha} -\delDir v_{\alpha} = M^{k+1}\d^{\alpha} f -k(k+1)M^{k-1}\d^{\alpha}u - 2(k+1)M^k\d_1\d^{\alpha}u.
\end{equation}
Since $f\in \Cc^{\infty}(\RRdh;X)$ and $u\in\SS(\RRdh;X)$, we find that the right hand side of \eqref{eq:sect_resol_eq2_short} is an element $L^p(\RRdh,w_{\gam};X)$. Therefore, the case $k=0$ (see Corollary \ref{cor:LVresult_MR}) implies that $v_{\alpha} \in W^{2,p}_{\Dir} (\RRdh,w_{\gam};X)$
and by Hardy's inequality (if $k\geq 1$) it holds that
\begin{equation}\label{eq:Dir_est1}
  \begin{aligned}
    \sum_{|\beta|\leq 2}|\lambda|^{1-\frac{|\beta|}{2}}\|&\d^{\beta}v_{\alpha}\|_{L^p(\RR^d_+,w_{\gam};X)} \\ \leq &\;  C \|M^{k+1}\d^{\alpha}f\|_{L^p(\RR^d_+,w_{\gam};X)}\\
    &+C\big( k\|M^{k-1}\d^{\alpha}u\|_{L^p(\RRdh,w_{\gam};X)} +\|M^k\d_1\d^{\alpha}u\|_{L^p(\RRdh,w_{\gam};X)}\big) \\
    \leq &\;  C \|f\|_{W^{k+1,p}(\RR^d_+,w_{\gam+(k+1)p};X)}+C\sum_{1\leq |\delta|\leq 2}\|\d^{\delta} u\|_{W^{k,p}(\RRdh,w_{\gam+kp};X)}.
  \end{aligned}
\end{equation}
  If $\gam\in(-1,p-1)$, then note that $\gam+(k+1)p\in((k+1)p-1,(k+2)p-1)$. With Hardy's inequality (Lemma \ref{lem:Hardy}) and the induction hypothesis twice (once with $\gam$ replaced by $\gam+p$), we find
  \begin{equation}\label{eq:est_sect_low}
      \begin{aligned}
 \sum_{1\leq |\delta|\leq 2}&\|\d^{\delta} u\|_{W^{k,p}(\RRdh,w_{\gam+kp};X)}\\
  \leq &\; C\sum_{|\delta|=1}\sum_{|\beta|\leq k}\|\d_1\d^{\beta}\d^{\delta} u\|_{L^p(\RRdh,w_{\gam+p+kp};X)} + \sum_{|\delta|=2}\|\d^{\delta} u\|_{W^{k,p}(\RRdh,w_{\gam+kp};X)} \\
 \leq &\;C \sum_{|\delta|=2}\big(\|\d^{\delta} u\|_{W^{k,p}(\RRdh,w_{\gam+p+kp};X)} + \|\d^{\delta} u\|_{W^{k,p}(\RRdh,w_{\gam+kp};X)} \big)\\
 \leq &\; C \big(1+ |\lambda|^{-\frac{k}{2}}\big)\|f\|_{W^{k,p}(\RRdh, w_{\gam+p+kp};X)}\\
 &\qquad+ C \big(1+ |\lambda|^{-\frac{(k-1)_+}{2}}\big)\|f\|_{W^{k,p}(\RRdh, w_{\gam+kp};X)}\\
 \leq&\; C \big(1+ |\lambda|^{-\frac{k}{2}}\big)\|f\|_{W^{k+1,p}(\RRdh, w_{\gam+(k+1)p};X)},
  \end{aligned}
  \end{equation}
  using Hardy's inequality (Corollary \ref{cor:Sob_embRRdh}) once more in the last estimate.
  If $\gam\in(p-1,2p-1)$, then the induction hypothesis and Hardy's inequality yield
  \begin{equation}\label{eq:est_sect_high}
      \begin{aligned}
 \sum_{1\leq |\delta|\leq 2}\|\d^{\delta} u\|_{W^{k,p}(\RRdh,w_{\gam+kp};X)}
 &\leq C \big(1+ |\lambda|^{-\frac{1}{2}}\big)g_{k,\gam}(\lambda)\|f\|_{W^{k,p}(\RRdh, w_{\gam+kp};X)}\\
 &\leq C \big(1+ |\lambda|^{-\frac{k+1}{2}}\big)\|f\|_{W^{k+1,p}(\RRdh, w_{\gam+(k+1)p};X)}.
  \end{aligned}
  \end{equation}
  From the above estimates, the definition of $v_{\alpha}$ and the fact that $|\alpha|\leq k+1$ was arbitrary, we further estimate \eqref{eq:Dir_est1} as
  \begin{equation}\label{eq:sect_es_short}
    \sum_{|\beta|\leq 2}\sum_{|\alpha|\leq k+1}|\lambda|^{1-\frac{|\beta|}{2}}\|\d^{\beta}M^{k+1}\d^{\alpha}u\|_{L^p(\RR^d_+,w_{\gam};X)}\leq C g_{k+1,\gam}(\lambda)\|f\|_{W^{k+1,p}(\RR^d_+,w_{\gam+(k+1)p};X)}.
  \end{equation}
From \eqref{eq:sect_es_short} we obtain
  \begin{equation*}\label{eq:sect_est2}
      \begin{aligned}
    \sum_{|\beta|\leq 2}&|\lambda|^{1-\frac{|\beta|}{2}} \|\d^{\beta}u\|_{W^{k+1,p}(\RRdh,w_{\gam+(k+1)p};X)}    \\
    \leq&\; \sum_{|\beta|\leq 2} \sum_{|\alpha|\leq k+ 1}|\lambda|^{1-\frac{|\beta|}{2}} \Big(\|\d^{\beta}M^{k+1}\d^{\alpha}u\|_{L^p(\RRdh,w_{\gam};X)}
    + \|[M^{k+1},\d_1^{|\beta|}]\d^{\alpha}u\|_{L^p(\RRdh,w_{\gam};X)}\Big)\\
    \leq&\; Cg_{k+1,\gam}(\lambda) \|f\|_{W^{k+1,p}(\RR^d_+,w_{\gam+(k+1)p};X)} + C\sum_{j\in\{1,2\}}\sum_{j-1\leq |\delta|\leq j}|\lambda|^{1-\frac{j}{2}} \|\d^{\delta}u\|_{W^{k,p}(\RRdh, w_{\gam+kp};X)},
  \end{aligned}
  \end{equation*}
where in the last step we also used Hardy's inequality and
  \begin{align*}
    [M^{k+1},\d_1]\d^{\alpha} u &= -(k+1)M^{k}\d^{\alpha} u,\\
    [M^{k+1},\d_1^2]\d^{\alpha} u &= -k(k+1)M^{k-1}\d^{\alpha} u-2(k+1)M^k\d_1 \d^{\alpha} u.
  \end{align*}
To estimate the last sums, consider the cases $\gam\in(-1,p-1)$ and $\gam\in(p-1,2p-1)$ separately. Then arguing similar as in \eqref{eq:est_sect_low} and \eqref{eq:est_sect_high}, respectively, gives
  \begin{align}\label{eq:est_kgeq1_2}
 \sum_{j\in\{1,2\}}\sum_{j-1\leq |\delta|\leq j}|\lambda|^{1-\frac{j}{2}} \|\d^{\delta}u\|_{W^{k,p}(\RRdh, w_{\gam+kp};X)}\leq Cg_{k+1,\gam}(\lambda) \|f\|_{W^{k+1,p}(\RRdh, w_{\gam+(k+1)p};X)}.
  \end{align}
This proves the required estimate for $f\in \Cc^{\infty}(\RRdh;X)$.

  To conclude, let $f\in W^{k+1,p}(\RRdh,w_{\gam+(k+1)p};X)$. Then by Lemma \ref{lem:density} there exists a sequence $(f_n)_{n\geq 1}\subseteq \Cc^{\infty}(\RRdh;X)$ such that
   \begin{equation*}
     f_n\to f\quad  \text{in }W^{k+1,p}(\RRdh,w_{\gam+(k+1)p};X)\quad\text{ as }n\to \infty.
   \end{equation*}
    Every $f_n$ defines a $u_n\in W^{k+3,p}_{\Dir}(\RRdh,w_{\gam+(k+1)p};X)$ such that $\lambda u_n-\delDir u_n = f_n$ and
\begin{equation*}
  \|u_n-u_m\|_{W^{k+3,p}_{\Dir}(\RRdh,w_{\gam+(k+1)p};X)}\lesssim \|f_n-f_m\|_{W^{k+1,p}(\RRdh,w_{\gam+(k+1)p};X)}\to 0\quad \text{as }m,n\to \infty.
\end{equation*}
Completeness of $W^{k+3,p}_{\Dir}(\RRdh,w_{\gam+(k+1)p};X)$ implies that $u_n$ converges to some $$u  \in W^{k+3,p}_{\Dir}(\RRdh,w_{\gam+(k+1)p};X).$$ Moreover,
\begin{align*}
  \|(\lambda-\delDir)(u_n-u)\|_{ W^{k+1,p}(\RRdh,w_{\gam+(k+1)p};X)}&\lesssim \|u_n-u\|_{ W^{k+3,p}_{\Dir}(\RRdh,w_{\gam+(k+1)p};X)}\to 0,
\end{align*}
as $n\to \infty$. Therefore, $(\lambda-\delDir)u=\lim_{n\to\infty} (\lambda-\delDir)u_n =\lim_{n\to \infty} f_n =f$.
Thus, for any $f\in W^{k+1,p}(\RRdh,w_{\gam+(k+1)p};X)$ there exists a $u\in W^{k+3,p}_{\Dir}(\RRdh,w_{\gam+(k+1)p};X)$ solving $\lambda u-\delDir u =f$ and the required estimate holds. This finishes the proof.
\end{proof}

\subsection{Elliptic regularity for the Neumann Laplacian}\label{subsec:sectNeu}
We continue with elliptic regularity for the Laplace operator with Neumann boundary conditions.
Analogous to Lemma \ref{lem:Schwartz} the Neumann resolvent equation can easily be solved if the right hand side $f$ is in  $C^{\infty}_{{\rm c},1}(\overline{\RRdh};X)$, i.e., $\d^{\alpha} f\in \Cc^{\infty}(\RRdh;X)$ for $\alpha=(\alpha_1,\tilde{\alpha})\in \NN_1\times\NN_0^{d-1}$, see \eqref{eq:setDense}.
\begin{lemma}\label{lem:Schwartz_neu}
   Let $X$ be a Banach space. Then for all $f\in C^{\infty}_{{\rm c},1}(\overline{\RRdh};X)$ and $\lambda\in \Sigma_{\pi-\om}$ with $\om\in(0,\pi)$ there exists a unique $u\in \SS(\RRdh;X)$ such that
  \begin{equation*}
    \lambda u - \del u=f, \quad (\d_1 u)(0, \cdot)=0.
  \end{equation*}
  \end{lemma}
\begin{proof}
This follows similarly as in Lemma \ref{lem:Schwartz} using an even extension from $\RRdh$ to $\RRd$ instead of an odd extension. Note that this extension satisfies $f_{\even}\in \Cc^{\infty}(\RRd;X)\subseteq \SS(\RRd;X)$.
\end{proof}

The following proposition contains the key result for proving the sectoriality of the Neumann Laplacian. A version of the proposition below with improved growth $g_{k+1,\gam}$ on the right hand side of the estimate will be stated in Section \ref{subsec:growth}.
\begin{proposition}\label{prop:sect_est_neu}
 Let $p\in(1,\infty)$, $k\in \NN_0\cup\{-1\}$, $\gam\in (-1,2p-1)\setminus\{p-1\}$ such that $\gam+kp>-1$, $\om\in(0,\pi)$ and let $X$ be a $\UMD$ Banach space. Let $\delNeu$ on $W^{k+1,p}(\RRdh, w_{\gam+kp};X)$ be as in Definition \ref{def:delRRdh}. Then for all $f\in W^{k+1,p}(\RRdh,w_{\gam+kp};X)$ and $\lambda\in \Sigma_{\pi-\om}$, there exists a unique $u\in W^{k+3,p}_{\Neu}(\RRdh,w_{\gam+kp};X)$ such that $\lambda u-\delNeu u =f$. Moreover, this solution satisfies
 \begin{equation*}
  \sum_{|\beta|\leq 2} |\lambda|^{1-\frac{|\beta|}{2}}\|\d^{\beta}u\|_{W^{k+1,p}(\RRdh,w_{\gam+kp};X)} \leq C g_{k+1,\gam}(\lambda)  \|f\|_{W^{k+1,p}(\RRdh,w_{\gam+kp};X)},
\end{equation*}
where (cf. \eqref{eq:K_IH})
\begin{equation*}
  g_{k+1,\gam}(\lambda) = \begin{cases}
       1 +|\lambda|^{-\frac{k}{2}}& \mbox{if }\gam\in (-1,p-1) \\
       1 +|\lambda|^{-\frac{k+1}{2}} & \mbox{if }\gam\in (p-1,2p-1)
      \end{cases},
\end{equation*}
and the constant $C$ only depends on $p, k,\gam, \om,d$ and $X$.
\end{proposition}
\begin{proof}
Throughout the proof, $C$ denotes a constant only depending on $p,k,\gam,\om, d$ and $X$, which may change from line to line. The proof proceeds in two steps. We first prove the case $\gam\in(p-1,2p-1)$ and $k\geq -1$ with induction similar as in the proof of Proposition \ref{prop:sect_est}. Secondly, we prove the case $\gam\in(-1,p-1)$ and $k\geq 0$ using the previous step.\\

\textit{Step 1: the case $\gam\in(p-1,2p-1)$ and $k\geq -1$.} Let $\gam\in(p-1,2p-1)$. The case $k=-1$ is stated in Corollary \ref{cor:LVresult_MR_neumann}.  Assume that the statement of the proposition holds for all $\gam\in (p-1,2p-1)$ and a fixed $k\in\NN_0\cup\{-1\}$. We prove the statement of the proposition for $k+1$.

First of all, note that uniqueness of $u\in W^{k+4,p}_{\Neu}(\RRdh,w_{\gam+(k+1)p};X)$ follows from Corollary \ref{cor:LVresult_MR_neumann} and Hardy's inequality (Corollary \ref{cor:Sob_embRRdh}).
Take $f\in \Cc^{\infty}(\RRdh;X)$ and let $u\in \mc{S}(\RRdh;X)$ be the unique solution to
\begin{equation}\label{eq:Neu_resol}
  \lambda u -\delNeu u = f,\qquad (\d_1 u)(0,\cdot)=0,
\end{equation}
see Lemma \ref{lem:Schwartz_neu}.
For $|\alpha|\leq k+2$ define $v_{\alpha}=M^{k+2}\d^{\alpha}u$.
Then $v_{\alpha}$ satisfies the Dirichlet boundary condition $v_{\alpha}(0,\cdot)=0$ and the equation
\begin{equation*}
  \lambda v_{\alpha} -\delDir v_{\alpha} = M^{k+2}\d^{\alpha} f -(k+1)(k+2)M^{k}\d^{\alpha}u - 2(k+2)M^{k+1}\d_1\d^{\alpha}u.
\end{equation*}
Applying Corollary \ref{cor:LVresult_MR} with $\gam-p\in (-1,p-1)$ yields the estimate \eqref{eq:Dir_est1} with $k$ and $\gam$ replaced by $k+1$ and $\gam-p$, respectively.
The rest of the argument is analogous to the Dirichlet case, see \eqref{eq:est_sect_high}-\eqref{eq:est_kgeq1_2}. For $f\in \Cc^{\infty}(\RRdh;X)$ we obtain the estimate
  \begin{align*}
 \sum_{|\beta|\leq2}|\lambda|^{1-\frac{|\beta|}{2}} \|\d^{\beta}u\|_{W^{k+2,p}(\RRdh, w_{\gam+(k+1)p};X)}&\leq C\big(1+|\lambda|^{-\frac{k+2}{2}}\big) \|f\|_{W^{k+2,p}(\RRdh, w_{\gam+(k+1)p};X)}.
  \end{align*}
By a density argument similar as in the proof of Proposition \ref{prop:sect_est} the induction is finished. This proves the statement of the proposition for $k\in\NN_0\cup\{-1\}$ and $\gam\in(p-1,2p-1)$.\\

\textit{Step 2: the case $\gam\in(-1,p-1)$ and $k\geq 0$.}
Let $\gam\in(-1,p-1)$ and note that for $k=0$ the result in already contained in Theorem \ref{thm:LVresult_Neumann}. Therefore, we let $k\in \NN_1$ from now on. Then uniqueness of $u\in W^{k+3,p}_{\Neu}(\RRdh,w_{\gam+kp};X)$ follows from the uniqueness in $W^{3,p}_{\Neu}(\RRdh,w_{\gam};X)$ (case $k=0$) and Hardy's inequality.

To continue, take $f\in C^{\infty}_{{\rm c},1}(\overline{\RRdh};X)$. Lemma \ref{lem:Schwartz_neu} implies that the resolvent equation \eqref{eq:Neu_resol} has a unique solution $u\in\mc{S}(\RRdh;X)$. Let $j\in \{1,\dots, d\}$ and define $v_j=\d_j u$. If $j=1$, then $v_1$ satisfies
\begin{equation*}\label{eq:v_1}
  \lambda v_1 -\delDir v_1 = \d_1 f,\qquad v_1(0,\cdot)=0.
\end{equation*}
If $j\in\{2,\dots, d\}$, then $v_j$ satisfies
\begin{equation*}\label{eq:v_j}
  \lambda v_j -\delNeu v_j = \d_j f,\qquad (\d_1v_j)(0,\cdot)=0.
\end{equation*}
Applying Step 1 for $k-1$ and $\gam+p\in (p-1,2p-1)$ to estimate $v_j$ for $j\in\{2,\dots, d\}$ and Proposition \ref{prop:sect_est} to estimate $v_1$, yields
\begin{align*}
   \sum_{|\beta|\leq 2} &|\lambda|^{1-\frac{|\beta|}{2}}\|\d^{\beta}u\|_{W^{k+1,p}(\RRdh,w_{\gam+kp};X)}\\
    & =\sum_{|\beta|\leq 2} |\lambda|^{1-\frac{|\beta|}{2}} \big(\|\d^{\beta} u\|_{W^{k,p}(\RRdh, w_{\gam+kp};X)}+ \sum_{j=1}^d \|\d^{\beta} v_j\|_{W^{k,p}(\RRdh, w_{\gam+kp};X)} \big) \\
   &\leq C\big(1+|\lambda|^{-\frac{k}{2}}\big)\big(\|f\|_{W^{k,p}(\RRdh, w_{\gam+kp};X)} + \sum_{j=1}^d\|\d_j f\|_{W^{k,p}(\RRdh, w_{\gam+kp};X)}\big)\\
   &\leq C\big(1+|\lambda|^{-\frac{k}{2}}\big) \|f\|_{W^{k+1,p}(\RRdh, w_{\gam+kp};X)}.
\end{align*}
From Lemma \ref{lem:density} we have that $C^{\infty}_{{\rm c},1}(\overline{\RRdh};X)$ is dense in $W^{k+1,p}(\RRdh,w_{\gam+kp};X)$. Therefore, a similar density argument as in the proof of Proposition \ref{prop:sect_est} yields that for any $f\in W^{k+1,p}(\RRdh,w_{\gam+kp};X)$ there exists a $u\in W^{k+3,p}_{\Neu}(\RRdh, w_{\gam+kp};X)$ solving \eqref{eq:Neu_resol} and the required estimate holds. This finishes the proof.
\end{proof}

\subsection{Sectoriality and the proof of Theorems \ref{thm:sect:Dir} and \ref{thm:sect:Neu}}\label{subsec:sect:proof}
The sectoriality of the Dirichlet and Neumann Laplacian in Theorems \ref{thm:sect:Dir} and \ref{thm:sect:Neu} is now an easy consequence of Propositions \ref{prop:sect_est} and \ref{prop:sect_est_neu}.

\begin{proof}[Proof of Theorems \ref{thm:sect:Dir} and \ref{thm:sect:Neu}] Let $\gam\in (-1,2p-1)\setminus\{p-1\}$ and $k\in\NN_0\cup\{-1\}$ be such that $\gam+kp>-1$. We first consider the Dirichlet Laplacian. Assume that $\gam+kp\in(-1,2p-1)$ and $\lambda\geq 0$, or, $\gam+kp>2p-1$ and $\lambda>0$. We prove that $\lambda-\delDir$ on $W^{k,p}(\RRdh, w_{\gam+kp};X)$ with domain $W^{k+2,p}_{\Dir}(\RRdh, w_{\gam+kp};X)$ is sectorial of angle zero. Note that this operator is closed by the estimate in Proposition \ref{prop:sect_est} (see Remark \ref{rem:Inj_XUMD}).
 The resolvent estimate for $\lambda-\delDir$ now follows from Proposition \ref{prop:sect_est} (corresponding to $\beta=0$ below, the estimates for $|\beta|\in\{1,2\}$ will be needed in Section \ref{sec:calculus}). Indeed, we have
 \begin{equation}\label{eq:resol_est_A}
     \begin{aligned}
    \sup_{z\in \CC\setminus\overline{\Sigma_{\om}}}&\sum_{|\beta|\leq 2}|z|^{1-\frac{|\beta|}{2}}\|\d^{\beta}R(z,\lambda-\delDir)f\|_{W^{k,p}(\RRdh,w_{\gam+kp};X)} \\ 
    &= \sup_{z\in \lambda+\Sigma_{\pi-\om}}\sum_{|\beta|\leq 2}|z-\lambda|^{1-\frac{|\beta|}{2}}\|\d^{\beta}R(z,\delDir)f\|_{W^{k,p}(\RRdh,w_{\gam+kp};X)} \\
     & \leq  C\sup_{z\in \lambda+\Sigma_{\pi-\om}}\sum_{|\beta|\leq 2 }|z|^{1-\frac{|\beta|}{2}}\|\d^{\beta}R(z,\delDir)f\|_{W^{k,p}(\RRdh,w_{\gam+kp};X)}\\
     &\leq C\sup_{z\in \lambda+\Sigma_{\pi-\om}}g_{k,\gam}(z)\|f\|_{W^{k,p}(\RRdh,w_{\gam+kp};X)}=\tilde{C}g_{k,\gam}(\lambda)\|f\|_{W^{k,p}(\RRdh,w_{\gam+kp};X)},
  \end{aligned}
 \end{equation}
  using that for $z-\lambda\in \Sigma_{\pi-\om}$
  \begin{equation*}
    |z-\lambda|\leq \begin{cases}
                |z| & \mbox{if } \om\in[\frac{\pi}{2},\pi) \\
                \frac{|z|}{\sin \om} & \mbox{if }\om\in(0,\frac{\pi}{2}]
              \end{cases}.
  \end{equation*}
  Note that $g_{k,\gam}(\lambda)=1$ if $\gam+kp<2p-1$, so that the constant in the sectoriality estimate is independent of $\lambda$. Thus $-\delDir$ is also sectorial in this case.

Furthermore, by Proposition \ref{prop:sect_est} the mapping $$\lambda-\delDir:W^{k+2,p}_{\Dir}(\RRdh,w_{\gam+kp};X)\to W^{k,p}(\RRdh,w_{\gam+kp};X)$$ is surjective. The sectoriality now follows from Remark \ref{rem:Inj_XUMD}. This finishes the proof of Theorem \ref{thm:sect:Dir}

The proof of Theorem \ref{thm:sect:Neu} for the Neumann Laplacian on $W^{k+1,p}(\RRdh,w_{\gam+kp};X)$ is similar to the Dirichlet case using Proposition \ref{prop:sect_est_neu}.
\end{proof}

\section{The Dirichlet and Neumann heat semigroup}\label{sec:semigroup}
In this section, we study the growth of the Dirichlet and Neumann heat semigroup.
Recall that the odd and even extensions from $\RRdh$ to $\RRd$ are defined as in \eqref{eq:odd_ext}. Let $z\in\CC_+$ and recall from Section \ref{sec:results} that we defined the kernels
\begin{equation}\label{eq:H}
  H^{d,\pm}_z(x,y):=G^d_z(x_1-y_1,\tilde{x}-\tilde{y})\pm G^d_z(x_1+y_1,\tilde{x}-\tilde{y}),\qquad x,y\in\RRdh,
\end{equation}
where $G^d_z$ is the standard heat kernel on $\RRd$, see \eqref{eq:HeatKernelRd}.
The Dirichlet heat semigroup $T_{\Dir}$ is defined for $f\in L^p(\RRdh, w_{\gam};X)$ with $\gam\in (-1,2p-1)\setminus\{p-1\}$ by
\begin{equation}\label{eq:Tt}
  T_{\Dir}(z)f(x):=H^{d,-}_{z} * f(x):=\int_{\RRdh}H^{d,-}_{z}(x,y)f(y)\dd y=\int_{\RR^d}G^d_z(x-y) f_{\odd}(y)\dd y,
\end{equation}
and the Neumann heat semigroup $T_{\Neu}$ is defined for $f\in W^{\ell,p}(\RRdh, w_{\gam};X)$ with $\gam\in (-1,p-1)$ and $\ell\in\{0,1\}$ by
\begin{equation}\label{eq:Tt_Neu}
  T_{\Neu}(z)f(x):=H^{d,+}_z * f(x):=\int_{\RRdh}H^{d,+}_z(x,y)f(y)\dd y=\int_{\RR^d}G^d_z(x-y) f_{\even}(y)\dd y.
\end{equation}
Note that the above definitions of the semigroups are well defined by Theorems \ref{thm:LVresult} and \ref{thm:LVresult_Neumann}.

We prove the following theorems concerning the heat semigroups.
\begin{theorem}[Dirichlet heat semigroup]\label{thm:Dirsemi}
  Let $p\in(1,\infty)$, $k\in\NN_0$, $\gam\in (-1,2p-1)\setminus\{p-1\}$ and let $X$ be a $\UMD$ Banach space. Let $\delDir$ on $W^{k,p}(\RRdh, w_{\gam+kp};X)$ be as in Definition \ref{def:delRRdh}.
Then $(T_{\Dir}(z))_{z\in \Sigma_{\sigma}}$ with $\sigma\in(0,\frac{\pi}{2})$ is an analytic $C_0$-semigroup on $W^{k,p}(\RR^d_+,w_{\gam+kp};X)$ which is generated by $\delDir$.
Moreover, the heat semigroup $T_{\Dir}(t)$ on $W^{k,p}(\RRdh, w_{\gam+kp};X)$ satisfies the following growth properties:
   \begin{enumerate}[(i)]
   \item\label{it:Dirsemi2} if $\gam+kp\in (-1,2p-1)$, then $T_\Dir(t)$ is bounded,
     \item\label{it:Dirsemi3}  if $\gam+kp>2p-1$, then $T_\Dir(t)$ has polynomial growth and for any $\eps>0$ there are constants $c,C>0$ only depending on $p,k,\gam, \eps, d $ and $X$, such that
   \begin{equation*}
       c\big(1+t^{\frac{\gam+kp-2p+1}{2p}}\big)\leq\|T_\Dir(t)\|\leq C \big(1+t^{\frac{\gam+kp-2p+1+\eps}{2p}}\big), \qquad t\geq 0.
   \end{equation*}
   \end{enumerate}
\end{theorem}

\begin{theorem}[Neumann heat semigroup]\label{thm:Neusemi}
  Let $p\in(1,\infty)$, $k\in\NN_0\cup\{-1\}$, $\gam\in (-1,2p-1)\setminus\{p-1\}$ such that $\gam+kp>-1$ and let $X$ be a $\UMD$ Banach space. Let $\delNeu$ on $W^{k+1,p}(\RRdh, w_{\gam+kp};X)$ be as in Definition \ref{def:delRRdh}.
Then $(T_{\Neu}(z))_{z\in \Sigma_{\sigma}}$ with $\sigma\in(0,\frac{\pi}{2})$ is an analytic $C_0$-semigroup on $W^{k+1,p}(\RR^d_+,w_{\gam+kp};X)$ which is generated by $\delNeu$.
Moreover, the heat semigroup $T_{\Neu}(t)$ on $W^{k+1,p}(\RRdh, w_{\gam+kp};X)$ satisfies the following growth properties:
   \begin{enumerate}[(i)]
   \item\label{it:Neusemi2}  if $\gam+kp\in (-1,p-1)$, then $T_\Neu(t)$ is bounded,
     \item\label{it:Neusemi3}  if $\gam+kp>p-1$, then $T_\Neu(t)$ has polynomial growth and for any $\eps>0$ there are constants $c,C>0$ only depending on $p,k,\gam,\eps,d $ and $X$, such that
   \begin{equation*}
       c\big(1+t^{\frac{\gam+kp-p+1}{2p}}\big)\leq\|T_\Neu(t)\|\leq C \big(1+t^{\frac{\gam+kp-p+1+\eps}{2p}}\big), \qquad t\geq 0.
   \end{equation*}
   \end{enumerate}
\end{theorem}
From Theorems \ref{thm:sect:Dir} and \ref{thm:sect:Neu} together with \cite[Theorem G.5.2]{HNVW17}, the statements in
 Theorems \ref{thm:Dirsemi}\ref{it:Dirsemi2} and \ref{thm:Neusemi}\ref{it:Neusemi2} immediately follow. In Section \ref{subsec:generator} we prove that the generators of the Dirichlet and Neumann heat semigroup are $\delDir$ and $\delNeu$, respectively. Finally, in Section \ref{subsec:growth} we study the growth of the semigroups and prove Theorems \ref{thm:Dirsemi}\ref{it:Dirsemi2} and \ref{thm:Neusemi}\ref{it:Neusemi2}.

\subsection{Generators of the Dirichlet and Neumann heat semigroup}\label{subsec:generator}
We start with some preliminaries about the consistency of the resolvents and the semigroups. \\

Let $X_0$ and $X_1$ be two compatible Banach spaces and suppose that $B_{0}\in\mc{L}(X_0)$ and $B_{1}\in \mc{L}(X_1)$. Then we call the operators $B_{0}$ and $B_{1}$ \emph{consistent} if
\begin{equation*}
  B_{0}u=B_{1}u\qquad \text{ for all }u\in X_0\cap X_1.
\end{equation*}
For $z\in \Sigma \subseteq \CC$ the two families of operators $B_{0}(z)\in\mc{L}(X_0)$  and $B_{1}(z)\in\mc{L}(X_1)$ are called consistent if $B_{0}(z)$ and $B_{1}(z)$ are consistent for all $z\in \Sigma$.\\

From Theorems \ref{thm:LVresult}\ref{it:LVresult2} and \ref{thm:LVresult_Neumann}\ref{it:LVresult_Neumann2} we obtain an easy corollary on the consistency of the semigroups.
\begin{corollary}\label{cor:cons_semigroupLp}
  Let $j\in\{0,1\}$. Let $p_j\in(1,\infty)$, $\sigma\in (0,\frac{\pi}{2})$ and let $X$ be a $\UMD$ Banach space. Let $z\in \Sigma_{\sigma}$, then the following assertions hold.
  \begin{enumerate}[(i)]
    \item If $\gam_j\in (-1,2p_j-1)\setminus\{p_j-1\}$, then $e^{z\delDir}$ on $L^{p_j}(\RRdh, w_{\gam_j};X)$ for $j\in\{0,1\}$ are consistent.
    \item If $\gam_j\in (-1,p_j-1)$ and $\ell_j\in\{0,1\}$, then $e^{z\delNeu}$ on $W^{\ell_j, p_j}(\RRdh, w_{\gam_j};X)$ for $j\in\{0,1\}$ are consistent.
  \end{enumerate}
\end{corollary}

To prove Theorems \ref{thm:Dirsemi} and \ref{thm:Neusemi} we need the semigroups to be consistent on the weighted Sobolev spaces that we consider. To this end, we prove the consistency of the resolvents. We start with the consistency of the resolvents for the Dirichlet Laplacian.
\begin{lemma}[Consistency of Dirichlet resolvents]\label{lem:consistent_resol_RRdh}
  Let $j\in\{0,1\}$. Let $p_j\in(1,\infty)$, $k_j\in \NN_0$, $\gam_j\in (-1,2p_j-1)\setminus\{p_j-1\}$ and let $X$ be a $\UMD$ Banach space. Let $A_j:=\delDir$ on $W^{k_j,p_j}(\RRdh,w_{\gam_j+k_jp_j};X)$
    be as in Definition \ref{def:delRRdh}. Then the resolvents $R(z, A_j-\lambda)$ for $j\in\{0,1\}$ are consistent for all $\lambda\geq0$ and $z\in \Sigma_{\pi-\om}$ with $\om\in(0,\pi)$.
\end{lemma}
\begin{proof}
First, let $k_0=k_1=0$ and $\lambda=0$, then by \cite[Proposition 2.4]{Ar94} and Corollaries \ref{cor:LVresult_MR} and \ref{cor:cons_semigroupLp}, we have that the resolvents are consistent.
If $\lambda>0$, then consistency of the resolvents for $\lambda=0$ implies
\begin{equation}\label{eq:resolDirlambda}
  R(z, A_0-\lambda)=R(z+\lambda, \delDir) = R(z,A_1-\lambda),
\end{equation}
since $z+\lambda \in \lambda+\Sigma_{\pi-\om}\subseteq \Sigma_{\pi-\om}$. This proves the case $k_0=k_1=0$.

To prove the general case it now suffices to prove that for fixed $j\in\{0,1\}$ the resolvents on $W^{k_j,p_j}(\RRdh,w_{\gam_j+k_jp_j};X)$ and $L^{p_j}(\RRdh, w_{\gam_j};X)$ are consistent.
By Proposition \ref{prop:sect_est} the resolvent equation
\begin{equation*}
  (z+\lambda)u - \delDir u=f,\qquad z\in \Sigma_{\pi-\om},\;\;f\in L^{p_j}(\RRdh, w_{\gam_j};X)\cap W^{k_j,p_j}(\RRdh, w_{\gam_j+k_jp_j};X)
\end{equation*}
has unique solutions $u_0\in W^{2,p_j}_{\Dir}(\RRdh, w_{\gam_j};X)$ and $u_1\in W^{k_j+2,p_j}_{\Dir}(\RRdh, w_{\gam_j+k_jp_j};X)$. By Hardy's inequality (Corollary \ref{cor:Sob_embRRdh}) it follows that $u_0=u_1$.
This proves the lemma.
\end{proof}

Moreover, we have a similar lemma on the consistency of the Neumann resolvents.
\begin{lemma}[Consistency of Neumann resolvents]\label{lem:consistent_resol_RRdh_Neumann}
  Let $j\in\{0,1\}$. Let $p_j\in(1,\infty)$, $k_j\in \NN_0$, $\ell_j\in\{0,1\}$, $\gam_j\in (-1,p_j-1)$ and let $X$ be a $\UMD$ Banach space. Let $A_j:=\delNeu$ on $W^{k_j+\ell_j,p_j}(\RRdh,w_{\gam_j+k_jp_j};X)$
    be as in Definition \ref{def:delRRdh}. Then the resolvents $R(z, A_j-\lambda)$ for $j\in\{0,1\}$ are consistent for all $\lambda\geq 0$ and $z\in \Sigma_{\pi-\om}$ with $\om\in(0,\pi)$.
\end{lemma}
\begin{proof}
  By \cite[Proposition 2.4]{Ar94} and Corollaries \ref{cor:LVresult_MR_neumann} and \ref{cor:cons_semigroupLp}, the case $k_0=k_1=0$ with $\lambda=0$ follows. Arguing as in \eqref{eq:resolDirlambda} gives the result for $\lambda>0$ as well.

  To prove the general case, it now suffices to prove that for fixed $j\in\{0,1\}$ the resolvents on $W^{k_j+\ell_j,p_j}(\RRdh,w_{\gam_j+k_jp_j};X)$ and $W^{\ell_j,p_j}(\RRdh, w_{\gam_j};X)$ are consistent.
  This follows similarly as in the proof of Lemma \ref{lem:consistent_resol_RRdh} using Proposition \ref{prop:sect_est_neu} and Hardy's inequality.
\end{proof}

We can now prove that the Dirichlet and Neumann semigroup are generated by $\delDir$ and $\delNeu$, as is stated in Theorems \ref{thm:Dirsemi} and \ref{thm:Neusemi}, respectively.
\begin{proof}[Proof of Theorems \ref{thm:Dirsemi} and \ref{thm:Neusemi}: generator identification]
Let $p\in(1,\infty)$, $k\in\NN_0$ and let $\gam\in(-1,2p - 1)\setminus\{p-1\}$. We first prove that  $(T_{\Dir}(z))_{z\in \Sigma_\sigma}$ with $\sigma\in(0,\frac{\pi}{2})$ is an analytic $C_0$-semigroup on $W^{k,p}(\RRdh, w_{\gam+kp};X)$ generated by $\delDir$.
Let
\begin{equation*}
  A_0:=\delDir \quad \text{on }L^p(\RRdh, w_{\gam};X)\quad\text{ and }\quad  A_1:=\delDir \quad \text{on }W^{k,p}(\RRdh, w_{\gam+kp};X)
\end{equation*}
be as in Definition \ref{def:delRRdh}. By Theorem \ref{thm:sect:Dir} and \cite[Theorem G.5.2]{HNVW17} we have that $A_1-\lambda$ with $\lambda>0$ generates the bounded analytic $C_0$-semigroup $(e^{-\lambda z}S(z))_{z\in \Sigma_{\sigma}}$ and
\begin{equation}\label{eq:inv_laplace_trnsfrm}
  e^{-\lambda z}S(z) f =\frac{1}{2\pi \ii}\int_{\Gam}e^{z s}R(s+\lambda, A_1)f\dd s,\qquad z\in \Sigma_\sigma, \;\; f\in W^{k,p}(\RRdh, w_{\gam+kp};X),
\end{equation}
where $\Gam$ is the upwards orientated boundary of $\Sigma_{\sigma'}\setminus \overline{B(0,r)}$ for some $r>0$ and $\sigma' \in (\frac{\pi}{2}+|\arg z|,\frac{\pi}{2}+\sigma)$.

We show that $S(z)=T_\Dir(z)$ on $W^{k,p}(\RRdh, w_{\gam+kp};X)$ for $z\in \Sigma_{\sigma}$.
By \eqref{eq:inv_laplace_trnsfrm}, Lemma \ref{lem:consistent_resol_RRdh} and Theorem \ref{thm:LVresult}\ref{it:LVresult2} we obtain for $f\in W^{k,p}(\RRdh, w_{\gam+kp};X)$
\begin{align*}
  e^{-\lambda z}S(z) f  & =\frac{1}{2\pi \ii}\int_{\Gam}e^{z s}R(s+\lambda, A_0)f\dd s
    =\frac{1}{2\pi \ii}\int_{\Gam}e^{(s-\lambda) z}R(s, A_0)f\dd s = e^{-\lambda z }T_\Dir(z) f.
\end{align*}
Therefore, $S(z)=T_\Dir(z)$ and thus $T_\Dir(z)$ is the analytic $C_0$-semigroup generated by $A_1$. This completes the Dirichlet case from Theorem \ref{thm:Dirsemi}.\\

We continue with the Neumann Laplacian.
If $\gam\in (-1,p-1)$ and $k\in \NN_0$, then define
\begin{equation*}
  A_0:=\delNeu \quad \text{on }W^{1,p}(\RRdh, w_{\gam};X)\quad\text{ and }\quad  A_1:=\delNeu \quad \text{on }W^{k+1,p}(\RRdh, w_{\gam+kp};X).
\end{equation*}
If $\gam\in (p-1,2p-1)$ and $k\in \NN_0\cup\{-1\}$, then define
\begin{equation*}
  A_0:=\delNeu \quad \text{on }L^p(\RRdh, w_{\gam-p};X)\quad\text{ and }\quad  A_1:=\delNeu \quad \text{on }W^{k+1,p}(\RRdh, w_{\gam+kp};X).
\end{equation*}
For both cases, we can proceed similarly as for the Dirichlet Laplacian using Theorem \ref{thm:sect:Neu}, Lemma \ref{lem:consistent_resol_RRdh_Neumann} and Theorem \ref{thm:LVresult_Neumann}\ref{it:LVresult_Neumann2}. This proves that $(T_{\Neu}(z))_{z\in \Sigma_\sigma}$ with $\sigma\in(0,\frac{\pi}{2})$ is an analytic $C_0$-semigroup on $W^{k+1,p}(\RRdh, w_{\gam+kp};X)$ which is generated by $\delNeu$, as stated in Theorem \ref{thm:Neusemi}.
\end{proof}

\subsection{Growth of the Dirichlet and Neumann semigroups}\label{subsec:growth}
We proceed with investigating the growth of the Dirichlet and Neumann heat semigroup on weighted Sobolev spaces and we prove Theorems \ref{thm:Dirsemi}\ref{it:Dirsemi3} and \ref{thm:Neusemi}\ref{it:Neusemi3}.

Before turning to the growth of the semigroups, we first reconsider the elliptic regularity estimates from Propositions \ref{prop:sect_est} and \ref{prop:sect_est_neu}. In particular, using complex interpolation we can improve these estimates with a sharper function $g_{k,\gam}(\lambda)$ on the right-hand side of the estimates. This will allow us to find sharper growth rates for the semigroups later on.

Let $p\in(1,\infty)$, $k\in \NN_0$, $\gam\in(-1,2p-1)$ and let $\eps>0$. Then we define
\begin{equation}\label{eq:improved_h}
  h_{k,\gam,\eps}(\lambda) := \begin{cases}
       1 & \mbox{if }\gam+kp\in (-1,2p-1) \\
        1+ |\lambda|^{-\frac{\gam+kp-2p+1+\eps}{2p}} & \mbox{if }\gam+kp>2p-1
      \end{cases}.
\end{equation}

For the Dirichlet Laplacian, we obtain the following improved elliptic regularity result.
\begin{proposition}[Elliptic regularity for the Dirichlet Laplacian]\label{prop:sect_est_improved}
 Let $p\in(1,\infty)$, $k\in \NN_0$, $\gam\in (-1,2p-1)\setminus\{p-1\}$, $\eps>0$, $\om\in(0,\pi)$ and let $X$ be a $\UMD$ Banach space. Let $\delDir$ on $W^{k,p}(\RRdh, w_{\gam+kp};X)$ be as in Definition \ref{def:delRRdh}. Then for all $f\in W^{k,p}(\RRdh,w_{\gam+kp};X)$ and $\lambda\in \Sigma_{\pi-\om}$, there exists a unique $u\in W^{k+2,p}_{\Dir}(\RRdh,w_{\gam+kp};X)$ such that $\lambda u-\delDir u =f$. Moreover, this solution satisfies
 \begin{equation*}
  \sum_{|\beta|\leq 2} |\lambda|^{1-\frac{|\beta|}{2}}\|\d^{\beta}u\|_{W^{k,p}(\RRdh,w_{\gam+kp};X)} \leq C h_{k,\gam,\eps}(\lambda)  \|f\|_{W^{k,p}(\RRdh,w_{\gam+kp};X)},
\end{equation*}
where $h_{k,\gam,\eps}$ is defined in \eqref{eq:improved_h} and the constant $C$ only depends on $p, k,\gam, \eps, \om,d$ and $X$.
\end{proposition}
\begin{proof}
  By Proposition \ref{prop:sect_est} we only have to prove the estimate for $k\in\NN_1$.   First, let  $\gam\in(p-1,2p-1)$  and without loss of generality we may assume $\eps\in (0,2p-1-\gam)$. Define $\gam_0=p-1-\eps\in  (-1,p-1)$ and $\gam_1=2p-1-\eps\in (p-1,2p-1)$. Then $\gam=(1-\theta)\gam_0+\theta\gam_1$ with $\theta=(\gam-p+1+\eps)/p$. Then Proposition \ref{prop:compl_int_gam} twice, properties of the complex interpolation method, Lemma \ref{lem:consistent_resol_RRdh} and the estimate in Proposition \ref{prop:sect_est}, yield
\begin{equation*}
  \begin{aligned}
\sum_{|\beta|\leq 2}|&\lambda|^{1-\frac{|\beta|}{2}}\|\d^{\beta}u\|_{W^{k,p}(\RR^d_+,w_{\gam+kp};X)}\\
 &\leq C \sum_{|\beta|\leq 2}|\lambda|^{1-\frac{|\beta|}{2}}\|\d^{\beta}u\|_{[W^{k,p}(\RR^d_+,w_{\gam_0+kp};X),W^{k,p}(\RR^d_+,w_{\gam_1+kp};X)]_\theta}\\
&\leq C (g_{k,\gam_0}(\lambda))^{1-\theta}(g_{k,\gam_1}(\lambda))^{\theta}\|f\|_{[W^{k,p}(\RR^d_+,w_{\gam_0+kp};X),W^{k,p}(\RR^d_+,w_{\gam_1 +kp};X)]_\theta}\\
&\leq C \big(1+ |\lambda|^{-\frac{k-1}{2}}\big)^{1-\theta}\big(1+ |\lambda|^{-\frac{k}{2}}\big)^{\theta}\|f\|_{W^{k,p}(\RR^d_+,w_{\gam+kp};X)}\\
&\leq C \big(1+ |\lambda|^{-\frac{\gam+kp-2p+1+\eps}{2p}}\big)\|f\|_{W^{k,p}(\RR^d_+,w_{\gam+kp};X)},
\end{aligned}
\end{equation*}
where we used that
\begin{equation*}
  -\frac{k-1}{2}(1-\theta)-\frac{k}{2}\theta=-\half(k-1+\theta)=-\frac{\gam+kp-2p+1+\eps}{2p}.
\end{equation*}

If $\gam\in (-1,p-1)$, then by inspection of the inductive proof of Proposition \ref{prop:sect_est}, we see that on $W^{k,p}(\RRdh, w_{\gam+kp};X)$ the function $g_{k,\gam}$ on the right-hand side of the elliptic regularity estimate \eqref{eq:est_sect_prop} is determined by $g_{k,\gam}(\lambda)=g_{k-1,\gam+p}(\lambda)$. Thus in a similar way, we obtain for $\gam\in (-1,p-1)$ and $k\geq 2$ the estimate
 \begin{equation*}
  \sum_{|\beta|\leq 2} |\lambda|^{1-\frac{|\beta|}{2}}\|\d^{\beta}u\|_{W^{k,p}(\RRdh,w_{\gam+kp};X)} \leq C h_{k-1,\gam+p,\eps}(\lambda)  \|f\|_{W^{k,p}(\RRdh,w_{\gam+kp};X)},
\end{equation*}
and the result follows upon noting that
$h_{k-1,\gam+p,\eps}(\lambda)=h_{k,\gam,\eps}(\lambda)$.
\end{proof}

In a similar way, there is the following improved elliptic regularity result for $\delNeu$.
\begin{proposition}[Elliptic regularity for the Neumann Laplacian]\label{prop:sect_est_neu_improved}
 Let $p\in(1,\infty)$, $k\in \NN_0\cup\{-1\}$ and $\gam\in (-1,2p-1)\setminus\{p-1\}$ be such that $\gam+kp>-1$. Moreover, let $\eps>0$, $\om\in(0,\pi)$ and let $X$ be a $\UMD$ Banach space. Let $\delNeu$ on $W^{k+1,p}(\RRdh, w_{\gam+kp};X)$ be as in Definition \ref{def:delRRdh}. Then for all $f\in W^{k+1,p}(\RRdh,w_{\gam+kp};X)$ and $\lambda\in \Sigma_{\pi-\om}$, there exists a unique $u\in W^{k+3,p}_{\Neu}(\RRdh,w_{\gam+kp};X)$ such that $\lambda u-\delNeu u =f$. Moreover, this solution satisfies
 \begin{equation*}
  \sum_{|\beta|\leq 2} |\lambda|^{1-\frac{|\beta|}{2}}\|\d^{\beta}u\|_{W^{k+1,p}(\RRdh,w_{\gam+kp};X)} \leq C h_{k+1,\gam,\eps}(\lambda)  \|f\|_{W^{k+1,p}(\RRdh,w_{\gam+kp};X)},
\end{equation*}
where $h_{k,\gam,\eps}$ is defined in \eqref{eq:improved_h} and the constant $C$ only depends on $p, k,\gam, \eps, \om,d$ and $X$.
\end{proposition}
\begin{proof}
The proof uses Proposition \ref{prop:sect_est_neu} and is analogous to the proof of Proposition \ref{prop:sect_est_improved}.
\end{proof}

With the elliptic regularity estimates from Propositions \ref{prop:sect_est_improved} and \ref{prop:sect_est_neu_improved} at hand, we can establish (almost) optimal growth bounds for the Dirichlet and Neumann heat semigroup.

\begin{proof}[Proof of Theorems \ref{thm:Dirsemi}\ref{it:Dirsemi3} and \ref{thm:Neusemi}\ref{it:Neusemi3}]
We start with the Dirichlet case, i.e., we prove that for $\gam+kp>2p-1$, the semigroup $T_{\Dir}(t)$ on $W^{k,p}(\RRdh, w_{\gam+kp};X)$ grows polynomially and there are constants $c,C>0$  only depending on $p,k,\gam,\eps, d $ and $X$, such that
\begin{equation}\label{eq:upperbound}
       c\big(1+t^{\frac{\gam+kp-2p+1}{2p}}\big)\leq\|T_\Dir(t)\|\leq C \big(1+t^{\frac{\gam+kp-2p+1+\eps}{2p}}\big), \qquad t\geq 0,\,\eps>0.
   \end{equation}

By Proposition \ref{prop:sect_est_improved}, Lemma \ref{lem:growth_semigroup_abstract} and the fact that $\delDir$ generates $T_\Dir$, it follows that $T_\Dir(t)$ grows at most polynomially and the upper bound in \eqref{eq:upperbound} holds.

It remains to show that $T_\Dir(t)$ grows at least polynomially. We provide an example of a function $h\in W^{k,p}(\RRdh, w_{\gam+kp};X)$ such that $\|T_\Dir(t) h\|_{W^{k,p}(\RRdh, w_{\gam+kp};X)}\geq c_{p,k,\gam} t^{\frac{\gam+kp-2p+1}{2p}}$.
By setting $h= \zeta\otimes g$ with $g\in \mc{S}(\RR^{d-1};X)$ if $d\geq 2$, it suffices to find an $\zeta\in W^{k,p}(\RR_+,w_{\gam+kp}) $ such that $\|T_\Dir(t) \zeta\|_{W^{k,p}(\RR_+,w_{\gam+kp})}\geq c_{p,k,\gam} t^{\frac{\gam+kp-2p+1}{2p}}$.

For $\zeta$ it suffices to take the cut-off function defined by
\begin{equation}\label{eq:cut-offzeta}
  \zeta\in C^{\infty}(\RR_+)\quad \text{ such that }\quad \zeta(x)=\begin{cases}
                                                                     1 & \mbox{if } x\leq \half \\
                                                                     0 & \mbox{if } x\geq \frac{3}{4}
                                                                   \end{cases}.
\end{equation}
Note that $\zeta\in W^{k,p}(\RR_+,w_{\gam+kp})$. Let $t\geq 1$ and using that $1-e^{-a}\geq \half\min\{a,1\}$ for $a\geq 0$, we obtain
\begin{align*}
    \|T_\Dir(t) \zeta\|^p_{L^p(\RR_+, w_{\gam+kp})} & \geq \int_{0}^{\infty} x^{\gam+kp}\Big|\int_0^{\half} \frac{1}{\sqrt{4\pi t}}e^{-\frac{|x-y|^2}{4t}}\left(1-e^{-\frac{xy}{t}}\right)\dd y\Big|^p\dd x \\
     &\geq \int_{0}^{\sqrt{t}} x^{\gam+kp}\Big|\int_0^{\half} \half\frac{1}{\sqrt{4\pi t}}e^{-\frac{|x-y|^2}{4t}}\min\big\{1,\frac{xy}{t}\big\}\dd y\Big|^p\dd x\\
& \stackrel{\mathclap{t\geq 1}}{\geq}\; (4\sqrt{\pi})^{-p}  t^{-\frac{3p}{2}}\int_0^{\sqrt{t}}x^{\gam+(k+1)p} \Big|\int_0^{\half} e^{-\frac{|x-y|^2}{4t}}y\dd y\Big|^p\dd x\\
     & \stackrel{\mathclap{x\mapsto \sqrt{t}x}}{=}\;\;\;(4\sqrt{\pi})^{-p}t^{\frac{\gam+kp-2p+1}{2}}\int_0^{1}x^{\gam+(k+1)p} \Big|\int_0^{\half} e^{-\frac{|x-\frac{y}{\sqrt{t}}|^2}{4}}y\dd y\Big|^p\dd x  \\
     &\geq\; (4\sqrt{\pi})^{-p}t^{\frac{\gam+kp-2p+1}{2}}\int_0^{1}x^{\gam+(k+1)p} \Big|\int_0^{\half} e^{-\frac{1}{4}}y\dd y\Big|^p\dd x \\
     &=c_{p,k,\gam}\;t^{\frac{\gam+kp-2p+1}{2}}.
  \end{align*}
  Therefore, if $\gam+kp>2p-1$, then the semigroup $T_\Dir(t)$ on $W^{k,p}(\RRdh,w_{\gam+kp};X)$ grows at least polynomially and $\|T_\Dir(t)\|\geq 1$ for all $t\geq 0$. This implies the lower bound in \eqref{eq:upperbound} and finishes the proof of Theorem \ref{thm:Neusemi}\ref{it:Neusemi2}.\\

  It remains to prove Theorem \ref{thm:Neusemi}\ref{it:Neusemi3}, i.e., for $\gam+kp>p-1$ the semigroup $T_{\Neu}(t)$ on $W^{k+1,p}(\RRdh, w_{\gam+kp};X)$ grows polynomially and there are constants $c,C>0$ only depending on $p,k,\gam, \eps, d $ and $X$, such that
\begin{equation}\label{eq:upperbound_Neu}
       c\big(1+t^{\frac{\gam+kp-p+1}{2p}}\big)\leq\|T_\Neu(t)\|\leq C \big(1+t^{\frac{\gam+kp-p+1+\eps}{2p}}\big), \qquad t\geq 0,\,\eps>0.
   \end{equation}
   The upper bound in \eqref{eq:upperbound_Neu} follows from  Proposition \ref{prop:sect_est_neu_improved}, Lemma \ref{lem:growth_semigroup_abstract} and the fact that $\delNeu$ generates $T_\Neu$. The lower estimate follows from the same example as in the Dirichlet case above. Indeed, we obtain for $t\geq 1$
   \begin{align*}
    \|T_\Neu(t) \zeta\|^p_{L^p(\RR_+, w_{\gam+kp})} & \geq\; (4\sqrt{\pi})^{-p}  t^{-\frac{p}{2}}\int_0^{\sqrt{t}}x^{\gam+kp} \Big|\int_0^{\half} e^{-\frac{|x-y|^2}{4t}}\dd y\Big|^p\dd x\\
     & \stackrel{\mathclap{x\mapsto \sqrt{t}x}}{=}\;\;\;c_{p,k,\gam}t^{\frac{\gam+kp-p+1}{2}}\int_0^{1}x^{\gam+kp} \Big|\int_0^{\half} e^{-\frac{|x-\frac{y}{\sqrt{t}}|^2}{4}}\dd y\Big|^p\dd x  \\
     &=\tilde{c}_{p,k,\gam}\;t^{\frac{\gam+kp-p+1}{2}}.
  \end{align*}
  This finishes the proof of Theorem \ref{thm:Neusemi}\ref{it:Neusemi3}.
\end{proof}

\begin{remark}
  By inspection of the above proof, it follows that the example for the polynomial growth fails if the weight $w_{\gam+kp}$ is replaced by a weight that behaves as $w_{\gam+kp}$ near zero but becomes constant as $x_1\to\infty$. Using such a weight we expect that the corresponding heat semigroup is actually bounded. Similarly, on bounded domains the semigroup will be bounded as well.
\end{remark}

We end this section with an example that shows that the range for $\gam$ in Theorems \ref{thm:sect_calculus} and \ref{thm:sect_calculus_neu} is optimal in the sense that the theorems fail for $\gam\geq 2p-1$. To this end, recall that
\begin{equation}\label{eq:int_log_finite}
  \int_0^{\half}y^{\alpha}|\log(y)|^{\beta}\dd y<\infty
\end{equation}
if and only if $\alpha>-1$ and $\beta\in\RR$,  or, $\alpha=-1$ and $\beta<-1$.

\begin{example}\label{ex:no_sect}
  Let $p\in(1,\infty), k\in\NN_0$ and $\gam\geq 2p-1$. We provide an example of a function $h\in W^{k,p}(\RR^d_+,w_{\gam+kp};X)$ such that $T_\Dir(t)h\notin W^{k,p}(\RR^d_+,w_{\gam+kp};X)$ for all $t>0$. By setting $h= f\otimes g$ with $g\in \mc{S}(\RR^{d-1};X)$ if $d\geq 2$, it suffices to find an $f\in W^{k,p}(\RR_+,w_{\gam+kp}) $ such that $T_\Dir(t)f\notin W^{k,p}(\RR_+,w_{\gam+kp})$.

 Let $\zeta$ be as in \eqref{eq:cut-offzeta} and define
  \begin{equation*}
    f(x):=x^{-2}|\log(x)|^{-\frac{p+1}{2p}}\zeta(x),\qquad x\in\RR_+.
  \end{equation*}
  A straightforward computation shows that for $j\in\{0,\dots, k\}$ we have $f^{(j)}\in L^p(\RR_+,w_{\gam+kp})$ and thus $f\in W^{k,p}(\RR_+,w_{\gam+kp})$. Note that the heat kernel satisfies $H_t^{1,-}\geq 0$, so that for $t>0$ and $x\in(0,\half)$ we obtain
\begin{align*}
  T_\Dir(t)f(x) &\geq
    \frac{1}{\sqrt{4\pi t}}\int_0^{\half} e^{-\frac{|x-y|^2}{4t}}\big(1-e^{-\frac{xy}{t}}\big) y^{-2}|\log(y)|^{-\frac{p+1}{2p}}\dd y \\
   & \geq \frac{1}{\sqrt{4\pi t}} c_{t,x}\int_0^{\half}  y^{-1}|\log(y)|^{-\frac{p+1}{2p}}\dd y \stackrel{\eqref{eq:int_log_finite}}{=}\infty,
\end{align*}
where
\begin{equation*}
  c_{t,x}:=\inf_{y\in(0,\half)}y^{-1}e^{-\frac{|x-y|^2}{4t}}\left(1-e^{-\frac{xy}{t}}\right)>0,\qquad \text{for } x\in\big(0,\tfrac{1}{2}\big) \text{ and }t>0.
\end{equation*}
Therefore, $T_\Dir(t) f(x)=\infty$ on $(0,\half)$ and in particular $T_\Dir(t)f\notin W^{k,p}(\RR_+,w_{\gam+kp})$.

Analogously, for the Neumann Laplacian the example $f(x):=x^{-1}|\log(x)|^{-\frac{p+1}{2p}}\zeta(x)$ shows that the range for $\gam$ in Theorem \ref{thm:sect_calculus_neu} is optimal.
\end{example}

\section{The \texorpdfstring{$\Hinf$-calculus}{holomorphic functional calculus} for the Dirichlet and Neumann Laplacian}\label{sec:calculus}
In this section, we prove the following theorems concerning the $\Hinf$-calculus.

\begin{theorem}[$\Hinf$-calculus for $\lambda-\delDir$]\label{thm:calc:Dir}
  Let $p\in(1,\infty)$, $k\in\NN_0$, $\gam\in (-1,2p-1)\setminus\{p-1\}$ and let $X$ be a $\UMD$ Banach space. Let $\delDir$ on $W^{k,p}(\RRdh,w_{\gam+kp};X)$ be as in Definition \ref{def:delRRdh}.
 Assume that either
   \begin{enumerate}[(i)]
    \item \label{it:calcDir1} $\gam+kp\in (-1,2p-1)$ and $\lambda\geq0$, or,
    \item \label{it:calcDir2}  $\gam+kp>2p-1$ and $\lambda>0$.
  \end{enumerate}
   Then $\lambda-\delDir$ has a bounded $\Hinf$-calculus of angle $\om_{\Hinf}(\lambda-\delDir)=0$.
\end{theorem}

\begin{theorem}[$\Hinf$-calculus for $\lambda-\delNeu$]\label{thm:calc:Neu}
  Let $p\in(1,\infty)$, $k\in\NN_0\cup\{-1\}$, $\gam\in (-1,2p-1)\setminus\{p-1\}$ and let $X$ be a $\UMD$ Banach space. Let $\delNeu$ on $W^{k+1,p}(\RRdh,w_{\gam+kp};X)$ be as in Definition \ref{def:delRRdh}.
 Assume that either
   \begin{enumerate}[(i)]
    \item  $\gam+kp\in (-1,p-1)$ and $\lambda\geq0$, or,
    \item $\gam+kp>p-1$ and $\lambda>0$.
  \end{enumerate}
  Then $\lambda-\delNeu$ has a bounded $\Hinf$-calculus of angle $\om_{\Hinf}(\lambda-\delNeu)=0$.
\end{theorem}

In Sections \ref{subsec:prelim_calcDir} and \ref{subsec:prelim_calcNeu} below we prove some preliminary estimates on the Dirichlet and Neumann resolvent, respectively. In Section \ref{subsec:proofscalc} we prove Theorems \ref{thm:calc:Dir} and \ref{thm:calc:Neu}.

\subsection{Preliminary estimates for the Dirichlet resolvent}\label{subsec:prelim_calcDir}
We start with a preliminary lemma on commutators of the Dirichlet resolvent and derivatives.
\begin{lemma}\label{lem:comm}
   Let $p\in(1,\infty)$, $k\in \NN_0$, $\gam\in (-1,2p-1)\setminus\{p-1\}$ and let $X$ be a Banach space. Let $\delDir$ on $W^{k,p}(\RRdh, w_{\gam+kp};X)$ be as in  Definition \ref{def:delRRdh}.
Then for all $u\in \Cc^{\infty}(\RRdh;X)$ and $z\in\rho(\delDir)$ we have
   \begin{enumerate}[(i)]
   \item \label{it:comm1} \makebox[8.7cm][l]{$ [ \d_j^n, R(z,\delDir)]u=0$,}\quad  $n\in \NN_1,\,j\in\{2,\dots,d\}$,
  \item \label{it:comm2}\makebox[8.7cm][l]{$[\d_1^2, R(z,\delDir)]u=0$,}
  \item \label{it:comm3} \makebox[8.7cm][l]{$
    [M\d_1^{\ell}, R(z,\delDir)]u= -2R(z,\delDir)\d_1^{\ell+1}R(z,\delDir)u$,}\quad   $ \ell\in \{0,1\} $.
   \end{enumerate}
\end{lemma}
It should be noted that $R(z,\delDir)$ and $\d_1$ do not commute since $\d_1 R(z,\delDir)u$ does not satisfy the Dirichlet boundary condition. This causes various complications in the proof of Theorem \ref{thm:calc:Dir}.
\begin{proof}
  Let $j\in\{2,\dots,d\}$, then using that $\delDir$ and $\d_j$ commute gives
  \begin{align*}
  (z-\delDir)[\d_j, R(z,\delDir)]u&=(z-\delDir)\d_j R(z,\delDir)u-\d_j u\\
  &=\d_j(z-\delDir) R(z,\delDir)u-\d_ju=0,
\end{align*}
  since $\d_j R(z,\delDir)u\in D(\delDir)$ by Lemma \ref{lem:Schwartz}. This proves \ref{it:comm1} for $n=1$ and the general case follows by iteration. Note that $\d_1^2=\delDir-\sum_{j=2}^d\d_j^2$, so \ref{it:comm2} follows from \ref{it:comm1} and the fact that $R(z,\delDir)$ and $\delDir$ commute.

Finally, using \eqref{eq:comm_Md_del} and $M\d_1^{\ell}R(z,\delDir)u\in D(\delDir)$, we compute
\begin{align*}
  (z-\delDir)&[M\d_1^{\ell},R(z,\delDir)]u \\
  & =  zM\d_1^{\ell}R(z,\delDir)u -\delDir M\d_1^{\ell}R(z,\delDir)u -M\d_1^{\ell}u \\\
   &=   zM\d_1^{\ell}R(z,\delDir)u-M\del\d_1^{\ell}R(z,\delDir)u-2\d_1^{\ell+1} R(z,\delDir)u-M\d_1^{\ell}u\\
   &=   M\d_1^{\ell}(z-\delDir)R(z,\delDir)u-2\d_1^{\ell+1} R(z,\delDir)u-M\d_1^{\ell}u\\
   &= -2\d_1^{\ell+1} R(z,\delDir) u,
\end{align*}
which proves \ref{it:comm3}.
\end{proof}

In the following two lemmas, we prove estimates for the commutators in Lemma \ref{lem:comm}\ref{it:comm3}, which are required in the proof of Theorem \ref{thm:calc:Dir}. We start with an estimate for the commutator in Lemma \ref{lem:comm}\ref{it:comm3} with $\ell=0$.
\begin{lemma}\label{lem:comm_est1}
 Let $p\in(1,\infty)$, $k\in \NN_0$, $\gam\in (-1,2p-1)\setminus\{p-1\}$, $\lambda>0$ and let $X$ be a $\UMD$ Banach space. Let $\delDir$ on $W^{k,p}(\RRdh, w_{\gam+kp};X)$ be as in  Definition \ref{def:delRRdh}. Let $0<\nu<\om<\pi$ and let $\Gam_{\nu}$ be the downwards orientated boundary of $\Sigma_{\nu}\setminus B(0, \frac{\lambda}{2})$. Then there exists a $C>0$ such that
  \begin{align*}
    \Big\|\int_{\Gam_\nu}f(z)&R(z,\lambda-\delDir)\d_1R(z,\lambda-\delDir)u \dd z\Big\|_{W^{k,p}(\RRdh, w_{\gam+k p};X)}\\
      \leq &\;C \|f\|_{\Hinf(\Sigma_{\om})}\|u\|_{W^{k,p}(\RRdh, w_{\gam+k p};X)},
  \end{align*}
  for all $f\in H^1(\Sigma_{\om})\cap\Hinf(\Sigma_\om)$ and $u\in W^{k,p}(\RRdh, w_{\gam+k p};X)$. Moreover, the constant $C$ only depends on $p, k, \gam, \lambda, \om, \nu, d $ and $X$.
\end{lemma}
\begin{proof}
From \eqref{eq:resol_est_A} in the proof of Theorem \ref{thm:sect:Dir} we obtain for any $\sigma\in (0,\pi)$
\begin{equation*}\label{eq:resol_est}
  \sup_{z\in \CC\setminus\overline{\Sigma_\sigma}}|z|\|R(z,\lambda-\delDir)\|<\infty\quad \text{ and } \quad\sup_{z\in \CC\setminus\overline{\Sigma_\sigma}}|z|^{\half}\|\d_1R(z,\lambda-\delDir)\|<\infty.
\end{equation*}
Therefore,
\begin{align*}
  \Big\|\int_{\Gam_\nu}& f(z)R(z,\lambda-\delDir)\d_1R(z,\lambda-\delDir)u\dd z\Big\|_{W^{k,p}(\RRdh, w_{\gam+kp};X)}\\
   & \leq \|f\|_{\Hinf(\Sigma_\om)}\int_{\Gam_\nu}\|zR(z,\lambda-\delDir)\|\|z^{\half}\d_1 R(z,\lambda-\delDir)\|\|u\|_{W^{k,p}(\RRdh, w_{\gam+kp};X)} \frac{|\mathrm{d}z|}{|z|^{\frac{3}{2}}}\\
    &\leq C \|f\|_{\Hinf(\Sigma_\om)}\|u\|_{W^{k,p}(\RRdh, w_{\gam+kp};X)}.\qedhere
\end{align*}
\end{proof}
To continue, we prove an estimate for the commutator in Lemma \ref{lem:comm}\ref{it:comm3} with $\ell=1$ under the additional assumption that $\lambda-\delDir$ has a bounded $\Hinf$-calculus.
\begin{lemma}\label{lem:comm_est2}
  Let $p\in(1,\infty)$, $k\in \NN_0$, $\gam\in (-1,2p-1)\setminus\{p-1\}$, $\lambda>0$ and let $X$ be a $\UMD$ Banach space. Let $\delDir$ on $W^{k,p}(\RRdh, w_{\gam+kp};X)$ be as in  Definition \ref{def:delRRdh}. Let $0<\nu<\om<\sigma<\pi$ and let $\Gam_{\nu}$ be the downwards orientated boundary of $\Sigma_{\nu}\setminus B(0, \frac{\lambda}{2})$. Assume that $\lambda-\delDir$ has a bounded $\Hinf(\Sigma_\om)$-calculus on $W^{k,p}(\RRdh, w_{\gam+kp};X)$. Then there exists a $C>0$ such that
   \begin{align*}
        \Big\|\int_{\Gam_\nu}f(z)R^2(z,\lambda-\delDir)\d_1^2u \dd z\Big\|_{W^{k,p}(\RRdh, w_{\gam+k p};X)} &\leq C \|f\|_{\Hinf(\Sigma_{\sigma})}\|u\|_{W^{k,p}(\RRdh, w_{\gam+k p};X)},
  \end{align*}
  for all $f\in H^1(\Sigma_{\sigma})\cap\Hinf(\Sigma_{\sigma})$ and $u\in W^{k,p}(\RRdh, w_{\gam+k p};X)$. Moreover, the constant $C$ only depends on $p,k,\gam,\lambda, \om, \nu,\sigma, d $ and $X$. In addition, if $k=0$, then the constant $C$ can be taken independent of $\lambda$.
\end{lemma}
\begin{proof}
  Let $g(z)=zf'(z)$. We first claim that $g\in \Hinf(\Sigma_{\om})$. Indeed, fix a $\sigma'\in (\om,\sigma)$, then by Cauchy's differentiation formula and the substitution $s\mapsto s|z|$
    \begin{align*}
  \sup_{z\in \Sigma_{\om}}|zf'(z)| &\leq \sup_{z\in \Sigma_{\om}}\Big|\int_{\d\Sigma_{\sigma'}}\frac{zf(s)}{(s-z)^2}\dd s \Big| \\
   & \leq \|f\|_{\Hinf(\Sigma_{\sigma})}\sup_{z\in \Sigma_{\om}} \int_{\d\Sigma_{\sigma'}}\Big|\frac{s}{|z|}-\frac{z}{|z|}\Big|^{-2}|z|^{-1}|\mathrm{d}s|\\
   &\leq\|f\|_{\Hinf(\Sigma_{\sigma})}\sup_{z\in \Sigma_{\om}} \int_{\d\Sigma_{\sigma'}}\Big|s-\frac{z}{|z|}\Big|^{-2}|\mathrm{d}s|\\
   &\leq C_{\om, \sigma}\|f\|_{\Hinf(\Sigma_{\sigma})}.
\end{align*}
Secondly, we claim that for $A_\Dir:=\lambda-\delDir$ and $v\in \mc{S}(\RRdh;X)$ with $v(0,\cdot)=0$, we have
\begin{equation*}
  \frac{1}{2\pi \ii} \int_{\Gam_\nu}f(z)A_{\Dir}R^2(z,A_{\Dir})v \dd z = g(A_{\Dir})v,
\end{equation*}
where $g(A_\Dir)$ is defined through the extended $\Hinf$-calculus, see \cite[Section 15.1]{HNVW24}. Define the regulariser
\begin{equation*}
  \zeta_n(z):=\frac{n}{n+z}-\frac{1}{1+nz}\qquad \text{ for } z\in \Sigma_{\om}\text{ and }n\in \NN_1,
\end{equation*}
and for $\nu'\in (0,\nu)$ define the downwards orientated contour $\Gam_{\nu'}=\d(\Sigma_{\nu'}\setminus B(0,\frac{3\lambda}{4}))$. Then using the dominated convergence theorem, the functional calculus for $A_{\Dir}R^2(z,A_{\Dir})$, Fubini's theorem and Cauchy's differentiation formula, we obtain
\begin{align*}
    \frac{1}{2\pi \ii} \int_{\Gam_\nu}f(z)A_{\Dir}R^2(z,A_{\Dir})v \dd z=&\; \lim_{n\to\infty} \frac{1}{2\pi \ii}\int_{\Gam_\nu}\zeta_n(z)f(z)A_{\Dir}R^2(z,A_{\Dir})v \dd z\\
    =&\; \lim_{n\to\infty} \frac{1}{(2\pi \ii)^2}\int_{\Gam_\nu}\int_{\Gam_{\nu'}}\frac{\zeta_n(z)f(z)}{(z-s)^2}sR(s,A_{\Dir})v\dd s \dd z\\
    =&\; \lim_{n\to \infty}\frac{1}{2\pi \ii}\int_{\Gam_{\nu'}}(\zeta_nf)'(s)sR(s,A_{\Dir})v\dd s\\
    =&\; \lim_{n\to \infty}\frac{1}{2\pi \ii}\int_{\Gam_{\nu'}}\zeta_n(s)f'(s)sR(s,A_{\Dir})v\dd s\\
     &\;+\lim_{n\to\infty}\frac{1}{2\pi \ii} \int_{\Gam_{\nu'}}\zeta_n'(s)f(s)sR(s,A_{\Dir})v\dd s\\
    =&\; g(A_{\Dir})v,
\end{align*}
where in the last step we used for the first term \cite[Theorem 10.2.13]{HNVW17} and for the second term we used that $|s\zeta'_n(s)|$ is uniformly bounded on $\Gam_{\nu'}$, the dominated convergence theorem and $\zeta'_n\to 0$ pointwise as $n\to\infty$. This proves the second claim.

Let $u\in \Cc^{\infty}(\RRdh;X)$. Then by invertibility of $A_{\Dir}$, introducing an additional $A_{\Dir}A_{\Dir}^{-1}$ and using the claims, we obtain
\begin{equation}\label{eq:estR^2}
  \begin{aligned}
        \Big\|\int_{\Gam_\nu}f(z)&R^2(z,A_{\Dir})A_{\Dir}A_{\Dir}^{-1}\d_1^2u \dd z\Big\|_{W^{k,p}(\RRdh, w_{\gam+k p};X)} \\
         &=\Big\|\int_{\Gam_\nu}f(z)A_{\Dir}R^2(z,A_{\Dir})A_{\Dir}^{-1}\d_1^2u \dd z\Big\|_{W^{k,p}(\RRdh, w_{\gam+k p};X)}\\
         &=2\pi \|g(A_{\Dir})A_{\Dir}^{-1}\d_1^2u\|_{W^{k,p}(\RRdh, w_{\gam+k p};X)}\\
         &\leq C \|g\|_{\Hinf(\Sigma_{\om})}\|A_{\Dir}^{-1}\d_1^2u\|_{W^{k,p}(\RRdh, w_{\gam+k p};X)}\\
         &\leq C_{\om,\sigma}\|f\|_{\Hinf(\Sigma_{\sigma})}\|A_{\Dir}^{-1}\d_1^2u\|_{W^{k,p}(\RRdh, w_{\gam+k p};X)},
  \end{aligned}
\end{equation}
  where in the penultimate step we used the $\Hinf(\Sigma_\om)$-calculus of $A_{\Dir}$.
  From Lemma \ref{lem:comm}\ref{it:comm2} and Proposition \ref{prop:sect_est} we find
  \begin{equation}\label{eq:est_resol}
  \begin{aligned}
    \|A_{\Dir}^{-1}\d_1^2u\|_{W^{k,p}(\RRdh, w_{\gam+k p};X)}
    &=\|\d_1^2R(\lambda, \delDir)u\|_{W^{k,p}(\RRdh, w_{\gam+k p};X)}\\
    &\leq C g_{k,\gam}(\lambda)\|u\|_{W^{k,p}(\RRdh, w_{\gam+k p};X)},
    \end{aligned}
  \end{equation}
  where $C$ is independent of $\lambda$ and $g_{k,\gam}$ is as in \eqref{eq:K_IH}.
  By density (Lemma \ref{lem:density}) the desired estimate follows.

 Finally, assume that $k=0$. Then we obtain \eqref{eq:estR^2} with a constant independent of $\lambda$. Indeed, this follows from using the $\Hinf$-calculus on $L^p(\RRdh, w_{\gam};X)$, see Theorem \ref{thm:LVresult} and \cite[Proposition 16.2.6]{HNVW24}. Moreover, \eqref{eq:est_resol} also holds with a constant independent of $\lambda$ by \eqref{eq:K_IH}. This finishes the proof.
\end{proof}

\subsection{Preliminary estimates for the Neumann resolvent}\label{subsec:prelim_calcNeu}
In this section, we provide some preliminary estimates required for the proof of the $\Hinf$-calculus for the Neumann Laplacian. This proof will be given in the next section, relying on the $\Hinf$-calculus for the Dirichlet Laplacian on $W^{k,p}(\RRdh, w_{\gam+kp};X)$ with $k\in\NN_1$ and $\gam\in(-1,2p-1)\setminus\{p-1\}$.

We start with a lemma on commutators for the Neumann Laplacian.
\begin{lemma}\label{lem:commNeu}
   Let $p\in(1,\infty)$, $k\in \NN_0\cup\{-1\}$, $\gam\in (-1,2p-1)\setminus\{p-1\}$ and let $X$ be a Banach space. Let $\delNeu$ on $W^{k+1,p}(\RRdh, w_{\gam+kp};X)$ be as in Definition \ref{def:delRRdh}.
Then for all $u\in C_{{\rm c},1}^{\infty}(\overline{\RRdh};X)$ and $z\in\rho(\delNeu)$, we have
   \begin{enumerate}[(i)]
   \item \label{it:commNeu1} \makebox[6cm][l]{$ [ \d_j^n, R(z,\delNeu)]u=0$,}\quad  $n\in \NN_1,\,j\in\{2,\dots,d\}$,
  \item \label{it:commNeu2}\makebox[6cm][l]{$[\d_1^2, R(z,\delNeu)]u=0$,}
   \end{enumerate}
   In addition, let $k\in\NN_0$, $\delDir$ on $W^{k,p}(\RRdh, w_{\gam+kp};X)$. Then, for all  $u\in C_{{\rm c},1}^{\infty}(\overline{\RRdh};X)$ and $z\in\rho(\delNeu)\cap\rho(\delDir)$, we have
   \begin{enumerate}[resume*]
   \item \label{it:commNeu4} \makebox[8.2cm][l]{$\d_1R(z,\delNeu)u=R(z,\delDir)\d_1 u$.}
   \end{enumerate}
\end{lemma}
\begin{proof}
  The statements \ref{it:commNeu1} and \ref{it:commNeu2} follow similarly as in Lemma \ref{lem:comm} using Lemma \ref{lem:Schwartz_neu}.

  For \ref{it:commNeu4}, set $\tilde{u}:=R(z,\delNeu)u$, then $\tilde{u}\in \mc{S}(\RRdh;X)$ with $(\d_1\tilde{u})(0,\cdot)=0$ and $z\tilde{u}-\delNeu \tilde{u}=u$ by Lemma \ref{lem:Schwartz_neu}. Moreover, it holds that $(z-\delDir)\d_1 \tilde{u}=\d_1(z-\delNeu)\tilde{u} = \d_1 u$ and applying the Dirichlet resolvent gives the result.
\end{proof}

We continue with an estimate for the commutator of $M$ and the Neumann resolvent.
\begin{lemma}\label{lem:comm_est1_Neumann}
 Let $p\in(1,\infty)$, $k\in \NN_0\cup\{-1\}$, $\gam\in (p-1,2p-1)$, $\lambda>0$ and let $X$ be a $\UMD$ Banach space. Let $\delNeu$ on $W^{k+1,p}(\RRdh, w_{\gam+kp};X)$ and $\delDir$ on $W^{k+1,p}(\RRdh, w_{\gam+kp};X)$ be as in Definition \ref{def:delRRdh}. Then, for all $u\in \Cc^{\infty}(\RRdh;X)$ and $z\in\rho(\delNeu)\cap\rho(\delDir)$, we have
 \begin{equation*}
   MR(z,\delNeu)u= R(z, \delDir)Mu -2R(z,\delDir)\d_1R(z,\delNeu)u.
 \end{equation*}
 Moreover, there exists a $C>0$ such that
 \begin{align*}
 \|Mf(\lambda-\delNeu)u\|_{W^{k+1,p}(\RRdh, w_{\gam+kp};X)}\leq &\; \|f(\lambda-\delDir)Mu\|_{W^{k+1,p}(\RRdh, w_{\gam+kp};X)}\\
 &+ C\|f\|_{\Hinf(\Sigma_{\om})}\|u\|_{W^{k+1,p}(\RRdh, w_{\gam+k p};X)},
 \end{align*}
 for all $f\in H^1(\Sigma_{\om})\cap\Hinf(\Sigma_\om)$  with $\om\in(0,\pi)$. Moreover, the constant $C$ only depends on $p,k,\gam,\lambda, \om, \nu, d $ and $X$.
\end{lemma}
\begin{proof}
Note that $R(z,\delDir )Mu\in D(\delDir)$ by Lemma \ref{lem:Schwartz} and $MR(z,\delNeu)u\in D(\delDir)$ by Lemma \ref{lem:Schwartz_neu}. Therefore, arguing as in Lemma \ref{lem:comm}\ref{it:comm3}, we obtain
\begin{align*}
(z-\delDir)\big[MR(z,\delNeu)u - R(z,\delDir ) Mu\big]=-2\d_1R(z,\delNeu)u.
\end{align*}
To prove the estimate, let $0<\nu<\om<\pi$ and $\Gam_{\nu}$ be the downwards orientated boundary of $\Sigma_{\nu}\setminus B(0, \frac{\lambda}{2})$. Then
\begin{align*}
  \|Mf(&\lambda-\delNeu)u\|_{W^{k+1,p}(\RRdh, w_{\gam+kp};X)}   \\
    \leq &\;\|f(\lambda-\delDir) M  u \|_{W^{k+1,p}(\RRdh, w_{\gam+kp};X)} \\
   &+ 2 \Big\|\int_{\Gam_{\nu}} f(z)R(z,\lambda-\delDir)\d_1R(z,\lambda-\delNeu)u \dd z\Big\|_{W^{k+1,p}(\RRdh, w_{\gam-p+(k+1)p};X)}
\end{align*}
and the latter term can be estimated similarly as in the proof of Lemma \ref{lem:comm_est1} using Theorem \ref{thm:sect:Dir} with $k+1$ and $\gam-p\in(-1,p-1)$, and Theorem \ref{thm:sect:Neu}.
\end{proof}

\subsection{Proofs of Theorems \ref{thm:calc:Dir} and \ref{thm:calc:Neu}}\label{subsec:proofscalc}
Before turning to the proof of Theorem \ref{thm:calc:Dir}, we prove an estimate for the derivatives in directions $2,\dots, d$ of the functional calculus on the spaces $W^{k,p}(\RRdh, w_{\gam+kp};X)$. For this we will use a perturbation argument in the parameter $\gam$ inspired by \cite[Proposition 5.3]{LV18}.
\begin{lemma}\label{lem:tilde_der_f(A)}
Let $p\in(1,\infty)$, $k\in \NN_0$ and let $X$ be a $\UMD$ Banach space. Let $\lambda>0$ and consider either of the following cases:
\begin{enumerate}[(i)]
  \item \label{it:1} $\gam\in (-1,p-1)$ and $A:=\lambda-\delDir$, or,
  \item \label{it:2} $\gam\in (p-1,2p-1)$ and $A:=\lambda-\delDir$, or,
  \item \label{it:3} $\gam\in (-1,p-1)$ and $A:=\lambda-\delNeu$,
\end{enumerate}
where $\delDir$ and $\delNeu$ on $W^{k,p}(\RRdh, w_{\gam+k p};X)$ are as in Definition \ref{def:delRRdh}.
Let $\om\in (0,\frac{\pi}{2})$ and $\alpha=(0,\tilde{\alpha})\in\NN_0\times\NN_0^{d-1}$ with $|\alpha|=k$. Then for all $f\in H^1(\Sigma_{\om})\cap\Hinf(\Sigma_{\om})$ and $u\in \Cc^{\infty}(\RRdh;X)$ we have
 \begin{align*}
   \|\d^{\tilde{\alpha}}f(A)u\|_{L^p(\RRdh, w_{\gam+kp};X)}\leq&\; C \|f\|_{\Hinf(\Sigma_{\om})}\|u\|_{W^{k,p}(\RRdh, w_{\gam+kp};X)},
 \end{align*}
where the constant $C$ only depends on $p,k,\gam,  \om, \alpha,d $ and $X$.
\end{lemma}
\begin{proof} Throughout the proof, $C$ denotes a constant only depending on $p,k,\gam, \om,\alpha, d$ and $X$, which may change from line to line. We proceed by induction on $k\geq 0$. The case $k=0$ follows from Theorems \ref{thm:LVresult} and \ref{thm:LVresult_Neumann}. Let $k\in\NN_1$ be fixed. Assume that the statement of the lemma holds with $k$ replaced by $\ell$ for all $\ell\in \{0,\dots, k-1\}$. We prove the estimate as stated in the lemma.

Let $\alpha=(0,\tilde{\alpha})\in\NN_0\times \NN_0^{d-1}$ be such that $|\alpha|=k$. For any $\theta\in (0,1)$ set $g:=M^{(k+1)\theta}u$ and $\psi(x_1,y_1):=\big(\frac{x^\theta_1}{y^\theta_1}-1\big)^{k+1}$. Then for $x\in\RRdh$ we have
\begin{equation}\label{eq:eq_lemtilde}
    \begin{aligned}
f(A)u(x) =&\; x_1^{-(k+1)\theta}f(A)\big((x_1^\theta-M^\theta)^{k+1}u\big)(x)\\
&\;+x_1^{-(k+1)\theta}\sum_{j=0}^k c_{j,k}f(A)
\big(x_1^{j\theta}M^{(k+1-j)\theta}u\big)(x)\\
=&\; x_1^{-(k+1)\theta}f(A)(\psi(x_1,\cdot)g )(x)\\
& \;+ \sum_{j=0}^k c_{j,k} x_1^{-(k+1-j)\theta} f(A)(M^{(k+1-j)\theta}u)(x),
\end{aligned}
\end{equation}
where $c_{j,k} = (-1)^{k+1-j}\binom{k+1}{j}$. For the latter sum, we find
\begin{equation}\label{eq:calc_rest}
  \begin{aligned}
  \Big\|x\mapsto \d^{\tilde{\alpha}} \sum_{j=0}^k c_{j,k}& x_1^{-(k+1-j)\theta}f(A)(M^{(k+1-j)\theta}u)(x)\Big\|_{L^p(\RRdh, w_{\gam+kp};X)}\\
  \leq &\; C \sum_{j=0}^k\|\d^{\tilde{\alpha}} f(A)M^{(k+1-j)\theta}u\|_{L^p(\RRdh,w_{\gam_j};X)},
\end{aligned}
\end{equation}
where $\gam_j:=\gam+kp-(k+1-j)\theta p$. Now choose $\theta\in (0,1)$ such that
\begin{equation}\label{eq:gamj}
  \begin{aligned}
  \gam_0&\in (-1,p-1)&\text{and}&\quad\gam_j\in ((j-1)p-1,jp-1) &&\quad\text{ if }\gam\in (-1,p-1),\\
\gam_0&\in (p-1,2p-1)&\text{and}&\quad\gam_j\in (jp-1,(j+1)p-1) && \quad\text{ if }\gam\in (p-1,2p-1),
\end{aligned}
\end{equation}
The conditions $\gam_0,\gam_1\in (-1,p-1)$ and $\gam_0,\gam_1\in (p-1,2p-1)$, respectively, lead to
\begin{equation*}
  \begin{aligned}
\theta & \in  \big(\tfrac{\gam +(k-1)p+1}{kp}, \tfrac{\gam+kp+1}{kp}\big)\cap\big(\tfrac{\gam +(k-1)p+1}{(k+1)p}, \tfrac{\gam+kp+1}{(k+1)p}\big)\quad &&\text{ if }\gam\in (-1,p-1),\\
\theta  & \in \big(\tfrac{\gam +(k-2)p+1}{kp}, \tfrac{\gam+(k-1)p+1}{kp}\big)\cap\big(\tfrac{\gam +(k-2)p+1}{(k+1)p}, \tfrac{\gam+(k-1)p+1}{(k+1)p}\big)\quad&& \text{ if }\gam\in (p-1,2p-1).
\end{aligned}
\end{equation*}
Therefore, we can take
\begin{equation}\label{eq:theta}
  \begin{aligned}
\theta & \in  \big(\tfrac{\gam +(k-1)p+1}{kp}, \tfrac{\gam+kp+1}{(k+1)p}\big)\quad &&\text{ if }\gam\in (-1,p-1),\\
\theta  & \in \big(\tfrac{\gam +(k-2)p+1}{kp}, \tfrac{\gam+(k-1)p+1}{(k+1)p}\big)\quad&& \text{ if }\gam\in (p-1,2p-1),
\end{aligned}
\end{equation}
and it is straightforward to verify that these intervals are non-empty for the given ranges of $\gam$. Moreover, if \eqref{eq:theta} holds, then the other conditions on $\gam_j$ for $j\in \{2,\dots, k\}$ in \eqref{eq:gamj} are automatically satisfied.

Write $\tilde{\alpha}=\tilde{\beta}+\tilde{\delta}$ for some $\tilde{\beta},\tilde{\delta}\in\NN_0^{d-1}$ such that $|\tilde{\beta}|=k-(j-1)$ and $|\tilde{\delta}|=j-1$. Using that derivatives on $\RR^{d-1}$ commute with the resolvent (Lemmas \ref{lem:comm}\ref{it:comm1} and \ref{lem:commNeu}\ref{it:commNeu1}) and that the resolvents are consistent (Lemmas \ref{lem:consistent_resol_RRdh} and \ref{lem:consistent_resol_RRdh_Neumann}), we can estimate the right-hand side of \eqref{eq:calc_rest} as
\begin{align*}
  \sum_{j=0}^k&\|\d^{\tilde{\alpha}} f(A)M^{(k+1-j)\theta}u\|_{L^p(\RRdh,w_{\gam_j};X)}\\
   & \leq \| f(A) \d^{\tilde{\alpha}} M^{(k+1)\theta}u\|_{L^p(\RRdh,w_{\gam_0};X)} +\sum_{j=1}^k\|\d^{\tilde{\delta}}f(A)\d^{\tilde{\beta}} M^{(k+1-j)\theta}u\|_{L^p(\RRdh, w_{\gam_j};X)} \\
   & \leq C \|f\|_{\Hinf(\Sigma_\om)} \Big(\|\d^{\tilde{\alpha}} M^{(k+1)\theta}u\|_{L^p(\RRdh,w_{\gam_0};X)} + \sum_{j=1}^k \|\d^{\beta} M^{(k+1-j)\theta}u\|_{W^{j-1,p}(\RRdh, w_{\gam_j};X)}\Big)\\
   &\leq C \|f\|_{\Hinf(\Sigma_\om)} \|u\|_{W^{k,p}(\RRdh, w_{\gam+kp};X)},
\end{align*}
using the induction hypothesis in the penultimate step, which is allowed by \eqref{eq:gamj} and \eqref{eq:theta}.
In addition, we have applied Lemma \ref{lem:M_bound_frac} in the last step.

In view of \eqref{eq:eq_lemtilde}, it remains to show the estimate
\begin{equation}\label{eq:est_pert_calc}
  \begin{aligned}
  \|x\mapsto x_1^{-(k+1)\theta} &\d^{\tilde{\alpha}}f(A)(\psi(x_1,\cdot)g )(x)\|_{L^p(\RRdh, w_{\gam+kp};X)}\\
   &\leq C\|f\|_{\Hinf(\Sigma_\om)}\|g\|_{L^p(\RRdh, w_{\gam-(k+1)\theta p};X)}\\
  &\leq C\|f\|_{\Hinf(\Sigma_\om)}\|u\|_{W^{k,p}(\RRdh, w_{\gam+kp};X)},
\end{aligned}
\end{equation}
where the last estimate follows from Hardy's inequality (Corollary \ref{cor:Sob_embRRdh}).

Let $\nu\in (0,\om)$ so that
\begin{equation*}
  f(A)=\frac{1}{2\pi \ii}\int_{\Gam_-\cup \Gam_+}f(z)R(z,A)\dd z,
\end{equation*}
where $\Gam_{\pm}=\{r e^{\pm \ii\nu}: r>0\}$ are both orientated downwards. To prove \eqref{eq:est_pert_calc} we estimate the integrals over $\Gam_\pm$ separately and by symmetry it suffices to consider only $\Gam_+$. We define $\delta=\frac{\pi -\nu}{2}$ and write $z=re^{\ii\nu}$ with $r>0$. Then we have the Laplace transform representation
\begin{equation}\label{eq:resol_laplace}
\begin{aligned}
  R(z,A)&=-(re^{\ii(\nu-\pi)}+A)^{-1}=-e^{\ii\delta}(re^{-\ii\delta}+e^{\ii\delta}A)^{-1}  \\&
  =-e^{\ii\delta}\int_0^\infty e^{-tre^{-\ii\delta}}e^{-t\lambda e^{\ii\delta}}T(te^{\ii\delta})\dd t,
  \end{aligned}
\end{equation}
where $T$ is the Dirichlet or Neumann heat semigroup from \eqref{eq:Tt} and \eqref{eq:Tt_Neu}, respectively. For $x,y\in \RRdh$ and $z\in\CC_+$, recall that $G_z^{d}$ is the standard heat kernel on $\RR^{d}$ (see \eqref{eq:HeatKernelRd}) and the Dirichlet and Neumann heat kernels are given by $H^{d,\pm}_z(x,y):=G^d_z(x_1-y_1,\tilde{x}-\tilde{y})\pm G^d_z(x_1+y_1,\tilde{x}-\tilde{y})$, see \eqref{eq:H}. It holds that $H^{d,\pm}_z(x,y)=H^{1,\pm}_z(x_1,y_1)G^{d-1}_z(\tilde{x}-\tilde{y})$. Moreover, for $|\tilde{\alpha}|=k$ and $z\in \CC_+$ we have the following estimate on derivatives of $G_z^{d-1}$
\begin{equation}\label{eq:est_derG}
  |\d^{\tilde{\alpha}}_{\tilde{x}}G^{d-1}_{z}(\tilde{x}-\tilde{y})|\leq C_{d,k} \Big(|z|^{-\frac{k}{2}}+\frac{|\tilde{x}-\tilde{y}|^k}{|z|^k}\Big) |G^{d-1}_{z}(\tilde{x}-\tilde{y})|,\qquad \tilde{x},\tilde{y}\in \RR^{d-1}.
\end{equation}
Indeed, let $q(s)=e^{-s^2}$ with $s\in\CC$ which satisfies $q^{(n)}(s)=p(s)q(s)$ where $p(s)$ is a polynomial of degree $n\in\NN_0$. Then, $|q^{(n)}(s)|\leq C_n (1+|s|^2)^{\frac{n}{2}}|q(s)|$ and upon noting that $G^1_z(x-y)=(4\pi z)^{-\half}q(\frac{x-y}{2\sqrt{z}})$ for $x,y\in\RR$, we find with the chain rule
\begin{equation*}
  \Big|\frac{\mathrm{d}^n }{\mathrm{d} x^n}G_z^1(x-y)\Big|\leq C_n |z|^{-\frac{n}{2}}\Big(1+\frac{|x-y|^2}{|z|}\Big)^{\frac{n}{2}}|G_z^1(x-y)|,\qquad x,y\in \RR,\; n\in\NN_0.
\end{equation*}
Write $\tilde{\alpha}=(\alpha_2,\dots, \alpha_{d})\in \NN_0^{d-1}$, then from $G_z^{d-1}(\tilde{x}-\tilde{y})=\prod_{j=2}^{d}G_z^1(x_j-y_j)$ and the one dimensional estimate above, we find
\begin{align*}
 |\d^{\tilde{\alpha}}_{\tilde{x}}G^{d-1}_{z}(\tilde{x}-\tilde{y})|&\leq C_{k}\prod_{j=2}^{d}|z|^{-\tilde{\alpha}_j/2}\Big(1+\frac{|x_j-y_j|^2}{|z|}\Big)^{\tilde{\alpha}_j/2}|G_z^1(x_j-y_j)|\\
 &\leq C_{d,k} |z|^{-\frac{k}{2}}\Big(1+\frac{|\tilde{x}-\tilde{y}|^2}{|z|}\Big)^{\frac{k}{2}}|G_z^{d-1}(\tilde{x}-\tilde{y})|\\
 &\leq C_{d,k} \Big(|z|^{-\frac{k}{2}}+\frac{|\tilde{x}-\tilde{y}|^k}{|z|^k}\Big) |G^{d-1}_{z}(\tilde{x}-\tilde{y})|,
\end{align*}
which proves \eqref{eq:est_derG}.

Let $h\in \Cc^{\infty}(\RRdh;X)$. It follows from \eqref{eq:est_derG} that
\begin{equation}\label{eq:der_Tt}
\begin{aligned}
  \|\d^{\tilde{\alpha}}T(te^{\ii\delta})h(x)\|_X&\leq \int_{\RRdh}|H^{1,\pm}_{te^{\ii\delta}}(x_1,y_1)||\d_{\tilde{x}}^{\tilde{\alpha}}G^{d-1}_{te^{\ii\delta}}(\tilde{x}-\tilde{y})|\|h(y)\|_X\dd y\\
  &\leq C_{\delta, d,k} \int_{\RRdh}\Big(t^{-\frac{k}{2}}+\frac{|\tilde{x}-\tilde{y}|^{k}}{t^k} \Big) |H^{1,\pm}_{te^{\ii\delta}}(x_1,y_1)||G^{d-1}_{te^{\ii\delta}}(\tilde{x}-\tilde{y})|\|h(y)\|_X\dd y\\
  &\leq C_{\delta,d,k} \sum_{m\in\{0,1\}}\int_{\RRdh}\frac{|\tilde{x}-\tilde{y}|^{km}}{t^{\frac{(m+1)k}{2}}} H^{d,\pm}_{t\cos^{-1}(\delta)}(x,y)\|h(y)\|_X\dd y,
\end{aligned}
\end{equation}
where it was used that
\begin{align*}
  |H^{1,-}_{te^{\ii\delta}}(x_1,y_1)|&=(4\pi t)^{-\half}\Big|e^{-\frac{e^{-\ii\delta}|x_1-y_1|^2}{4t}}\Big|\Big|1-e^{\frac{-e^{-\ii\delta }x_1y_1}{t}}\Big|\\
  &=(4\pi t)^{-\half}e^{-\frac{\cos(\delta)|x_1-y_1|^2}{4t}}\Big|\int_0^{\frac{x_1y_1}{t}}e^{-e^{-\ii\delta }s}\dd s\Big|\\
  &\leq \frac{1}{\cos(\delta)}\frac{1}{\sqrt{4\pi t}}e^{-\frac{\cos(\delta)|x_1-y_1|^2}{4t}}\big(1-e^{-\frac{\cos(\delta)x_1y_1}{t}}\big)= C_{\delta} H^{1,-}_{t\cos^{-1}(\delta)}(x_1,y_1),\\
  |H^{1,+}_{te^{\ii\delta}}(x_1,y_1)|&= (4\pi t)^{-\half}e^{-\frac{\cos(\delta)|x_1-y_1|^2}{4t}}\big(1+e^{-\frac{\cos(\delta)x_1y_1}{t}}\big)= H^{1,+}_{t\cos^{-1}(\delta)}(x_1,y_1),\\
  |G^{d-1}_{te^{\ii\delta}}(\tilde{x}-\tilde{y})|&=(4\pi t)^{-\frac{(d-1)}{2}}e^{-\frac{\cos(\delta)|\tilde{x}-\tilde{y}|^2}{4t}}= G^{d-1}_{t\cos^{-1}(\delta)}(\tilde{x}-\tilde{y}).
\end{align*}
So by \eqref{eq:resol_laplace}, \eqref{eq:der_Tt} and the substitution $t\mapsto \cos(\delta)t$, we find
\begin{equation}\label{eq:est_Xnorm}
  \begin{aligned}
  &\Big\|x\mapsto \int_{\Gam_+}f(z)\d^{\tilde{\alpha}}R(z,A)(\psi(x_1,\cdot)g)(x)\dd z\Big\|_X\\
  &\leq \|f\|_{\Hinf(\Sigma_{\om})}\int_0^{\infty}\int_0^\infty e^{-tr\cos(\delta)}\|\d^{\tilde{\alpha}}T(te^{\ii\delta})(\psi(x_1,\cdot)g)(x)\|_X\dd t\dd r\\
  &=\frac{1}{\cos(\delta)} \|f\|_{\Hinf(\Sigma_{\om})}\int_0^{\infty}\|\d^{\tilde{\alpha}}T(te^{\ii\delta})(\psi(x_1,\cdot)g)(x)\|_X\ddtt\\
  &\leq C_{\delta,d,k} \|f\|_{\Hinf(\Sigma_{\om})} \sum_{m\in\{0,1\}} \int_0^\infty\int_{\RRdh}\frac{|\tilde{x}-\tilde{y}|^{km}}{t^{\frac{(m+1)k}{2}}}H^{d,\pm}_{t}(x,y)|\psi(x_1,y_1)|\|g(y)\|_X
\dd y \ddtt.
\end{aligned}
\end{equation}
For any $x\in \RRd$ we have
\begin{align*}
  \int_0^\infty t^{-\frac{(m+1)k}{2}}G^d_t(x)\ddtt&=C \int_0^\infty t^{-\frac{(m+1)k}{2}-\frac{d}{2}}e^{-\frac{|x|^2}{4t}}\ddtt\\
  &=C|x|^{-d-(m+1)k}\int_0^\infty s^{\frac{(m+1)k}{2}+\frac{d}{2}}e^{-\frac{s}{4}}\frac{\mathrm{d}s}{s}= C_{d,k}|x|^{-d-(m+1)k}.
\end{align*}
Therefore, after the substitution $\tilde{y}\mapsto\tilde{x}-\tilde{y}$ we obtain for $m\in \{0,1\}$
\begin{align*}
  &\int_0^\infty\int_{\RRdh}\frac{|\tilde{x}-\tilde{y}|^{km}}{t^{\frac{(m+1)k}{2}}}H^{d,\pm}_{t}(x,y)|\psi(x_1,y_1)|\|g(y)\|_X
\dd y \ddtt\\
&=\int_{\RRdh}|\tilde{y}|^{km}\int_0^{\infty}t^{-\frac{(m+1)k}{2}}\big(G^d_t(x_1-y_1,\tilde{y})\pm G^d_t(x_1+y_1,\tilde{y})\big)
\ddtt |\psi(x_1,y_1)|\|g(y_1, \tilde{x}-\tilde{y})\|_X\dd y\\
&= C_{d,k}\int_{\RRdh}|\tilde{y}|^{km}\bigg[\frac{1}{|(\frac{x_1}{y_1}-1,\frac{\tilde{y}}{y_1})|^{d+(m+1)k}} \pm \frac{1}{|(\frac{x_1}{y_1}+1,\frac{\tilde{y}}{y_1})|^{d+(m+1)k}}\bigg]\\
&\qquad\qquad \cdot\Big|\frac{x_1^\theta}{y_1^\theta}-1\Big|^{k+1}\|g(y_1, \tilde{x}-\tilde{y})\|_X\frac{\dd y}{y_1^{d+(m+1)k}}\\
&=\int_{\RRdh}|\tilde{y}|^{km}\ell^{\pm}_m\big(\frac{x_1}{y_1}, \tilde{y}\big)\|g(y_1, \tilde{x}-y_1\tilde{y})\|_X\frac{\dd y}{y_1^{k+1}}=:L_m(x)
\end{align*}
performing the substitution $\tilde{y}\mapsto y_1\tilde{y}$ in the last step, where
\begin{equation*}
  \ell^{\pm}_m(v_1,\tilde{v}):=\bigg[\frac{1}{|(v_1-1,\tilde{v})|^{d+(m+1)k}} \pm \frac{1}{|(v_1+1,\tilde{v})|^{d+(m+1)k}}\bigg]|v_1^\theta-1|^{k+1},\quad (v_1,\tilde{v})\in \RRdh.
\end{equation*}
Recall that $\gam_0=\gam+kp-(k+1)\theta p$ and to show \eqref{eq:est_pert_calc} we estimate the $L^p(\RRdh, w_{\gam_0})$-norm of the left-hand side of \eqref{eq:est_Xnorm}. To this end, it remains to bound $\|L_m\|_{L^p(\RRdh, w_{\gam_0})}$ for $m\in\{0,1\}$.
With Minkowski's inequality and the substitutions $\tilde{x}\mapsto \tilde{x}+y_1\tilde{y}$, $y_1\mapsto x_1y_1^{-1}$ and $x_1\mapsto x_1y_1$, we obtain
\begin{align*}
  \|L_m\|_{L^p(\RRdh, w_{\gam_0})}
  &\lesssim \Big\|x_1\mapsto \int_{\RRdh}|\tilde{y}|^{km}\ell^\pm_m\big(\frac{x_1}{y_1}, \tilde{y}\big)\|g(y_1, \cdot-y_1\tilde{y})\|_{L^p(\RR^{d-1};X)}\frac{\dd y}{y_1^{k+1}}\Big\|_{L^p(\RR_+,w_{\gam_0})}\\
  &= \Big\|x_1\mapsto \int_{\RRdh}|\tilde{y}|^{km}\ell^\pm_m\big(\frac{x_1}{y_1}, \tilde{y}\big)\|g(y_1, \cdot)\|_{L^p(\RR^{d-1};X)}\frac{\dd y}{y_1^{k+1}}\Big\|_{L^p(\RR_+, w_{\gam_0})}\\
  &=\Big\|x_1\mapsto \int_{\RRdh}|\tilde{y}|^{km}\ell^\pm_m\big(y_1, \tilde{y}\big)\big\|g\big(\frac{x_1}{y_1}, \cdot\big)\big\|_{L^p(\RR^{d-1};X)}x_1^{-k}y_1^{k-1}\dd y\Big\|_{L^p(\RR_+, w_{\gam_0})}\\
  &\leq \int_{\RRdh}|\tilde{y}|^{km}\ell^\pm_m\big(y_1, \tilde{y}\big)\Big\|x_1\mapsto x_1^{-k}\big\|g\big(\frac{x_1}{y_1}, \cdot\big)\big\|_{L^p(\RR^{d-1};X)}\Big\|_{L^p(\RR_+, w_{\gam_0})}y_1^{k-1}\dd y\\
  &= \|g\|_{L^p(\RRdh, w_{\gam_0-kp};X)}\int_{\RRdh}|\tilde{y}|^{km}\ell^\pm_m\big(y_1, \tilde{y}\big)y_1^{\frac{\gam+1}{p}-(k+1)\theta+k-1}\dd y,
\end{align*}
which yields \eqref{eq:est_pert_calc} if the latter integral is finite for the given ranges of $\gam$. Indeed, note that for any $a\in\RR$ we have using the substitution $\tilde{y}\mapsto a\tilde{y}$
\begin{equation*}
  \int_{\RR^{d-1}}\frac{|\tilde{y}|^{km}}{|(a,\tilde{y})|^{d+(m+1)k}}\dd\tilde{y} =\frac{1}{|a|^{k+1}}\int_{\RR^{d-1}}\frac{|\tilde{y}|^{km}}{|(1,\tilde{y})|^{d+(m+1)k}}\dd\tilde{y}=\frac{C_{d,k,m}}{|a|^{k+1}},
\end{equation*}
therefore
\begin{align}\label{eq:lasteq}
  \int_{\RRdh}|\tilde{y}|^{km}&\ell^\pm_m\big(y_1, \tilde{y}\big)y_1^{\frac{\gam+1}{p}-(k+1)\theta+k-1}\dd y\nonumber\\
  &= C_{d,k,m} \int_{\RR_+}\bigg[\frac{1}{|y_1-1|^{k+1}}\pm\frac{1}{|y_1+1|^{k+1}}\bigg]|y_1^\theta-1|^{k+1}y_1^{\frac{\gam+1}{p}-(k+1)\theta+k-1}\dd y_1\nonumber\\
  &=:C_{d,k,m}\int_{\RR_+}\ell^\pm(y_1)|y_1^\theta-1|^{k+1}y_1^{\frac{\gam+1}{p}-(k+1)\theta+k-1}\dd y_1.
\end{align}
For $y_1\in(0,\half)$ we have $\ell^{\pm}(y_1)  \leq C$ and $|y_1^\theta-1|^{k+1}\leq C$ since $\theta>0$. Moreover, by \eqref{eq:theta} it follows
\begin{equation*}
  \tfrac{\gam+1}{p}-(k+1)\theta+k-1>-1,
\end{equation*}
so that the integral in \eqref{eq:lasteq} over $y_1\in(0,\half)$ is finite. For $y_1\in (\half,\frac{3}{2})$ it follows by the mean value theorem that there is no singularity at $y_1=1$ and that the integral in \eqref{eq:lasteq} over the interval $(\half,\frac{3}{2})$ is finite. Finally, if $y_1>\frac{3}{2}$, then $(y_1^\theta-1)^{k+1}\leq C y_1^{(k+1)\theta}$ and
\begin{equation*}
  \ell^{\pm}(y_1) =\frac{(y_1+1)^{k+1}\pm(y_1-1)^{k+1}}{(y_1-1)^{k+1}(y_1+1)^{k+1}}\leq C\begin{cases}\frac{1}{y_1^{k+1}} & \mbox{if }A=\lambda-\delNeu\\
                                                                                       \frac{1}{y_1^{k+2}} & \mbox{if }A=\lambda-\delDir 
                                                                                     \end{cases},
\end{equation*}
where we recall that the Dirichlet and Neumann Laplacian correspond to $\ell^-$ and $\ell^+$, respectively.
Therefore, using that $\gam\in(-1,2p-1)\setminus\{p-1\}$ for the Dirichlet Laplacian, we obtain
\begin{equation*}
\int_{\frac{3}{2}}^\infty\ell^{-}(y_1)|y_1^\theta-1|^{k+1}y_1^{\frac{\gam+1}{p}-(k+1)\theta+k-1}\dd y_1\leq C
  \int_{\frac{3}{2}}^\infty y_1^{\frac{\gam+1}{p}-3}\dd y_1<\infty.
\end{equation*}
Similarly, using that $\gam\in(-1,p-1)$ for the Neumann Laplacian, we obtain
\begin{equation*}
  \int_{\frac{3}{2}}^\infty\ell^{+}(y_1)|y_1^\theta-1|^{k+1}y_1^{\frac{\gam+1}{p}-(k+1)\theta+k-1}\dd y_1\leq C\int_{\frac{3}{2}}^\infty y_1^{\frac{\gam+1}{p}-2}\dd y_1<\infty.
\end{equation*}
This proves \eqref{eq:est_pert_calc} and therefore this finishes the proof.
\end{proof}

We can now give the proof of the bounded $\Hinf$-calculus for the Dirichlet Laplacian using Lemmas \ref{lem:comm}-\ref{lem:comm_est2} and \ref{lem:tilde_der_f(A)}.

\begin{proof}[Proof of Theorem \ref{thm:calc:Dir}]
For $\lambda>0$ we show with induction on $k\geq 0$ that $A_\Dir:=\lambda-\Delta_{\operatorname{Dir}}$ on $W^{k,p}(w_{\gam+kp}):=W^{k,p}(\RRdh, w_{\gam+kp};X)$ has a bounded $\Hinf$-calculus of angle $\om_{\Hinf}(A_\Dir)=0$. Throughout the proof, $C$ denotes a constant only depending on $p,k,\gam,\lambda, \om, d$ and $X$, which may change from line to line.

The case $k=0$ is stated in Theorem \ref{thm:LVresult}. For some fixed $k\in\NN_0$ we assume that $A_\Dir$ on $W^{k,p}(w_{\gam+k p})$ has a bounded $\Hinf$-calculus of angle 0 for all $\gam\in (-1,2p-1)\setminus\{p-1\}$. That is, for $\om\in(0,\frac{\pi}{2})$, we have the estimate
\begin{equation}\label{eq:IH_calc}
  \|f(A_\Dir)u\|_{W^{k,p}(w_{\gam+k p})}\leq  C\|f\|_{\Hinf(\Sigma_{\om})}\|u\|_{W^{k,p}(w_{\gam+k p})},
\end{equation}
for all $\gam\in (-1,2p-1)\setminus\{p-1\}$, $f\in H^1(\Sigma_{\om})\cap \Hinf(\Sigma_{\om})$ and $u\in W^{k,p}(w_{\gam+k p})$.
We prove that $A_\Dir$ has a bounded $\Hinf(\Sigma_{\sigma})$-calculus for $\sigma\in(\om,\frac{\pi}{2})$ on $W^{k+1,p}(w_{\gam+(k+1)p})$.

Let $0<\nu<\om<\sigma<\frac{\pi}{2}$, $f\in H^1(\Sigma_{\sigma})\cap \Hinf(\Sigma_{\sigma})$ and $u\in \Cc^{\infty}(\RRdh;X)$. By writing $\alpha=(\alpha_1,\tilde{\alpha})\in \NN_0\times \NN_0^{d-1}$ we obtain
\begin{align*}
  \|f(A_\Dir)u\|_{W^{k+1,p}(w_{\gam+(k+1) p})} =&\; \|f(A_\Dir)u\|_{W^{k,p}(w_{\gam+(k+1) p})} + \sum_{\substack{|\alpha|=k+1\\\alpha_1\geq 1}}\|\d^{\alpha}f(A_\Dir)u\|_{L^p(w_{\gam+(k+1) p})}\\
  &+ \sum_{|(0,\tilde{\alpha})|=k+1}\|\d^{\tilde{\alpha}}f(A_\Dir)u\|_{L^p(w_{\gam+(k+1) p})}=:T_1+T_2+T_3.
\end{align*}
To apply the induction hypothesis we use that all the resolvents on the spaces we consider are consistent by Lemma \ref{lem:consistent_resol_RRdh}. Furthermore, we define the downwards orientated contour $\Gam_{\nu}=\d( \Sigma_{\nu}\setminus B(0, \frac{\lambda}{2}))$, see Remark \ref{rem:inv_calc}.

 By Lemma \ref{lem:LV3.13_ext} (using Remark \ref{rem:W=W_0}), Lemma \ref{lem:comm}\ref{it:comm3} and \ref{lem:comm_est1}, the induction hypothesis \eqref{eq:IH_calc} and Hardy's inequality (Corollary \ref{cor:Sob_embRRdh}), it follows that
\begin{align*}
  T_1=&\;\|M^{-1}Mf(A_\Dir)u\|_{W^{k,p}(w_{\gam+(k+1) p})}\lesssim \|Mf(A_\Dir)u\|_{W^{k,p}(w_{\gam+k p})}
  \\
  \leq &\; \|f(A_\Dir)Mu\|_{W^{k,p}(w_{\gam+k p})} + 2\:\Big\|\int_{\Gam_\nu}f(z)R(z,A_\Dir)\d_1R(z,A_\Dir)u \dd z\Big\|_{W^{k,p}(w_{\gam+k p})}\\
  \leq &\; C \|f\|_{\Hinf(\Sigma_{\om})}\big(\|Mu\|_{W^{k,p}(w_{\gam+k p})}+\|u\|_{W^{k,p}(w_{\gam+k p})}\big)\\
  \leq &\; C \|f\|_{\Hinf(\Sigma_{\om})}\|u\|_{W^{k+1,p}(w_{\gam+(k+1)p})}.
\end{align*}

We continue with $T_2$, so let $\alpha=(\alpha_1,\tilde{\alpha})$ with $|\alpha|=k+1$ and $\alpha_1\geq 1$. Using that $M\d_1\d_1^{\ell}v = \d_1^{\ell}M\d_1 v -\ell \d_1^{\ell}v$ for $\ell=\alpha_1-1$, we compute
\begin{align*}
  \|\d_1^{\alpha_1}\d^{\tilde{\alpha}}&f(A_\Dir)u\|_{L^p(w_{\gam+(k+1)p})}\\
  &= \|M\d_1 \d_1^{\alpha_1-1}\d^{\tilde{\alpha}}f(A_\Dir)u\|_{L^p(w_{\gam+kp})}\\
  &\leq \|M\d_1f(A_\Dir)u\|_{W^{k,p}(w_{\gam+k p})} +(\alpha_1-1)\|f(A_\Dir)u\|_{W^{k,p}(w_{\gam+k p})}.
\end{align*}
The second term can be estimated with the induction hypothesis \eqref{eq:IH_calc} and Hardy's inequality. For the first term, by Lemma \ref{lem:comm}\ref{it:comm2}+\ref{it:comm3}, the induction hypothesis \eqref{eq:IH_calc}, Lemma \ref{lem:comm_est2} and Hardy's inequality, we obtain
\begin{equation}\label{eq:est_T_2}
  \begin{aligned}
 \|M\d_1&f(A_\Dir)u\|_{W^{k,p}(w_{\gam+k p})}\\
 \leq& \;\|f(A_\Dir)M\d_1u\|_{W^{k,p}(w_{\gam+k p})}
  +2 \Big\|\int_{\Gam_\nu}f(z)R^2(z,A_\Dir)\d_1^2u \dd z\Big\|_{W^{k,p}( w_{\gam+k p})}\\
  \leq&\;C\|f\|_{\Hinf(\Sigma_\om)} \|M\d_1 u\|_{W^{k,p}(w_{\gam+k p})}+ C \|f\|_{\Hinf(\Sigma_{\sigma})}\|u\|_{W^{k,p}(w_{\gam+k p})}\\
  \leq&\; C\|f\|_{\Hinf(\Sigma_{\sigma})}\|u\|_{W^{k+1,p}(w_{\gam+(k+1) p})}.
\end{aligned}
\end{equation}
This shows the estimate
\begin{equation*}
  T_2\leq C\|f\|_{\Hinf(\Sigma_{\sigma})}\|u\|_{W^{k+1,p}(w_{\gam+(k+1) p})}.
\end{equation*}
From Lemma \ref{lem:tilde_der_f(A)}\ref{it:1}+\ref{it:2} it immediately follows that
\begin{equation*}
   T_3=\sum_{|(0,\tilde{\alpha})|=k+1}\|\d^{\tilde{\alpha}}f(A_\Dir)u\|_{L^p(w_{\gam+(k+1)p})}\leq C \|f\|_{\Hinf(\Sigma_{\om})}\|u\|_{W^{k+1,p}(w_{\gam+(k+1)p})}.
\end{equation*}

Combining the above estimates for $T_1$, $T_2$ and $T_3$ yields that for any $0<\om<\sigma<\frac{\pi}{2}$
\begin{equation*}
  \|f(A_\Dir)u\|_{W^{k+1,p}(w_{\gam+(k+1) p})}\leq  C\|f\|_{\Hinf(\Sigma_{\sigma})}\|u\|_{W^{k+1,p}(w_{\gam+(k+1) p})},\quad u\in \Cc^{\infty}(\RRdh;X)
\end{equation*}
and by density (Lemma \ref{lem:density}) this estimate extends to $u\in W^{k+1,p}(w_{\gam+(k+1) p})$. Since $\om$ and $\sigma$ were arbitrary, this proves that $A_\Dir$ has a bounded $\Hinf$-calculus on $W^{k+1,p}(w_{\gam+(k+1) p})$ of angle 0. This finishes the induction argument.

Finally, we show that if $\gam+kp\in(-1,2p-1)$, then we can allow for $\lambda=0$. For $k=0$ and $\gam\in(-1,2p-1)\setminus\{p-1\}$ the result is already contained in Theorem \ref{thm:LVresult}. Therefore, it remains to consider $k=1$ and $\gam\in (-1,p-1)$. From here on, let $C>0$ be a constant that is independent of $\lambda$. Note that $\gam+p\in (p-1,2p-1)$, so that by Theorem \ref{thm:LVresult} and \cite[Proposition 16.2.6]{HNVW24}
\begin{equation*}
  T_1=\|f(A_\Dir)u\|_{L^p(w_{\gam+p})} \leq C\|f\|_{\Hinf(\Sigma_{\om})}\|u\|_{L^p(w_{\gam+p})}.
\end{equation*}
Using Theorem \ref{thm:LVresult} instead of \eqref{eq:IH_calc} in \eqref{eq:est_T_2} with $k=0$ yields
  \begin{equation*}
    T_2=\|M\d_1f(A_\Dir)u\|_{L^p(w_{\gam})}\leq C \|f\|_{\Hinf(\Sigma_{\sigma})}\|u\|_{W^{1,p}( w_{\gam+p})},
  \end{equation*}
  using that the constant in Lemma \ref{lem:comm_est2} is independent of $\lambda$.
 Finally, by Lemma \ref{lem:comm}\ref{it:comm1} and Theorem \ref{thm:LVresult} we obtain
  \begin{align*}
    T_3=\sum_{j=2}^d\|f(A_\Dir)\d_ju\|_{L^p(w_{\gam+p})}& \leq C \|f\|_{\Hinf(\Sigma_{\om})}\sum_{j=2}^d\|\d_ju\|_{L^p( w_{\gam+p})}\\
    &\leq C\|f\|_{\Hinf(\Sigma_{\om})}\|u\|_{W^{1,p}( w_{\gam+p})}.
  \end{align*}
 By combining the above estimates, we have for $\gam\in (-1,p-1)$ and for any $\lambda>0$
  \begin{equation*}
    \|f(\lambda-\delDir)u\|_{W^{1,p}( w_{\gam+p})}\leq C\|f\|_{\Hinf(\Sigma_{\sigma})}\|u\|_{W^{1,p}(w_{\gam+p})},
  \end{equation*}
  where the constant $C$ is independent of $\lambda$. Letting $\lambda\downarrow 0$ and using the dominated convergence theorem yields that $-\delDir$ on $W^{1,p}(w_{\gam+p})$ has a bounded $\Hinf$-calculus of angle 0.
\end{proof}

We conclude this section with the proof of the $\Hinf$-calculus for the Neumann Laplacian, which can be derived from Theorem \ref{thm:calc:Dir} for the Dirichlet Laplacian and Lemmas \ref{lem:commNeu}-\ref{lem:tilde_der_f(A)}.
\begin{proof}[Proof of Theorem \ref{thm:calc:Neu}]
We prove that
 $A_\Neu:=\lambda-\Delta_{\Neu}$ on $$W^{k+1,p}(w_{\gam+kp}):=W^{k+1,p}(\RRdh, w_{\gam+kp};X)$$ has a bounded $\Hinf$-calculus of angle $\om_{\Hinf}(A_\Neu)=0$ for $\lambda>0$. Note that if $\gam+kp\in (-1,p-1)$, then the result is already contained in Theorem \ref{thm:LVresult_Neumann} and holds for $\lambda=0$. \\

We first consider the case $k\in \NN_0$ and $\gam\in (p-1,2p-1)$.
Let $0<\nu<\om<\frac{\pi}{2}$, $f\in H^1(\Sigma_{\om})\cap \Hinf(\Sigma_{\om})$ and $u\in \Cc^{\infty}(\RRdh;X)$. By writing $\alpha=(\alpha_1,\tilde{\alpha})\in \NN_0\times \NN_0^{d-1}$ we obtain
\begin{align*}
  \|f(A_\Neu)u\|_{W^{k+1,p}(w_{\gam+k p})} =&\; \|f(A_\Neu)u\|_{W^{k,p}(w_{\gam+k p})}\\
   &\;+ \sum_{\substack{|\alpha|=k+1\\\alpha_1\geq 1}}\|\d^{\alpha}f(A_\Neu)u\|_{L^p(w_{\gam+k p})}\\
  &\;+ \sum_{|(0,\tilde{\alpha})|=k+1}\|\d^{\tilde{\alpha}}f(A_\Neu)u\|_{L^p(w_{\gam+k p})}=:T_1+T_2+T_3.
\end{align*}
We use that all the resolvents on the spaces we consider are consistent by Lemma \ref{lem:consistent_resol_RRdh_Neumann}. Throughout the proof, $C$ denotes a constant only depending on $p,k,\gam,\lambda, \om, d$ and $X$, which may change from line to line. Moreover, we write $A_\Dir:=\lambda-\delDir$.

By Lemma \ref{lem:LV3.13_ext} (using Remark \ref{rem:W=W_0} and $\gam>p-1$), Lemma \ref{lem:comm_est1_Neumann} and Theorem \ref{thm:calc:Dir} with $\gam-p\in (-1,p-1)$, it follows that
\begin{align*}
  T_1=&\;\|M^{-1}Mf(A_\Neu)u\|_{W^{k,p}(w_{\gam+k p})}\lesssim \|Mf(A_\Neu)u\|_{W^{k,p}(w_{\gam-p+k p})}
  \\
  \leq &\; \|f(A_\Dir)Mu\|_{W^{k,p}(w_{\gam-p+k p})} + C\|f\|_{\Hinf(\Sigma_{\om})}\|u\|_{W^{k,p}(w_{\gam-p+k p})}\\
  \leq &\; C \|f\|_{\Hinf(\Sigma_{\om})}\big(\|Mu\|_{W^{k,p}(w_{\gam-p+kp})}+\|u\|_{W^{k,p}(w_{\gam-p+k p})}\big)\\
  \leq &\; C \|f\|_{\Hinf(\Sigma_{\om})}\|u\|_{W^{k+1,p}(w_{\gam+kp})},
\end{align*}
where we used Hardy's inequality (Corollary \ref{cor:Sob_embRRdh}) in the last step.
By Lemma \ref{lem:commNeu}\ref{it:commNeu4} and Theorem \ref{thm:calc:Dir}, we obtain
\begin{align*}
T_2 &=\sum_{|\alpha|=k}\|\d^{\alpha}\d_1f(A_\Neu)u\|_{L^p(w_{\gam+kp})}\leq \|f(A_{\Dir})\d_1 u\|_{W^{k,p}(w_{\gam+kp})}\\
&\leq C \|f\|_{\Hinf(\Sigma_{\om})}\|\d_1 u\|_{W^{k,p}(w_{\gam+kp})}\leq C \|f\|_{\Hinf(\Sigma_{\om})}\| u\|_{W^{k+1,p}(w_{\gam+kp})}.
\end{align*}
Finally, using Lemma \ref{lem:tilde_der_f(A)}\ref{it:3} with $k+1$ and $\gam-p\in (-1,p-1)$, gives
\begin{equation*}
  T_3=\sum_{|(0,\tilde{\alpha})|=k+1}\|\d^{\tilde{\alpha}}f(A_\Neu)u\|_{L^p(w_{\gam-p+(k+1) p})}\leq C \|f\|_{\Hinf(\Sigma_{\om})}\| u\|_{W^{k+1,p}(w_{\gam+kp})}.
\end{equation*}
The above estimates and density (Lemma \ref{lem:density}) prove boundedness of the $\Hinf$-calculus for $A_{\Neu}$ on $W^{k+1,p}(w_{\gam+kp})$ with $k\in\NN_1$ and $\gam\in(p-1,2p-1)$.\\

It remains to consider the case $k\in\NN_1$ and $\gam\in (-1,p-1)$.
Let $\om\in(0,\frac{\pi}{2})$, $f\in H^1(\Sigma_{\om})\cap \Hinf(\Sigma_{\om})$ and $u\in C^{\infty}_{{\rm c},1}(\overline{\RRdh};X)$. Then writing $\alpha=(0,\tilde{\alpha})\in\NN_0\times\NN_0^{d-1}$, gives
\begin{align*}
\|f(A_\Neu)u \|_{W^{k+1,p}(w_{\gam+kp})}&=\sum_{|(0,\tilde{\alpha})|\leq 1} \|\d^{\alpha} f(A_{\Neu})u\|_{W^{k,p}(w_{\gam+kp})} +\|\d_1f(A_\Neu)u\|_{W^{k,p}(w_{\gam+kp})}.
\end{align*}
We estimate these terms separately using that the resolvents on the spaces we consider are consistent by Lemma \ref{lem:consistent_resol_RRdh_Neumann}. Moreover, note that $\gam+p\in (p-1,2p-1)$, so that $W^{k,p}(w_{\gam+kp})=W^{(k-1)+1}(w_{\gam+p+(k-1)p})$ and we can apply the result proved above with $k-1$. Together with Lemma \ref{lem:commNeu}\ref{it:commNeu1}, we obtain
\begin{align*}
\sum_{|(0,\tilde{\alpha})|\leq 1} \|\d^{\alpha} f(A_{\Neu})u\|_{W^{k,p}(w_{\gam+kp})}
    &= \sum_{|(0,\tilde{\alpha})|\leq 1} \| f(A_{\Neu}) \d^{\alpha} u\|_{W^{(k-1)+1,p}(w_{\gam+p+(k-1)p})} \\
   &\leq C \|f\|_{\Hinf(\Sigma_{\om})}\sum_{|(0,\tilde{\alpha})|\leq 1} \|\d^{\alpha} u\|_{W^{(k-1)+1,p}(w_{\gam+p+(k-1)p})} \\
   &\leq C \|f\|_{\Hinf(\Sigma_{\om})} \|u\|_{W^{k+1,p}(w_{\gam+kp})}.
\end{align*}
Moreover, using Lemma \ref{lem:commNeu}\ref{it:commNeu4} and Theorem \ref{thm:calc:Dir} for $\gam\in(-1,p-1)$, yield
\begin{align*}
  \|\d_1f(A_\Neu)u\|_{W^{k,p}(w_{\gam+kp})}&= \|f(A_{\Dir})\d_1 u\|_{W^{k,p}(w_{\gam+kp})}\\
  &\leq C\|f\|_{\Hinf(\Sigma_{\om})} \|\d_1u\|_{W^{k,p}(w_{\gam+kp})}\leq C\|f\|_{\Hinf(\Sigma_{\om})} \|u\|_{W^{k+1,p}(w_{\gam+kp})}.
\end{align*}
The above estimates and density (see Lemma \ref{lem:density}) prove boundedness of the $\Hinf$-calculus for $A_\Neu$ in the case $\gam\in (-1,p-1)$. This completes the proof.
\end{proof}

\section{Elliptic and parabolic maximal regularity}\label{sec:MR}

In this final section, we study the elliptic maximal regularity of the Laplace equation and parabolic maximal regularity of the heat equation on homogeneous and inhomogeneous Sobolev spaces.
As discussed in the introduction, there is extensive literature on maximal regularity for elliptic and parabolic problems on homogeneous weighted Sobolev spaces, which is obtained via completely different techniques. We refer to \cite{Kr99b,Kr01} for Dirichlet boundary conditions and \cite{DK18, DKZ16} for Neumann boundary conditions.

Using our regularity results on inhomogeneous spaces and a scaling argument, we derive maximal regularity on homogeneous weighted Sobolev spaces as well. Moreover, in certain cases our regularity results correspond to results earlier obtained in \cite{DK18, DKZ16, Kr99b,Kr01}.\\

For $p\in(1,\infty)$, $k\in\NN_0$, $\gam\in \RR$ and $X$ a Banach space, we define the homogeneous space
\begin{equation*}
  \dot{W}^{k,p}(\RRdh, w_{\gam+kp};X):=\Big\{f\in \mc{D}'(\RRdh;X): \forall |\alpha|\leq k, \d^{\alpha}f\in L^p(\RRdh,w_{\gam+|\alpha|p};X)\Big\}
\end{equation*}
equipped with the canonical norm.
In comparison to the existing literature, we have for $\gam>-1$ by Hardy's inequality (Corollary \ref{cor:Sob_embRRdh})
\begin{equation*}
  W^{k,p}(\RRdh,w_{\gam+kp};X)\hookrightarrow \dot{W}^{k,p}(\RRdh, w_{\gam+kp};X)=H^{k}_{p,\gam+d}(\RRdh;X),
\end{equation*}
where for $\theta\in \RR$ the homogeneous spaces
\begin{equation*}
  H^{k}_{p,\theta}(\RRdh;X)=\Big\{f\in \mc{D}'(\RRdh;X): \forall |\alpha|\leq k, \d^{\alpha}f\in L^p(\RRdh,w_{\theta-d+|\alpha|p};X)\Big\},
\end{equation*}
are as used in for instance \cite{Kr99b,Kr01} with $X=\CC$. For these homogeneous spaces, we recall the following density result.
\begin{lemma}\label{lem:density_hom}
  Let $p\in(1,\infty)$, $k\in\NN_0$, $\gam\in \RR$ and let $X$ be a Banach space. Then $\Cc^{\infty}(\RRdh;X)$ is dense in $\dot{W}^{k,p}(\RRdh, w_{\gam+kp};X)$.
\end{lemma}
\begin{proof}
  This follows from \cite[Theorem 1.19 \& Corollary 2.3]{Kr99b}, where it should be noted that these results also hold in the vector-valued case.
\end{proof}

For $\gam\in (-1,2p-1)\setminus\{p-1\}$, we define the space with Dirichlet boundary conditions
\begin{equation*}
  \begin{aligned}
  \dot{\WW}^{k+2,p}_{\Dir}&(\RRdh, w_{\gam+kp};X)\\
  &:=\Big\{f\in \mc{D}'(\RRdh;X): \forall |\beta|\leq 2, \d^{\beta}f\in \dot{W}^{k,p}(\RRdh,w_{\gam+kp};X), \Tr(f)=0\Big\}
\end{aligned}
\end{equation*}
equipped with the canonical norm. Similarly, for $\gam\in(-1,p-1)$, we define the space with Neumann boundary conditions
\begin{equation*}
  \begin{aligned}
  \dot{\WW}^{k+2,p}_{\Neu}&(\RRdh, w_{\gam+kp};X)\\
  &:=\Big\{f\in \mc{D}'(\RRdh;X): \forall |\beta|\leq 2, \d^{\beta}f\in \dot{W}^{k,p}(\RRdh,w_{\gam+kp};X), \Tr(\d_1f)=0\Big\}
\end{aligned}
\end{equation*}
equipped with the canonical norm.
Note that for the given ranges of $\gam$ we have that
\begin{equation}\label{eq:dom_embed}
  \dot{\WW}^{k+2,p}_{\Dir}(\RRdh, w_{\gam+kp};X), \dot{\WW}^{k+2,p}_{\Neu}(\RRdh, w_{\gam+kp};X)\hookrightarrow W^{2,p}(\RRdh,w_{\gam};X),
\end{equation}
so that the Dirichlet and Neumann traces are well defined. For Neumann boundary conditions, we will not consider the homogeneous counterpart of $W^{k+1,p}(\RRdh, w_{\gam+kp};X)$ with $\gam\in(-1,p-1)$ to avoid weights $w_{\gam}$ with $\gam<-1$.

\begin{definition}\label{def:delhom}
  Let $p\in(1,\infty)$, $k\in\NN_0$ and let $X$ be a Banach space. The Dirichlet and Neumann Laplacian on homogeneous spaces are defined as follows.
  \begin{enumerate}[(i)]
    \item For $\gam\in(-1,2p-1)\setminus\{p-1\}$ the \emph{Dirichlet Laplacian $\delDir$ on $\dot{W}^{k,p}(\RRdh,w_{\gam+kp};X)$} is defined by
          \begin{equation*}
    \delDir u := \sum_{j=1}^d \d_j^2 u\quad \text{ with }\quad D(\delDir):=\dot{\WW}^{k+2,p}_{\Dir}(\RRdh, w_{\gam+kp};X).
  \end{equation*}
    \item For $\gam\in(-1,p-1)$ the \emph{Neumann Laplacian $\delNeu$ on $\dot{W}^{k,p}(\RRdh,w_{\gam+kp};X)$} is defined by
          \begin{equation*}
    \delNeu u := \sum_{j=1}^d \d_j^2 u\quad \text{ with }\quad D(\delNeu):=\dot{\WW}^{k+2,p}_{\Neu}(\RRdh, w_{\gam+kp};X).
  \end{equation*}
  \end{enumerate}
\end{definition}

\subsection{Elliptic maximal regularity on homogeneous Sobolev spaces}
We recall that elliptic regularity on the inhomogeneous spaces $W^{k,p}(\RRdh,w_{\gam+kp};X)$ has already been studied in Sections \ref{subsec:sectDir}, \ref{subsec:sectNeu} and \ref{subsec:growth}.
By Proposition \ref{prop:sect_est} and a scaling argument, we obtain the following regularity result for the Dirichlet Laplacian.
\begin{corollary}[Homogeneous elliptic regularity for $-\delDir$]\label{cor:hom_ell_reg_Dir} Let $p\in(1,\infty)$, $k\in\NN_0$, $\gam\in (-1,2p-1)\setminus\{p-1\}$, $\om\in (0,\pi)$ and  let $X$ be a $\UMD$ Banach space. Let $\delDir$ on $\dot{W}^{k,p}(\RRdh,  w_{\gam+kp}    ;X)$ be as in Definition \ref{def:delhom}. Then for all $f\in \dot{W}^{k,p}( \RRdh, w_{\gam+kp};X)$ and $\lambda\in \Sigma_{\pi-\om}$, there exists a unique solution $$u\in \dot{\WW}^{k+2,p}_{\Dir}(\RRdh, w_{\gam+kp};X)$$ such that $\lambda u -\delDir u =f$. Moreover, this solution satisfies
\begin{equation*}
  \sum_{|\beta|\leq 2}|\lambda|^{1-\frac{|\beta|}{2}}\|\d^{\beta}u \|_{\dot{W}^{k,p}(\RRdh,w_{\gam+kp};X)}\leq C \|f\|_{\dot{W}^{k,p}(\RRdh,w_{\gam+kp};X)},
\end{equation*}
  where the constant $C$ only depends on $p,k,\gam,\om,d$ and $X$.
\end{corollary}
 \begin{proof}
 Note that the uniqueness follows from \eqref{eq:dom_embed} and Corollary \ref{cor:LVresult_MR}. For existence and the estimate we use a scaling argument. Let $f\in\Cc^{\infty}(\RRdh;X)$, then Lemma \ref{lem:Schwartz} yields a smooth solution $u$ to the equation. Let $r>0$ and set $u_r(x):=u(rx)$ and $f_r(x):=f(rx)$. Then $u_r$ satisfies $u_r(0,\cdot)=0$ and the equation $r^2\lambda u_r-\delDir u_r = r^2 f_r$, so by Proposition \ref{prop:sect_est} we have for all $\ell\in\{0,\dots, k\}$ the estimate
   \begin{equation*}
     \sum_{|\beta|\leq 2}|r^2\lambda|^{1-\frac{|\beta|}{2}}\|\d^{\beta}u_r\|_{W^{\ell,p}(\RRdh,w_{\gam+\ell p};X)}\leq C g_{\ell,\gam}(r^2\lambda)\|r^2f_r\|_{W^{\ell,p}(\RRdh,w_{\gam+\ell p};X)},
   \end{equation*}
   where $g_{\ell,\gam}$ is defined in \eqref{eq:K_IH}. After the substitution $x\mapsto r^{-1}x$ we obtain
   \begin{align*}
     \sum_{|\beta|\leq 2}\sum_{|\alpha|\leq \ell} & |\lambda|^{1-\frac{|\beta|}{2}}r^{|\alpha|-\ell}\|\d^{\alpha+\beta}u\|_{L^p(\RRdh,w_{\gam+\ell p};X)}\\
     &\leq C g_{\ell,\gam}(r^2\lambda)\sum_{|\alpha|\leq \ell}r^{|\alpha|-\ell}\|\d^{\alpha}f\|_{L^p(\RRdh,w_{\gam+\ell p};X)}.
   \end{align*}
   Letting $r\to\infty$ and summing over $\ell\in \{0,\dots, k\}$ yields the estimate for $f\in\Cc^{\infty}(\RRdh;X)$. Here it should be noted that
   \begin{equation*}
     \lim_{r\to\infty} g_{\ell,\gam}(r^2\lambda)= \lim_{y\to\infty}g_{\ell,\gam}(y)=C,
   \end{equation*}
   where $C=1$ or $C=2$ depending on $\ell$ and $\gam$.
  By Lemma \ref{lem:density_hom} and a similar density argument as in the proof of Proposition \ref{prop:sect_est}, we obtain that for every $f\in \dot{W}^{k,p}(\RRdh,w_{\gam+kp};X)$ there exists an $u\in \dot{\WW}^{k+2,p}_{\Dir}(\RRdh, w_{\gam+kp};X)$ solving the equation and the required estimate holds.
 \end{proof}

In a similar fashion, we obtain the following elliptic regularity result for the Neumann Laplacian.
\begin{corollary}[Homogeneous elliptic regularity for $-\delNeu$]\label{cor:hom_ell_reg_Neu} Let $p\in(1,\infty)$, $k\in\NN_0$, $\gam\in (-1,p-1)$, $\om\in (0,\pi)$ and let $X$ be a $\UMD$ Banach space.  Let $\delNeu$ on $\dot{W}^{k,p}(\RRdh, w_{\gam+kp};X)$ be as in Definition \ref{def:delhom}. Then for all $f\in \dot{W}^{k,p}(\RRdh, w_{\gam+kp})$ and $\lambda\in \Sigma_{\pi-\om}$, there exists a unique solution $$u\in \dot{\WW}^{k+2,p}_\Neu( \RRdh, w_{\gam+kp} ; X)$$ to $\lambda u -\delNeu u =f$. Moreover, this solution satisfies
\begin{equation*}
  \sum_{|\beta|\leq 2}|\lambda|^{1-\frac{|\beta|}{2}}\|\d^{\beta}u \|_{\dot{W}^{k,p}(\RRdh,w_{\gam+kp};X)}\leq C \|f\|_{\dot{W}^{k,p}(\RRdh,w_{\gam+kp};X)},
\end{equation*}
  where the constant $C$ only depends on $p,k,\gam,\om,d$ and $X$.
\end{corollary}
\begin{proof}
  The proof is analogous to the proof of Corollary \ref{cor:hom_ell_reg_Dir} using Proposition \ref{prop:sect_est_neu}.
\end{proof}
\begin{remark} \hspace{2em}
\begin{enumerate}[(i)]
\item \emph{Second order Riesz transforms}: for $\lambda=0$ we obtain the regularity estimates from Corollaries  \ref{cor:hom_ell_reg_Dir} and \ref{cor:hom_ell_reg_Neu} as well. Indeed, if $u\in \dot{\WW}^{k+2,p}_\Dir(\RRdh,w_{\gam+kp};X)$ is a solution to $-\delDir u =f$, then $u$ also satisfies the equation $\lambda u -\delDir u =\lambda u +f$ with $\lambda>0$. Therefore, by Corollary \ref{cor:hom_ell_reg_Dir} we obtain
  \begin{align*}
  \sum_{|\beta|= 2}\|\d^{\beta}u \|_{\dot{W}^{k,p}(\RRdh,w_{\gam+kp};X)}&\leq C \lim_{\lambda\downarrow 0} \|\lambda u + f\|_{\dot{W}^{k,p}(\RRdh,w_{\gam+kp};X)}\\ &= C \|f\|_{\dot{W}^{k,p}(\RRdh,w_{\gam+kp};X)}.
  \end{align*}
The same proof works for the Neumann Laplacian.

\item The results from Corollaries \ref{cor:hom_ell_reg_Dir} and \ref{cor:hom_ell_reg_Neu} with $k=0$ and $\gam\in (-1,p-1)$ coincide with the inhomogeneous results in Corollaries \ref{cor:LVresult_MR} and \ref{cor:LVresult_MR_neumann}. Moreover, these cases have already been obtained in \cite{DK18,DKZ16} for the scalar-valued case.
\item The regularity result in Corollary \ref{cor:hom_ell_reg_Dir} with $k\in\NN_0$, $\gam\in(p-1,2p-1)$ and $X=\CC$ is already covered in \cite[Theorem 4.1]{Kr99b}.
\end{enumerate}
\end{remark}

\subsection{Maximal \texorpdfstring{$L^q$}{Lq}-regularity}\label{sec:MRRRdh} We now turn to maximal regularity for the heat equation $$\d_t u -\del u =f,$$ with Dirichlet or Neumann boundary conditions and zero initial condition. We will view this as an abstract Cauchy problem. For an introduction to maximal regularity in this setting, we refer to \cite[Chapter 17]{HNVW24}.\\

For any $T\in (0,\infty]$, we define the time interval $I:=(0,T)$. Let $A$ be a linear operator with domain $D(A)$ on a Banach space $Y$ and for some $f\in L^1_\loc(\overline{I};Y)$ consider the abstract Cauchy problem
\begin{equation}\label{eq:Cauchy}
  \begin{aligned}
  \d_t u(t)+A u(t)&=f(t),\qquad  t\in I, \\
    u(0)&=0.
  \end{aligned}
\end{equation}
We call a strongly measurable function $u:I\to Y$ a \emph{strong solution to \eqref{eq:Cauchy}} if $u$ takes values in $D(A)$ almost everywhere, $Au\in L^1_\loc(\overline{I};Y)$ and $u$ solves
\begin{equation}\label{eq:strongsol}
  u(t)+\int_0^t Au(s)\dd s=\int_0^t f(s)\dd s, \quad \text{ for almost all }t\in I.
\end{equation}
Moreover, for $q\in (1,\infty)$, $v\in A_q(I)$ and $f\in L^q(I,v;Y)$ a strong solution $u$ to \eqref{eq:Cauchy} is called an \emph{$L^q(v)$-solution} if $Au \in L^q(I,v;Y)$.

\begin{definition}[Maximal $L^q(v)$-regularity]
  A linear operator $A$ on a Banach space $Y$ has \emph{maximal $L^q(v)$-regularity on $I$} if for all $f\in L^q(I,v;Y)$ the Cauchy problem \eqref{eq:Cauchy} has a unique $L^q(v)$-solution $u$ on $I$ and
  \begin{equation*}
    \|Au\|_{L^q(I,v;Y)}\leq C\|f\|_{L^q(I,v;Y)},
  \end{equation*}
  where the constant $C$ is independent of $f$.
\end{definition}

The following two corollaries on maximal regularity for the Dirichlet and Neumann Laplacian follow immediately from  Theorems \ref{thm:var_setting}, \ref{thm:calc:Dir} and \ref{thm:calc:Neu} together with \cite[Theorems 17.3.18 \& 17.2.39 \& Proposition 17.2.7]{HNVW24}.

\begin{corollary}[Inhomogeneous maximal regularity for $-\delDir$]\label{cor:MR_RRdh}
  Let $p,q\in(1,\infty)$, $k\in\NN_0\cup\{-1\}$, $\gam\in(-1,2p-1)\setminus\{p-1\}$ and let $X$ be a $\UMD$ Banach space.  Let $\delDir$ on $W^{k,p}(\RRdh, w_{\gam+kp};X)$ be as in Definition \ref{def:delRRdh}. Assume that either
  \begin{enumerate}[(i)]
    \item $\gam+kp\in (-1,2p-1)$, $T\in (0,\infty]$ and $v\in A_q(I)$, or,
    \item $\gam+kp>2p-1$, $T\in(0,\infty)$ and $v\in A_q(I)$.
  \end{enumerate}
  Then $-\delDir$ has maximal $L^q(v)$-regularity on $I$.  In particular, for every $T\in (0,\infty)$ and $f\in L^q(I,v;W^{k,p}(\RRdh,w_{\gam+kp};X))$ there exists a unique $L^q(v)$-solution $u$ to \eqref{eq:Cauchy} with $A=-\delDir$, which satisfies
        \begin{equation}\label{eq:est_MR}
            \begin{aligned}
    \|u&\|_{W^{1,q}(I,v; W^{k,p}(\RRdh, w_{\gam+kp};X))} + \| u\|_{L^q(I,v; W^{k+2,p}_{\Dir}(\RRdh, w_{\gam+kp};X))}\\
    &\;\leq C \|f\|_{L^q(I,v; W^{k,p}(\RRdh, w_{\gam+kp};X))},
    \end{aligned}
        \end{equation}
        where the constant $C$ only depends on $p,q,k,\gam,v,T,d$ and $X$.
\end{corollary}

\begin{corollary}[Inhomogeneous maximal regularity for $-\delNeu$]\label{cor:MR_RRdh_Neu}
  Let $p,q\in(1,\infty)$, $k\in\NN_0\cup\{-1\}$, $\gam\in(-1,2p-1)\setminus\{p-1\}$ and let $X$ be a $\UMD$ Banach space.  Let $\delNeu$ on $W^{k+1,p}(\RRdh, w_{\gam+kp};X)$ be as in Definition \ref{def:delRRdh}.
  Assume that either
  \begin{enumerate}[(i)]
    \item $\gam+kp\in (-1,p-1)$, $T\in (0,\infty]$ and $v\in A_q(I)$, or,
    \item $\gam+kp>p-1$, $T\in(0,\infty)$ and $v\in A_q(I)$.
  \end{enumerate}
  Then $-\delNeu$  has maximal $L^q(v)$-regularity on $I$.  In particular, for every $T\in (0,\infty)$ and $f\in L^q(I,v;W^{k+1,p}(\RRdh,w_{\gam+kp};X))$ there exists a unique $L^q(v)$-solution $u$ to \eqref{eq:Cauchy} with $A=-\delNeu$, which satisfies
        \begin{equation*}
            \begin{aligned}
    \|u&\|_{W^{1,q}(I,v; W^{k+1,p}(\RRdh, w_{\gam+kp};X))} + \| u\|_{L^q(I,v; W^{k+3,p}_{\Neu}(\RRdh, w_{\gam+kp};X))}\\
    &\;\leq C \|f\|_{L^q(I,v; W^{k+1,p}(\RRdh, w_{\gam+kp};X))},
    \end{aligned}
        \end{equation*}
        where the constant $C$ only depends on $p,q,k,\gam,v,T,d$ and $X$.
  \end{corollary}

\begin{remark}\label{rem:MR} \hspace{2em}
    \begin{enumerate}[(i)]
      \item Corollaries \ref{cor:MR_RRdh} and \ref{cor:MR_RRdh_Neu} do not hold for $T=\infty$ in general. This follows from Dore's theorem (see \cite{Do00} or \cite[Theorem 17.2.15]{HNVW24}) and Theorems \ref{thm:Dirsemi}\ref{it:Dirsemi3} and \ref{thm:Neusemi}\ref{it:Neusemi3}, respectively. On the other hand, the shifted operators $\lambda-\delDir$ and $\lambda-\delNeu$ with $\lambda>0$ do have maximal $L^q(v)$-regularity on $\RR_+$, see \cite[Propositions 17.2.27 \& 17.2.29]{HNVW24}.
      \item Corollaries \ref{cor:MR_RRdh} and \ref{cor:MR_RRdh_Neu} only concern the Cauchy problem \eqref{eq:Cauchy} with zero initial data. In addition, maximal regularity can be used to obtain existence and uniqueness for the Cauchy problem with non-zero initial data, see \cite[Section 4.4]{GV17} and \cite[Section 17.2.b]{HNVW24}. In particular, if $v=t^{\eta}$ is a power weight with $\eta\in(-1,q-1)$, then the space for the initial data is the real interpolation space $$\big(W^{k,p}(\RRdh, w_{\gam+kp};X), D(A)\big)_{1-\frac{1+\eta}{p},p}$$ with $A=-\delDir$ or $A=-\delNeu$, see \cite[Example 4.16]{GV17}. Characterisations of real interpolation spaces of weighted Sobolev spaces can be found in \cite[Chapter 3]{Tr78}.
    \end{enumerate}
  \end{remark}

In the following corollaries, we derive maximal regularity on homogeneous spaces.
\begin{corollary}[Homogeneous maximal regularity for $-\delDir$]\label{cor:MR_RRdhKrylov}
  Let $p,q\in(1,\infty)$, $k\in\NN_0$, $\gam\in(-1,2p-1)\setminus\{p-1\}$, $v\in A_q(\RR_+)$ and let $X$ be a $\UMD$ Banach space.  Let $\delDir$ on $\dot{W}^{k,p}(\RRdh,w_{\gam+kp};X)$ be as in Definition \ref{def:delhom}. Then $-\delDir$  has maximal $L^q(v)$-regularity on $\RR_+$.  In particular, for every $f\in L^q(\RR_+,v;\dot{W}^{k,p}(\RRdh,w_{\gam+kp};X))$ there exists a unique $L^q(v)$-solution $u$ to \eqref{eq:Cauchy} with $A=-\delDir$, which satisfies
        \begin{equation}\label{eq:est_hom2}
            \begin{aligned}
    \|\d_t u&\|_{L^q(\RR_+,v; \dot{W}^{k,p}(\RRdh, w_{\gam+kp};X))} + \sum_{|\beta|=2}\|\d^{\beta} u\|_{L^q(\RR_+,v; \dot{W}^{k,p}(\RRdh, w_{\gam+kp};X))}\\
    &\;\leq C \|f\|_{L^q(\RR_+,v; \dot{W}^{k,p}(\RRdh, w_{\gam+kp};X))},
    \end{aligned}
        \end{equation}
        where the constant $C$ only depends on $p,q,k,\gam,v,d$ and $X$.
\end{corollary}
\begin{proof}
  Let $f\in \Cc^{\infty}(\RR_+\times\RRdh;X)$, then $f\in L^q((0,T),v;W^{k,p}(\RRdh, w_{\gam+kp};X))$ for all $T>0$. Corollary \ref{cor:MR_RRdh} yields that for any $T>0$ there exists a unique $L^q(v)$-solution $u_T$ to \eqref{eq:Cauchy} on $(0,T)$ satisfying the estimate \eqref{eq:est_MR}.
  If $T_1\leq T$, then the restriction of $u_T$ to $(0,T_1)$ coincides with $u_{T_1}$. This allows us to construct a solution $u:\RR_+\to W^{k+2,p}_\Dir(\RRdh,w_{\gam+kp};X)\hookrightarrow \dot{\WW}^{k+2,p}_\Dir(\RRdh,w_{\gam+kp};X)$ that coincides with $u_T$ on $(0,T)$ for every $T>0$. Moreover, $u$ is a strong solution to \eqref{eq:Cauchy} on $\RR_+$ (with $-\delDir$ considered as an operator on the homogeneous space $\dot{W}^{k,p}(\RRdh,w_{\gam+kp};X)$).

  To show the regularity of $u$, let $r>0$ and set $u_r(t,x):=u(r^2t, rx)$ and $f_r(t,x):=f(r^2t, rx)$. Then $u_r$ satisfies the equation $\d_tu_r-\delDir u_r=r^2f_r$, so by \eqref{eq:est_MR} we have for all $\ell\in\{0,\dots, k\}$ the estimate
\begin{align*}
  \|\d_t u_r\|&_{L^q((0,1), v(r^2\cdot);W^{\ell,p}(\RRdh,w_{\gam+\ell p};X))}+\|u_r\|_{L^q((0,1), v(r^2\cdot);W^{\ell+2,p}(\RRdh,w_{\gam+\ell p};X))}\\
  &\leq C \|r^2 f_r\|_{L^q((0,1), v(r^2\cdot);W^{\ell,p}(\RRdh,w_{\gam+\ell p};X))},
\end{align*}
since $v(r^2\cdot)\in A_q(\RR_+)$ with the same $A_q$ constant by \cite[Proposition 7.1.5(1)]{Gr14_classical_3rd}. After the substitutions $x\mapsto r^{-1}x$ and $t\mapsto r^{-2}t$, we obtain
\begin{align*}
  \sum_{|\alpha|\leq \ell}r^{|\alpha|-\ell} &\|\d_t\d^{\alpha}  u \|_{L^q((0,r^2),v;L^p(\RRdh, w_{\gam+\ell p};X))}
 +\sum_{|\alpha|\leq \ell+2}r^{|\alpha|-\ell-2}\|\d^{\alpha}u\|_{L^q((0,r^2),v;L^p(\RRdh, w_{\gam+\ell p};X))}\\
   \leq&\;C\sum_{|\alpha|\leq \ell}r^{|\alpha|-\ell}
    \|\d^{\alpha}f\|_{L^q((0,r^2),v;L^p(\RRdh, w_{\gam+\ell p};X))}.
\end{align*}
Letting $r\to \infty$ and summing over $\ell\in\{0,\dots, k\}$ yields the maximal regularity estimate \eqref{eq:est_hom2}.
Thus $u$ is an $L^q(v)$-solution and we claim that it is also unique on $\RR_+$. Indeed, if $f=0$, then for every $T>0$ we have
\begin{align*}
  \|u\|_{C([0,T]; \dot{W}^{k,p}(\RRdh,w_{\gam+kp};X))}&\leq \|Au \|_{L^1([0,T]; \dot{W}^{k,p}(\RRdh,w_{\gam+kp};X))}\\
  &\leq \|A u \|_{L^q([0,T],v; \dot{W}^{k,p}(\RRdh,w_{\gam+kp};X))}\Big(\int_0^T|v^{-\frac{1}{q-1}}(t)|\dd t\Big)^{\frac{q-1}{q}}\\
  &\leq 0,
\end{align*}
which follows from taking the supremum over $[0,T]$ in \eqref{eq:strongsol}, H\"older's inequality and \eqref{eq:est_hom2}. This yields that $u=0$ on $[0,T]$ for all $T>0$. Therefore $u=0$ on $\RR_+$ as well and this proves the uniqueness.

Finally, from \cite[Lemma 3.5]{LV18} and Lemma \ref{lem:density_hom}, it follows that $\Cc^{\infty}(\RR_+\times \RRdh;X)$ is dense in $L^q(\RR_+, v;\dot{W}^{k,p}(\RRdh, w_{\gam+kp};X))$. Therefore, a density argument (see \cite[Proposition 17.2.10]{HNVW24} using that $-\delDir$ is closed by Corollary \ref{cor:hom_ell_reg_Dir}) gives the result.
\end{proof}

Similarly, we have maximal regularity for the Neumann Laplacian on homogeneous spaces.
\begin{corollary}[Homogeneous maximal regularity for $-\delNeu$]\label{cor:MR_RRdhKrylov_Neu}
Let $p,q\in(1,\infty)$, $k\in\NN_0$, $\gam\in(-1,p-1)$, $v\in A_q(\RR_+)$ and let $X$ be a $\UMD$ Banach space. Let $\delNeu$ on $\dot{W}^{k,p}(\RRdh,w_{\gam+kp};X)$ be as in Definition \ref{def:delhom}. Then $-\delNeu$ has maximal $L^q(v)$-regularity on $\RR_+$. In particular, for every $f\in L^q(\RR_+,v;\dot{W}^{k,p}(\RRdh,w_{\gam+kp};X))$ there exists a unique $L^q(v)$-solution $u$ to \eqref{eq:Cauchy} with $A=-\delNeu$, which satisfies
        \begin{equation*}
            \begin{aligned}
    \|\d_t u&\|_{L^q(\RR_+,v; \dot{W}^{k,p}(\RRdh, w_{\gam+kp};X))} + \sum_{|\beta|=2}\|\d^{\beta} u\|_{L^q(\RR_+,v; \dot{W}^{k,p}(\RRdh, w_{\gam+kp};X))}\\
    &\;\leq C \|f\|_{L^q(\RR_+,v; \dot{W}^{k,p}(\RRdh, w_{\gam+kp};X))},
    \end{aligned}
        \end{equation*}
        where the constant $C$ only depends on $p,q,k,\gam,v,d$ and $X$.
\end{corollary}
\begin{proof}
  The proof is analogous to the proof of Corollary \ref{cor:MR_RRdhKrylov}, using Corollary \ref{cor:MR_RRdh_Neu}.
\end{proof}
We conclude with some remarks concerning Corollaries \ref{cor:MR_RRdhKrylov} and \ref{cor:MR_RRdhKrylov_Neu}.
\begin{remark}\label{rem:Dir_MR} \hspace{2em}
  \begin{enumerate}[(i)]
    \item The maximal regularity estimate from Corollary \ref{cor:MR_RRdhKrylov} for $\gam\in(p-1,2p-1)$ and $v=1$ is already obtained in \cite[Theorem 3.2]{Kr01}.
\item For $k=0$ and $\gam\in(-1,p-1)$ similar regularity results as those in Corollaries \ref{cor:MR_RRdhKrylov} and \ref{cor:MR_RRdhKrylov_Neu} are contained in \cite[Theorem 6.4]{DK18} and \cite[Theorem 2.3]{DKZ16} for the shifted operators $\lambda-\delDir$ and $\lambda-\delNeu$ with $\lambda>0$. The time weights for $k\geq 1$ appear to be new in the homogeneous setting. These time weights play an important role in nonlinear evolution equations, see, e.g., \cite{DK18}, \cite[Chapter 18]{HNVW24} and \cite{PS16}.
    \item More general elliptic operators with time and space-dependent coefficients are treated in for example \cite[Theorem 2.1]{DK15}, \cite[Theorem 6.4]{DK18} or \cite[Theorem 3.2]{Kr01}. We only proved results for the Laplacian by taking limits in our maximal regularity estimates for inhomogeneous spaces.
  \end{enumerate}
\end{remark}

\bibliographystyle{plain} 
\bibliography{referencesLLRV_v2}

\end{document}